\documentclass[12pt]{amsart}
\usepackage{amsmath}
\usepackage{amssymb}
\usepackage{amsthm}
\usepackage{amscd}
\usepackage{enumerate}
\usepackage{enumitem}
\usepackage{verbatim}
\usepackage[all]{xy}
\usepackage{mathrsfs}
\usepackage{mathabx}
\usepackage{url}

\theoremstyle{plain}
\newtheorem{thm}{Theorem}[section]
\newtheorem{prop}[thm]{Proposition}
\newtheorem{lem}[thm]{Lemma}
\newtheorem{cor}[thm]{Corollary}

\theoremstyle{definition}
\newtheorem{dfn}[thm]{Definition}
\newtheorem{rmk}[thm]{Remark}

\newcommand{\alg}{\mathrm{alg}}
\newcommand{\Ann}{\mathrm{Ann}}

\newcommand{\Aut}{\mathrm{Aut}}
\newcommand{\can}{\mathrm{can}}
\newcommand{\Car}{\mathrm{Car}}
\newcommand{\Coker}{\mathrm{Coker}}
\newcommand{\Cot}{\mathrm{Cot}}
\newcommand{\Cusps}{\mathrm{Cusps}}

\newcommand{\Der}{\mathrm{Der}}
\newcommand{\DM}{\mathrm{DM}}
\newcommand{\DR}{\mathrm{DR}}

\newcommand{\End}{\mathrm{End}}
\newcommand{\et}{\mathrm{et}}
\newcommand{\ev}{\mathrm{ev}}
\newcommand{\Ext}{\mathrm{Ext}}

\newcommand{\Frac}{\mathrm{Frac}}
\newcommand{\Gal}{\mathrm{Gal}}
\newcommand{\Fix}{\mathrm{Fix}}

\newcommand{\Hom}{\mathrm{Hom}}

\newcommand{\HTT}{\mathrm{HTT}}

\newcommand{\id}{\mathrm{id}}
\newcommand{\Img}{\mathrm{Im}}
\newcommand{\Isom}{\mathrm{Isom}}
\newcommand{\KS}{\mathrm{KS}}
\newcommand{\Ker}{\mathrm{Ker}}
\newcommand{\Lie}{\mathrm{Lie}}

\newcommand{\ord}{\mathrm{ord}}

\newcommand{\rank}{\mathrm{rank}}
\newcommand{\red}{\mathrm{red}}
\newcommand{\Rep}{\mathrm{Rep}}

\newcommand{\sep}{\mathrm{sep}}
\newcommand{\Sh}{\mathrm{Sh}}

\newcommand{\Shv}{\mathrm{Shv}}

\newcommand{\Spec}{\mathrm{Spec}}

\newcommand{\Sym}{\mathrm{Sym}}

\newcommand{\TD}{\mathrm{TD}}
\newcommand{\TDt}{\mathrm{TD}^{\triangledown}}

\newcommand{\trd}{\triangledown}
\newcommand{\univ}{\mathrm{un}}

\newcommand{\sHom}{\mathscr{H}\!\!\mathit{om}}
\newcommand{\sIsom}{\mathscr{I}\!\!\mathit{som}}
\newcommand{\sExt}{\mathscr{E}\!\mathit{xt}}

\newcommand{\sAut}{\mathscr{A}\!\!\mathit{ut}}

\newcommand{\okey}{\mathcal{O}_K}

\newcommand{\cA}{\mathcal{A}}

\newcommand{\cC}{\mathcal{C}}
\newcommand{\cD}{\mathcal{D}}
\newcommand{\cE}{\mathcal{E}}

\newcommand{\cG}{\mathcal{G}}
\newcommand{\cH}{\mathcal{H}}

\newcommand{\cL}{\mathcal{L}}
\newcommand{\cM}{\mathcal{M}}

\newcommand{\cO}{\mathcal{O}}
\newcommand{\cP}{\mathcal{P}}
\newcommand{\cS}{\mathcal{S}}
\newcommand{\cT}{\mathcal{T}}
\newcommand{\cU}{\mathcal{U}}

\newcommand{\cZ}{\mathcal{Z}}

\newcommand{\Ga}{\mathbb{G}_\mathrm{a}}
\newcommand{\Gm}{\mathbb{G}_\mathrm{m}}

\newcommand{\Crys}{\mathrm{Crys}}

\newcommand{\Gr}{\mathrm{Gr}}

\newcommand{\phiMod}{\varphi\text{-}\mathrm{Mod}}
\newcommand{\phiShv}{\varphi\text{-}\mathrm{Shv}}
\newcommand{\vMod}{v\text{-}\mathrm{Mod}}
\newcommand{\vShv}{v\text{-}\mathrm{Shv}}

\newcommand{\bA}{\mathbb{A}}

\newcommand{\bC}{\mathbb{C}}

\newcommand{\bE}{\mathbb{E}}
\newcommand{\bF}{\mathbb{F}}

\newcommand{\bP}{\mathbb{P}}

\newcommand{\bR}{\mathbb{R}}
\newcommand{\bS}{\mathbb{S}}

\newcommand{\bV}{\mathbb{V}}
\newcommand{\bZ}{\mathbb{Z}}

\newcommand{\sH}{\mathscr{H}}

\newcommand{\frem}{\mathfrak{m}}
\newcommand{\frn}{\mathfrak{n}}

\newcommand{\frq}{\mathfrak{q}}

\newcommand{\frX}{\mathfrak{X}}

\newcommand{\Derin}{\mathrm{Der}_{\mathrm{in}}}
\newcommand{\Dersi}{\mathrm{Der}_{\mathrm{si}}}

\makeatletter
    
    \@addtoreset{equation}{section}
\makeatother

\makeatletter
\renewcommand{\p@enumii}{}
\makeatother

\begin{document}

\title[$\wp$-adic properties of Drinfeld modular forms]{Duality of Drinfeld modules and $\wp$-adic properties of Drinfeld modular forms}
\author{Shin Hattori}
\address[Shin Hattori]{Faculty of Mathematics, Kyushu University}

\date{\today}


\begin{abstract}
Let $p$ be a rational prime and $q$ a power of $p$. Let $\wp$ be a monic irreducible polynomial of degree $d$ in 
$\bF_q[t]$. In this paper, we define an analogue of the Hodge-Tate map which is suitable for the study of Drinfeld 
modules over $\bF_q[t]$ and, using it, develop a geometric theory of $\wp$-adic Drinfeld modular forms similar to 
Katz's theory in the case of elliptic modular forms. In particular, we show that for Drinfeld modular forms with 
congruent Fourier coefficients at $\infty$ modulo $\wp^n$, their weights are also congruent modulo $(q^d-1)p^{\lceil 
\log_p(n)\rceil}$, 
and that Drinfeld modular forms of level $\Gamma_1(\frn)\cap \Gamma_0(\wp)$, weight $k$ and type $m$ are $\wp$-adic 
Drinfeld modular forms for any tame level $\frn$ with a prime factor of degree prime to $q-1$. 
\end{abstract}

\maketitle
\tableofcontents



\section{Introduction}

Let $p$ be a rational prime and $q$ a power of $p$. The theory of $p$-adic modular forms, which originated from the 
work of Serre \cite{Ser_P}, has been highly developed, and now we have various $p$-adic families of eigenforms which 
play important roles in modern number theory. At the early stage of its development, Katz \cite{Katz_p} initiated a 
geometric treatment of $p$-adic modular forms, and from the work of Katz to recent works on geometric study of $p$-adic 
modular forms including \cite{AIS,AIP-SMF,Pil_oc}, one of the key ingredients is the theory of canonical subgroups of 
abelian varieties and Hodge-Tate maps for finite locally free (commutative) group schemes. 

Let us briefly recall the definition. For a finite locally free group scheme $\cG$ over a scheme $S$, we denote by $\omega_{\cG}$ the sheaf of invariant differentials of $\cG$ and by $\Car(\cG)$ the Cartier dual of $\cG$. Then the Hodge-Tate map for $\cG$ is by definition
\[
\Car(\cG)=\sHom_S(\cG,\Gm)\to \omega_{\cG},\quad x\mapsto x^*\left(\frac{d T}{T} \right).
\]
It can be considered as a comparison map between the etale side and the de Rham side; in fact, for any abelian scheme $\cA$ with ordinary reduction over a complete discrete valuation ring $\cO$ of mixed characteristic $(0,p)$, the Cartier dual $\Car(\cA[p^n]^0)$ of the unit component of $\cA[p^n]^0$ is etale, and the Hodge-Tate map gives an isomorphism of $\cO/(p^n)$-modules
\[
\Car(\cA[p^n]^0)\otimes_{\bZ} \cO \to \omega_{\cA} \otimes_{\cO}\Spec(\cO/(p^n)).
\]
Moreover, if $\cA$ is close enough to having ordinary reduction, then there exists a canonical subgroup of $\cA$ which has a similar comparison property via the Hodge-Tate map, instead of $\cA[p^n]^0$.

On the other hand, an analogue of the theory of $p$-adic modular forms in the function field case---the theory of $v$-adic modular forms---has also been actively investigated in this decade (see for example \cite{Goss_v,Petrov,Vincent}). A Drinfeld modular form is a rigid analytic function on the Drinfeld upper half plane over $\bF_q((1/t))$, and it can be viewed as a section of an automorphic line bundle over a Drinfeld modular curve. The latter is a moduli space over $\bF_q(t)$ classifying Drinfeld modules (of rank two), which are analogues of elliptic curves. 
It is widely believed that, for each finite place $v$ of $\bF_q(t)$, Drinfeld modular forms have deep $v$-adic structures comparable to the $p$-adic theory of modular forms. However, we do not fully understand what it is like yet.

What is lacking is a geometric description of $v$-adic modular forms as in \cite{Katz_p}. For this, the problem is that the usual Cartier duality does not work in the Drinfeld case: Since Drinfeld modules are additive group schemes, the Cartier dual of any non-trivial finite locally free closed subgroup scheme of a Drinfeld module is never etale and we cannot obtain an etale-to-de Rham comparison isomorphism via the Hodge-Tate map. 

In this paper, we resolve this and develop a geometric theory of $v$-adic Drinfeld modular forms. In particular, we show the following theorems.

\begin{thm}[Corollary to Theorem \ref{WeightCongr}]\label{MainWeight}
Let $\frn$ be a monic polynomial in $A=\bF_q[t]$ and $\wp$ a monic irreducible polynomial in $A$ which is prime to $\frn$.
For $i=1,2$, let $f_i$ be a Drinfeld modular form of level $\Gamma_1(\frn)$, weight $k_i$ and type $m_i$. Suppose that their Fourier expansions $(f_i)_\infty(x)$ at $\infty$ in the sense of Gekeler \cite{Gek_Coeff} have coefficients in the localization $A_{(\wp)}$ of $A$ at $(\wp)$ and satisfy the congruence
\[
(f_1)_\infty(x)\equiv (f_2)_\infty(x)\notequiv 0 \bmod \wp^n.
\]
Then we have 
\[
k_1\equiv k_2 \bmod (q^d-1)p^{l_p(n)},\quad l_p(n)=\min\{N\in \bZ\mid p^N\geq n\}.
\]
\end{thm}

\begin{thm}[Theorem \ref{GammaWp}]\label{MainV}
Suppose that $\frn$ has a prime factor of degree prime to $q-1$. Let $f$ be a Drinfeld modular form of level $\Gamma_1(\frn)\cap \Gamma_0(\wp)$, weight $k$ and type $m$ such that Gekeler's Fourier expansion $f_\infty(x)$ at $\infty$ has coefficients in $A_{(\wp)}$. Then $f$ is a $\wp$-adic Drinfeld modular form. Namely, $f_\infty(x)$ is the $\wp$-adic limit of Fourier expansions of Drinfeld modular forms of level $\Gamma_1(\frn)$, type $m$ and some weights.
\end{thm}

Note that Theorem \ref{MainWeight} generalizes \cite[Corollary (12.5)]{Gek_Coeff} of the case $n=1$, and Theorem \ref{MainV} is a variant of \cite[Theorem 4.1]{Vincent} with non-trivial tame level $\frn$.

The novelty of this paper lies in the systematic use of the duality theory of Taguchi \cite{Tag} for Drinfeld modules and a certain class of finite locally free group schemes called finite $v$-modules. Using Taguchi's duality, we define an analogue of the Hodge-Tate map, which we refer to as the Hodge-Tate-Taguchi map. For a Drinfeld module $E$ with ordinary reduction, we construct canonical subgroups of $E$ such that their Taguchi duals are etale and the Hodge-Tate-Taguchi maps for them give isomorphisms between the etale and de Rham sides similar to the case of elliptic curves. This enables us to prove the above theorems by almost verbatim arguments as in \cite{Katz_p}.

The organization of this paper is as follows. In \S\ref{Sec_Taguchi}, we review Taguchi's duality theory. Here we need a description of the duality for Drinfeld modules in terms of biderivations \cite{Gek_dR}, which is done by Papanikolas-Ramachandran \cite{PapRam} in the case over fields. For this reason, we follow the exposition of \cite{PapRam} and generalize their results to general bases. 

In \S\ref{Sec_Cansub}, we develop the theory of canonical subgroups of Drinfeld modules with ordinary reduction and Hodge-Tate-Taguchi maps. In our case, the role of $\mu_{p^n}$ for elliptic curves is played by the $\wp^n$-torsion part $C[\wp^n]$ of the Carlitz module $C$, where the dual of $C[\wp^n]$ in the sense of Taguchi is the constant $A$-module scheme $\underline{A/(\wp^n)}$.

\S\ref{Sec_DMC} is devoted to a study of Drinfeld modular curves and their compactifications via Tate-Drinfeld modules in a similar way to \cite{KM}. Though it may be classical, we give necessary details due to the lack of appropriate references. The main differences from \cite{KM} are threefold: First, the $j$-invariant of the usual Tate-Drinfeld module does not give (the inverse of) a uniformizer of the $j$-line at the infinity, contrary to the case of the Tate curve. For this, we use a descent of the Tate-Drinfeld module by an $\bF_q^\times$-action on the coefficients to obtain a right $j$-invariant (see (\ref{EqnTDdj})). The author learned this idea from a work of Armana \cite{Armana}. Second, a Drinfeld module is not dual to itself in general, while every elliptic curve has autoduality. Instead, we have a weak version of autoduality for Drinfeld modules (Remark \ref{RmkAutodual}), which is enough to show that the square of the Hodge bundle in our case is base point free. Third, since we are in the positive characteristic situation, we cannot use Abhyankar's lemma to study the structure of Drinfeld modular curves around cusps. This is bypassed by a direct computation of the formal completion along each cusp (Corollary \ref{GammaDeltaCusp}). 

Then in \S\ref{Sec_DMF} we prove the main theorems in a similar way to \cite[Chapter 4]{Katz_p}, the point being the 
fact that the Riemann-Hilbert correspondence of Katz over the truncated Witt ring $W_n(\bF_q)$ \cite[Proposition 
4.1.1]{Katz_p} can be suitably generalized to the case over $A/(\wp^n)$.

\subsection*{Acknowledgments} The author would like to thank Yuichiro Taguchi for directing the author's attention to arithmetic of function fields, and also for answering many questions on his duality theory. The author also would like to thank Gebhard B\"{o}ckle and Rudolph Perkins for enlightening discussions on Drinfeld modules and Drinfeld modular forms. A part of this work was done during the author's visit to Interdisciplinary Center for Scientific Computing, Heidelberg University. He is grateful to its hospitality. This work was supported by JSPS KAKENHI Grant Numbers JP26400016, JP17K05177.



\section{Taguchi duality}\label{Sec_Taguchi}

In this section, we review the duality theory for Drinfeld modules of rank two and an analogue of Cartier duality for 
this context, which are both due to Taguchi \cite{Tag}. Let $p$ be a rational prime, $q$ a $p$-power and $\bF_q$ the 
finite field with $q$ elements. We put $A=\bF_q[t]$. For any scheme $S$ over $\bF_q$, we denote the $q$-th power 
Frobenius map on $S$ by $F_S:S\to S$.
For any $S$-scheme $T$ and $\cO_S$-module $\cL$, we put $T^{(q)}=T\times_{S,F_S}S$ and $\cL^{(q)}=F_S^*(\cL)$. 
Note that for any $\cO_S$-algebra $\cA$, the $q$-th power Frobenius map induces an $\cO_S$-algebra homomorphism 
$f_{\cA}:\cA^{(q)}\to \cA$. 
For any $A$-scheme $S$, the image of $t\in A$ by the structure map $A\to \cO_S(S)$ is denoted by $\theta$.

\subsection{Line bundles and Drinfeld modules}

For any scheme $S$ over $\bF_q$ and any invertible $\cO_S$-module $\cL$, we write the associated covariant and contravariant line bundles to $\cL$ as
\[
\bV_*(\cL)=\Spec_S(\Sym_{\cO_S}(\cL^{\otimes -1})), \quad \bV^*(\cL)=\Spec_S(\Sym_{\cO_S}(\cL))
\]
with $\cL^{\otimes -1}:=\cL^\vee=\sHom_{\cO_S}(\cL,\cO_S)$. Note that they represent functors over $S$ defined by
\[
T\mapsto \cL|_T(T),\quad T\mapsto \cL^{\otimes -1}|_T(T),
\]
where $\cL|_T$ and $\cL^{\otimes -1}|_T$ denote the pull-backs to $T$. The additive group $\Ga$ acts on the group schemes $\bV_*(\cL)$ and $\bV^*(\cL)$ through the natural actions of $\cO_T(T)$ on $\cL|_T(T)$ and $\cL^{\otimes -1}|_T(T)$, respectively. We often identify $\cL$ with $\bV_*(\cL)$. We have the $q$-th power Frobenius map
\[
\tau:\cL\to \cL^{\otimes q},\quad l\mapsto l^{\otimes q}.
\]
This map induces a homomorphism of group schemes over $S$
\[
\tau: \bV_*(\cL)\to \bV_*(\cL^{\otimes q}).
\]
Note that $\tau$ also induces an $\cO_S$-linear isomorphism $\cL^{(q)}\to \cL^{\otimes q}$, by which we identify $\bV_*(\cL^{\otimes q})$ with $\bV_*(\cL)^{(q)}$. Then the relative $q$-th Frobenius map $\bV_*(\cL)\to \bV_*(\cL)^{(q)}=\bV_*(\cL^{\otimes q})$ is induced by the natural inclusion
\begin{equation}\label{EqnVectFrob}
\Sym(\cL^{\otimes -q})\to \Sym(\cL^{\otimes -1}).
\end{equation}

For $S=\Spec(B)$ and $\cL=\cO_S$, we have $\bV_*(\cO_S)=\Ga$ and $\tau$ induces the endomorphism of $\Ga=\Spec(B[X])$ over $B$ defined by $X\mapsto X^q$. This gives the equality
\[
\End_{\bF_q,S}(\Ga)= B\{\tau\},
\]
where $B\{\tau\}$ is the skew polynomial ring over $B$ whose multiplication is defined by $a\tau^i\cdot b\tau^j=a b^{q^i}\tau^{i+j}$ for any $a,b\in B$.

\begin{dfn}[\cite{Laumon}, Remark (1.2.2)]
Let $S$ be a scheme over $A$ and $r$ a positive integer. A (standard) Drinfeld ($A$-)module of rank $r$ over $S$ is a pair $E=(\cL,\Phi^E)$ of an invertible sheaf $\cL$ on $S$ and an $\bF_q$-algebra homomorphism 
\[
\Phi^E:A\to \End_S(\bV_*(\cL))
\]
satisfying the following conditions for any $a\in A\setminus \{0\}$:
\begin{itemize}
\item the image $\Phi^E_a$ of $a$ by $\Phi^E$ is written as
\[
\Phi^E_a=\sum_{i=0}^{r\deg(a)} \alpha_i(a)\tau^i,\quad \alpha_i(a)\in \cL^{\otimes 1-q^i}(S)
\]
with $\alpha_{r\deg(a)}(a)$ nowhere vanishing.
\item $\alpha_0(a)$ is equal to the image of $a$ by the structure map $A\to \cO_S(S)$.
\end{itemize}
We often refer to the underlying $A$-module scheme $\bV_*(\cL)$ as $E$.
A morphism $(\cL,\Phi)\to (\cL',\Phi')$ of Drinfeld modules over $S$ is defined to be a morphism of $A$-module schemes $\bV_*(\cL)\to \bV_*(\cL')$ over $S$. The category of Drinfeld modules over $S$ is denoted by $\DM_S$.
\end{dfn}

We denote the Carlitz module over $S$ by $C$: it is the Drinfeld module $(\cO_S,\Phi^C)$ of rank one over $S$ defined by $\Phi^C_t=\theta +\tau$.
We identify the underlying group scheme of $C$ with $\Ga=\Spec_S(\cO_S[Z])$ using $1\in \cO_S(S)$.



\subsection{$\varphi$-modules and $v$-modules}

Let $S$ be a scheme over $A$. Let $\cG$ be an $\bF_q$-module scheme $\cG$ over $S$ whose structure map $\pi:\cG\to S$ is affine. Note that the additive group $\Ga$ over $S$ is endowed with a natural action of $\bF_q$. Put 
\[
\cE_\cG=\sHom_{\bF_q,S}(\cG,\Ga),
\]
the $\cO_S$-module of $\bF_q$-linear homomorphisms $\cG\to \Ga$ over $S$. The Zariski sheaf $\cE_\cG$ is naturally considered as an $\cO_S$-submodule of $\pi_*(\cO_\cG)$.

On the other hand, if the formation of $\cE_\cG$ commutes with any base change, then the relative $q$-th Frobenius map $F_{\cG/S}:\cG\to \cG^{(q)}$ defines an $\cO_S$-linear map
\[
\varphi_\cG:\cE_{\cG^{(q)}}=\cE_{\cG}^{(q)}\to \cE_\cG
\]
which commutes with $\bF_q$-actions.

\begin{dfn}\label{DfnPhiMod}
We say an $\bF_q$-module scheme $\cG$ over $S$ is a $\varphi$-module over $S$ if the following conditions hold:
\begin{itemize}
\item the structure morphism $\pi:\cG\to S$ is affine,
\item the $\cO_S$-module $\cE_\cG$ is locally free (not necessarily of finite rank) and its formation commutes with any base change,
\item the induced $\bF_q$-action on the sheaf of invariant differentials $\omega_\cG$ agrees with the action via the structure map $\bF_q\to \cO_S(S)$,
\item the natural $\cO_S$-algebra homomorphism $\cS:=\Sym_{\cO_S}(\cE_\cG)\to \pi_*(\cO_\cG)$ induces an isomorphism
\[
\cS/((f_{\cS}\otimes 1-\varphi_{\cG})(\cE_{\cG}^{(q)}))\to \pi_*(\cO_\cG).
\]
\end{itemize}
A morphism of $\varphi$-modules over $S$ is defined as a morphism of $\bF_q$-module schemes over $S$. The category of $\varphi$-modules over $S$ is denoted by $\phiMod_S$.
\end{dfn}
The last condition of Definition \ref{DfnPhiMod} yields a natural isomorphism
\[
\Coker(\varphi_{\cG})\to \omega_\cG.
\]
We also note that for any $\varphi$-module $\cG$ over $S$, the natural map $\Sym_{\cO_S}(\cE_\cG)\to \pi_*(\cO_{\cG})$ defines a closed immersion of $\bF_q$-module schemes
\[
i_\cG:\cG\to \bV^*(\cE_\cG).
\]

\begin{dfn}
A $\varphi$-sheaf over $S$ is a pair $(\cE,\varphi_{\cE})$ of a locally free $\cO_S$-module $\cE$ and an $\cO_S$-linear 
homomorphism $\varphi_{\cE}:\cE^{(q)}\to \cE$. We abusively denote the pair $(\cE,\varphi_\cE)$ by $\cE$. A morphism 
of $\varphi$-sheaves is defined as a morphism of $\cO_S$-modules compatible with $\varphi_{\cE}$'s. A sequence of 
$\varphi$-sheaves is said to be exact if the underlying sequence of $\cO_S$-modules is exact. The exact category of 
$\varphi$-sheaves over $S$ is denoted by $\phiShv_S$.
\end{dfn}
We have a contravariant functor
\[
\Sh: \phiMod_S\to \phiShv_S,\quad \cG\mapsto (\cE_{\cG},\varphi_{\cG}).
\]
On the other hand, for any object $(\cE,\varphi_{\cE})$ of the category $\phiShv_S$, put $\cS_{\cE}=\Sym_{\cO_S}(\cE)$ and 
\[
\Gr(\cE)=\Spec_S(\cS_\cE/((f_{\cS_\cE}\otimes 1-\varphi_\cE)(\cE^{(q)}))).
\]
Then the diagonal map $\cE\to \cE\oplus\cE$ and the natural $\bF_q$-action on $\cE$ define on $\Gr(\cE)$ a structure of an affine $\bF_q$-module scheme over $S$. The formation of $\Gr(\cE)$ is compatible with any base change. We also have a natural identification
\begin{equation}\label{GrDef}
\Gr(\cE)(T)=\Hom_{\varphi, \cO_S}(\cE,\pi_*(\cO_T))
\end{equation}
for any morphism $\pi:T\to S$, where we consider on $\pi_*(\cO_T)$ the natural $\varphi$-structure induced by the $q$-th power Frobenius map \cite[Proposition (1.8)]{Tag}.
Since we have a natural isomorphism $\cE\to\cE_{\Gr(\cE)}$, we obtain a contravariant functor
\[
\Gr: \phiShv_S\to \phiMod_S,\quad \cE\mapsto \Gr(\cE),
\]
which gives an anti-equivalence of categories with quasi-inverse $\Sh$. 

A sequence of $\varphi$-modules is said to be $\Shv$-exact if the corresponding sequence in the category $\phiShv_S$ via the functor $\Sh$ is exact. We consider $\phiMod_S$ as an exact category by this notion of exactness. The author does not know if it is equivalent to the exactness as group schemes.

The commutativity of $\cE_\cG$ with any base change in Definition \ref{DfnPhiMod} holds in the case where $\cG$ is a line bundle over $S$. From this we can show that any Drinfeld module is a $\varphi$-module. Another case it holds is that of finite $\varphi$-modules, which is defined as follows.

\begin{dfn}[\cite{Tag}, Definition (1.3)] 
We say an $\bF_q$-module scheme $\cG$ over $S$ is a finite $\varphi$-module over $S$ if the following conditions hold:
\begin{itemize}
\item the structure morphism $\pi:\cG\to S$ is affine,
\item the induced $\bF_q$-action on $\omega_\cG$ agrees with the action via the structure map $\bF_q\to \cO_S(S)$,
\item the $\cO_S$-modules $\pi_*(\cO_\cG)$ and $\cE_\cG$ are locally free of finite rank with
\[
\rank_{\cO_S}(\pi_*(\cO_\cG))=q^{\rank_{\cO_S}(\cE_\cG)},
\]
\item $\cE_\cG$ generates the $\cO_S$-algebra $\pi_*(\cO_\cG)$.
\end{itemize}
A morphism of finite $\varphi$-modules over $S$ is defined as a morphism of $\bF_q$-module schemes over $S$. 
\end{dfn}

\begin{dfn}
A finite $\varphi$-sheaf over $S$ is a $\varphi$-sheaf such that its underlying $\cO_S$-module is locally free of finite rank. The full subcategory of $\phiShv_S$ consisting of finite $\varphi$-sheaves is denoted by $\phiShv_S^f$.
\end{dfn}

Let $\cG$ be a finite $\varphi$-module over $S$. Then we also have the natural closed immersion $i_\cG:\cG\to \bV^*(\cE_\cG)$, which implies that the Cartier dual $\Car(\cG)$ of $\cG$ is of height $\leq 1$ in the sense of \cite[\S4.1.3]{SGA3-7A}. Then, by \cite[Th\'{e}or\`{e}me 7.4, footnote]{SGA3-7A}, the sheaf of invariant differentials $\omega_{\Car(\cG)}$ is a locally free $\cO_S$-module of finite rank, and thus the formation of the Lie algebra
\[
\Lie(\Car(\cG))\simeq \sHom_S(\cG,\Ga)
\]
commutes with any base change. Since $q-1$ is invertible in $\cO_S(S)$, the $\cO_S$-module $\cE_\cG$ is the image of the projector 
\[
\Lie(\Car(\cG))\to \Lie(\Car(\cG)),\quad x\mapsto \frac{1}{q-1}\sum_{a\in \bF_q^\times}\alpha(a)^{-1}\psi_a(x),
\]
where $\alpha:A\to \cO_S(S)$ is the structure map and $\psi_a$ is the action of $a$ on $\Lie(\Car(\cG))$ induced by the $\bF_q$-action on $\cG$. Since the formation of this projector commutes with any base change, so does that of $\cE_\cG$. From this we see that any finite $\varphi$-module is a $\varphi$-module. We denote by $\phiMod_S^f$ the full subcategory of $\phiMod_S$ consisting of finite $\varphi$-modules. Then the functor $\Gr$ gives an anti-equivalence of categories $\phiShv^f_S\to \phiMod_S^f$ with quasi-inverse given by $\Sh$.

On the category $\phiMod_S^f$, the $\Shv$-exactness agrees with the usual exactness of group schemes. Indeed, from (\ref{GrDef}) and comparing ranks we see that the $\Shv$-exactness implies the usual exactness, and the converse also follows by using the compatibility of $\Sh$ with any base change and reducing to the case over a field by Nakayama's lemma.

\begin{lem}\label{DMFF}
Let $E$ be a line bundle over $S$.
Let $\cG$ be a finite locally free closed $\bF_q$-submodule scheme of $E$ over $S$. Suppose that the rank of $\cG$ is a $q$-power. Then $\cG$ is a finite $\varphi$-module.
\end{lem}
\begin{proof}
We may assume that $S=\Spec(B)$ is affine, the underlying invertible sheaf of $E$ is trivial and $\cG=\Spec(B_\cG)$ is free of rank $q^n$ over $S$. We write as $E=\Spec(B[X])$. We have a surjection $B[X]\to B_\cG$ of Hopf algebras over $B$. Let $P(X)\in B[X]$ be the characteristic polynomial of the action of $X$ on $B_\cG$. Since $\deg(P(X))=q^n$, the Cayley-Hamilton theorem implies that this surjection induces an isomorphism $B[X]/(P(X))\simeq B_\cG$. 

Since $P(X)$ is monic, we can see that $P(X)$ is an additive polynomial as in \cite[\S8, Exercise 7]{Waterhouse}. Since $\cG$ is stable under the $\bF_q$-action on $\Ga$, we have the equality of ideals $(P(\lambda X))=(P(X))$ of $B[X]$ for any $\lambda\in\bF_q^\times$. 
From this we see that $P(X)$ is $\bF_q$-linear and
\[
\cE_{\cG}=\bigoplus_{i=0}^{n-1}BX^{q^i},
\]
from which the lemma follows.
\end{proof}

\begin{cor}\label{DMIFF}
Let $\pi: E\to F$ be an $\bF_q$-linear isogeny of line bundles over $S$. Then the group scheme $\cG=\Ker(\pi)$ is a finite $\varphi$-module over $S$, and we have a natural exact sequence of $\varphi$-sheaves
\begin{equation}\label{EqnExactShvIsog}
\xymatrix{
0 \ar[r] & \cE_{F} \ar[r] & \cE_{E} \ar[r] & \cE_{\cG} \ar[r] & 0.
}
\end{equation}
\end{cor}
\begin{proof}
The first assertion follows from Lemma \ref{DMFF}. For the second one, it is enough to show the surjectivity of the natural map $i^*:\cE_E\to \cE_{\cG}$. By Nakayama's lemma, we may assume $S=\Spec(k)$ for some field $k$. Then $\pi$ is defined by an $\bF_q$-linear additive polynomial as
\[
X\mapsto P(X)=a_0 X+a_1 X^q+\cdots + a_n X^{q^n},\quad a_n\neq 0
\]
and the map $i^*$ is identified with the natural map
\[
\cE_E=\bigoplus_{i\in\bZ_{\geq 0}} k X^{q^i}\to \cE_{\cG}=\bigoplus_{i=0}^{n-1} k X^{q^i}
\]
of taking modulo $\bigoplus_{l\geq 0} k P(X)^{q^l}$. Hence $i^*$ is surjective.
\end{proof}

\begin{lem}\label{DMQFF}
\begin{enumerate}
\item\label{DMQFF-Phi} Let $E$ be a line bundle over $S$. Let $\cG$ and $\cH$ be finite locally free closed $\bF_q$-submodule schemes of $E$ over $S$ satisfying $\cH\subseteq \cG$. Suppose that the ranks of $\cG$ and $\cH$ are constant $q$-powers. Then $E/\cH$ is a line bundle over $S$ and $\cG/\cH$ is a finite $\varphi$-module over $S$. 

\item\label{DMQFF-Dr} Let $E$ be a Drinfeld module of rank $r$. Let $\cH$ be a finite locally free closed $A$-submodule scheme of $E$ of constant $q$-power rank over $S$. Suppose either 
\begin{itemize}
	\item $\cH$ is etale over $S$, or
	\item $S$ is reduced and for any maximal point $\eta$ of $S$, the fiber $\cH_\eta$ of $\cH$ over $\eta$ is etale. 
	\end{itemize}
Then $E/\cH$ is a Drinfeld module of rank $r$ with the induced $A$-action. 
\end{enumerate}
\end{lem}
\begin{proof}
	For (\ref{DMQFF-Phi}), \cite[Ch.~1, Proposition 3.2]{Lehmkuhl} implies that $E/\cH$ is a line bundle 
	over $S$. Moreover, applying Lemma \ref{DMFF} to the natural closed immersion $\cG/\cH\to E/\cH$, we see that
	$\cG/\cH$ is a finite $\varphi$-module over $S$.

For (\ref{DMQFF-Dr}), we may assume that $S=\Spec(B)$ is affine, the underlying invertible sheaves of $E$ and $E/\cH$ 
are trivial and 
$\cH$ is free of rank $q^n$ over $S$. We write the $t$-multiplication maps of $E$ and $E/\cH$ as 
\[
\Phi^E_t(X)=\theta X+a_1 X^q+\cdots+a_r X^{q^r},\quad \Phi^{E/\cH}_t(X)=b_0 X+b_1 X^q+\cdots + b_s X^{q^s}
\]
with $b_s\neq 0$. From the proof of \cite[Ch.~1, Proposition 3.2]{Lehmkuhl}, we may also assume that the map $E\to 
E/\cH$ is 
defined by an $\bF_q$-linear monic additive polynomial
\[
X\mapsto P(X)=p_1 X+\cdots+ p_{n-1} X^{q^{n-1}}+X^{q^n}.
\]
From the equality $\Phi^{E/\cH}_t(P(X))=P(\Phi^E_t(X))$, we obtain $r=s$, $b_r=a_r^{q^n}$ and $p_1(b_0-\theta)=0$. If $\cH$ is etale over $B$, then we have $p_1\in B^\times$ and thus $b_0=\theta$. If the latter assumption in the lemma holds, then $p_1\in B$ is a non-zero divisor in the ring $B/\mathfrak{p}$ for any minimal prime ideal $\mathfrak{p}$. Since $B$ is reduced, it is a subring of $\prod B/\mathfrak{p}$,
where the product is taken over the set of minimal prime ideals $\mathfrak{p}$ of $B$. This also yields $b_0=\theta$, and thus $E/\cH$ is a Drinfeld module of rank $r$ in both cases. 
\end{proof}

\begin{dfn}[\cite{Tag}, Definition (2.1)]\label{DfnTMod} We say an $A$-module scheme $\cG$ over $S$ is a $t$-module 
over $S$ if the following conditions hold:
\begin{itemize}
\item the induced $A$-action on $\omega_\cG$ agrees with the action via the structure map $A\to \cO_S(S)$, 
\item the underlying $\bF_q$-module scheme of $\cG$ is a $\varphi$-module over $S$. 
\end{itemize}
We say $\cG$ is a finite $t$-module if in addition the underlying $\bF_q$-module scheme of $\cG$ is a finite $\varphi$-module over $S$.
\end{dfn}
Note that the former condition in Definition \ref{DfnTMod} is automatically satisfied if $\cG$ is etale.

\begin{lem}\label{DMQFt}
Let $E$ be a line bundle over $S$. Let $\cG$ and $\cH$ be finite locally free closed $\bF_q$-submodule schemes of $E$ over $S$ satisfying $\cH\subseteq \cG$. Suppose that $\cG$ is endowed with a $t$-action which makes it a finite $t$-module, $\cH$ is stable under the $A$-action on $\cG$ and the ranks of $\cG$ and $\cH$ are constant $q$-powers. 

\begin{enumerate}
	\item\label{DMQFt-Sub}
	The $A$-module scheme $\cH$ is a finite $t$-module over $S$.
	\item\label{DMQFt-Quot}
	Suppose moreover that $a\cG=0$ for some $\cO_S$-regular element $a\in A$. Then the $A$-module scheme $\cG/\cH$ is a finite $t$-module over $S$.
\end{enumerate}
\end{lem}
\begin{proof}
From Lemma \ref{DMFF} and Lemma \ref{DMQFF} (\ref{DMQFF-Phi}), we see that $\cH$ and $\cG/\cH$ are finite $\varphi$-modules. We have an exact sequence of $\cO_S$-modules
\[
\xymatrix{
 \omega_{\cG/\cH} \ar[r]^-{\pi^*} & \omega_\cG \ar[r] & \omega_\cH \ar[r] & 0
}
\]
which is compatible with $A$-actions. Since the $t$-action on $\omega_\cG$ is equal to the multiplication by $\theta$, 
so is that on $\omega_\cH$ and (\ref{DMQFt-Sub}) follows. For (\ref{DMQFt-Quot}), using co-Lie complexes we can deduce from 
the assumption that the map $\pi^*$ is injective. This yields (\ref{DMQFt-Quot}).
\end{proof}

\begin{dfn}[\cite{Tag}, Definition (3.1)] A $v$-module over $S$ is a pair $(\cG,v_\cG)$ of a $t$-module $\cG$ and an $\cO_S$-linear map $v_\cG:\cE_{\cG}\to \cE^{(q)}_{\cG}$ such that the map $\psi^{\cG}_t: \cE_{\cG}\to \cE_{\cG}$ induced by the $t$-action on $\cG$ satisfies 
\[
\psi_t^\cG=\theta+\varphi_{\cG}\circ v_\cG,\quad (\psi_t^\cG\otimes 1)\circ v_\cG=v_\cG\circ \psi_t^\cG.
\]
We refer to such $v_\cG$ as a $v$-structure on $\cG$ and denote the pair $(\cG,v_\cG)$ abusively by $\cG$. A morphism $g:\cG\to \cH$ of $v$-modules over $S$ is defined as a morphism of $A$-module schemes over $S$ which commutes with $v$-structures, in the sense that the following diagram is commutative.
\[
\xymatrix{
\cE_{\cH}\ar[r]^{v_{\cH}}\ar[d]_{g^*} & \cE_{\cH}^{(q)}\ar[d]^{g^* \otimes 1}\\
\cE_{\cG}\ar[r]_{v_\cG} & \cE_{\cG}^{(q)}
}
\]
A sequence of $v$-modules over $S$ is said to be exact if the underlying sequence of $\varphi$-modules is $\Shv$-exact. The category of $v$-modules over $S$ is denoted by $\vMod_S$. 

A $v$-module over $S$ is said to be a finite $v$-module if the underlying $\varphi$-module is a finite $\varphi$-module. The full subcategory of $\vMod_S$ consisting of finite $v$-modules is denoted by $\vMod^f_S$.
\end{dfn}

\begin{dfn}[\cite{Tag}, Definition (3.2)]
A $v$-sheaf over $S$ is a quadruple $(\cE,\varphi_{\cE},\psi_{\cE,t},v_\cE)$, which we abusively write as $\cE$, consisting of the following data:
\begin{itemize}
\item $(\cE,\varphi_{\cE})$ is a $\varphi$-sheaf over $S$,
\item $\psi_{\cE,t}:\cE\to \cE$ is an $\cO_S$-linear map which commutes with $\varphi_\cE$,
\item $v_\cE:\cE\to \cE^{(q)}$ is an $\cO_S$-linear map which commutes with $\psi_{\cE,t}$ and satisfies $\psi_{\cE,t}=\theta+\varphi_{\cE}\circ v_\cE$.
\end{itemize}
A morphism of $v$-sheaves is defined as a morphism of underlying $\cO_S$-modules which is compatible with the other data, and we say that a sequence of $v$-sheaves is exact if the underlying sequence of $\cO_S$-modules is exact. The exact category of $v$-sheaves over $S$ is denoted by $\vShv_S$.

A $v$-sheaf is said to be a finite $v$-sheaf if the underlying $\cO_S$-module is locally free of finite rank. The full subcategory of $\vShv_S$ consisting of finite $v$-sheaves is denoted by $\vShv_S^f$.
\end{dfn}
Then the functor $\Gr$ induces anti-equivalences of categories
\[
\vShv_S\to \vMod_S ,\quad \vShv^f_S\to \vMod^f_S
\]
with quasi-inverses given by $\Sh$.

Note that for any $v$-module (\textit{resp.} finite $v$-module) $\cG$ over $S$ and any $S$-scheme $T$, the base change $\cG|_T=\cG\times_S T$ has a natural structure of a $v$-module (\textit{resp.} finite $v$-module) over $T$. For any Drinfeld module $E$ over $S$, the map $\varphi_E: \cE_E^{(q)}\to \cE_E$ is injective and $\Coker(\varphi_E)$ is killed by $\psi^E_t-\theta$. Then $E$ has a unique $v$-structure 
\[
v_E=\varphi^{-1}_E\circ (\psi^E_t-\theta)
\]
and any morphism of Drinfeld modules is compatible with the unique $v$-structures. Thus we may consider the category $\DM_S$ as a full subcategory of $\vMod_S$. Moreover, for any isogeny $\pi:E\to F$ of Drinfeld modules over $S$, Corollary \ref{DMIFF} implies that $\Ker(\pi)$ has a unique structure of a finite $v$-module such that the exact sequence (\ref{EqnExactShvIsog}) is compatible with $v$-structures. Note that a $v$-structure of $\Ker(\pi)$ is not necessarily unique without this compatibility condition.
On the other hand, in some cases a finite $t$-module over $S$ has a unique $v$-structure, as follows.

\begin{lem}[\cite{Tag}, Proposition 3.5]\label{SredV}
Let $\cG$ be a finite $t$-module over $S$. Suppose either
\begin{enumerate}
\item\label{SredV-Et} $\cG$ is etale over $S$, or
\item\label{SredV-Red} $S$ is reduced and for any maximal point $\eta$ of $S$, the fiber $\cG_\eta$ of $\cG$ over $\eta$ is etale. 
\end{enumerate}
Then the map $\varphi_\cG:\cE_\cG^{(q)}\to \cE_\cG$ is injective. In particular, there exists a unique $v$-structure on $\cG$, and for any $v$-module $\cH$, any morphism $\cG\to \cH$ of $t$-modules over $S$ is compatible with $v$-structures.
\end{lem}

\begin{cor}\label{DMFV}
Let $S$ be a reduced scheme which is flat over $A$ and $E$ a Drinfeld module of rank $r$ over $S$. Let $a\in A$ be a non-zero element and $\cG$ a finite locally free closed $A$-submodule scheme of the $a$-torsion part $E[a]$ of $E$ over $S$ of constant $q$-power rank. 
Then $E/\cG$ has a natural structure of a Drinfeld module of rank $r$. Moreover, $\cG$ has a unique structure of a finite $v$-module induced from that of $E$ and, for any $v$-module $\cH$, any morphism $\cG\to \cH$ of $t$-modules over $S$ is compatible with $v$-structures.
\end{cor}
\begin{proof}
The going-down theorem implies that $a$ is invertible in the residue field of every maximal point $\eta$ of $S$, 
and thus $E[a]$ is etale over $\eta$. Then the first assertion follows from Lemma \ref{DMQFF} (\ref{DMQFF-Dr}). 
Moreover, Lemma \ref{DMQFt} (\ref{DMQFt-Sub}) implies that $\cG$ is a finite $t$-module. Since $\cG$ is the kernel of an isogeny of Drinfeld modules, the $v$-structure on $E$ induces that on $\cG$. The other assertions follow from Lemma \ref{SredV} (\ref{SredV-Red}). 
\end{proof}

\begin{rmk}
The notation here is slightly different from the literature including \cite{Tag}. Finite $\varphi$-sheaves are usually referred to as $\varphi$-sheaves. In \cite{Tag}, finite $t$-modules, finite $v$-modules and finite $v$-sheaves are assumed to be killed by some nonzero element of $A$.
\end{rmk}



\subsection{Duality for finite $v$-modules}

Let $S$ be a scheme over $A$. We denote by $C$ the Carlitz module over $S$, as before. We have
\[
\cE_C=\sHom_{\bF_q,S}(C,\Ga)=\bigoplus_{i\in \bZ_{\geq 0}}\cO_S Z^{q^i}
\]
with its unique $v$-structure given by
\[
v_C: \cE_C\to \cE_C^{(q)},\quad Z^{q^i}\mapsto Z^{q^{i-1}}\otimes (\theta^{q^i}-\theta)+ Z^{q^i}\otimes 1.
\]
Note that $v_C$ is surjective. We have 
\[
\psi^C_{t^i}(Z):=(\psi^C_t)^i(Z)=\theta^{i}Z+\cdots+Z^{q^{i}} 
\]
and thus the set $\{\psi^C_{t^i}(Z)\}_{i\geq 0}$ forms a basis of $\cE_C$. For any scheme $T$ over $S$ and any $v$-module $\cH$ over $T$, we denote by $\Hom_{v,T}(\cH,C|_T)$ the $A$-module of morphisms $\cH\to C|_T$ in the category $\vMod_T$.

The following theorem, due to Taguchi, gives a duality for finite $v$-modules over $S$ which is more suitable to analyze Drinfeld modules and Drinfeld modular forms than usual Cartier duality for finite locally free group schemes.

\begin{thm}[\cite{Tag}, \S4]\label{ThmFVDual}
\begin{enumerate}
\item Let $\cG$ be a finite $v$-module over $S$. Then the big Zariski sheaf
\[
\sHom_{v,S}(\cG,C): (S\text{-schemes})\to (A\text{-modules})
\]
given by $T\mapsto \Hom_{v,T}(\cG|_T,C|_T)$
is represented by a finite $v$-module $\cG^D$ over $S$. We refer to $\cG^D$ as the Taguchi dual of $\cG$.
\item $\rank(\cG)=\rank(\cG^D)$.
\item The functor 
\[
\vMod_S^f\to \vMod^f_S,\quad \cG\mapsto \cG^D
\]
is exact (in the usual sense) and commutes with any base change.
\item There exists a natural isomorphism of $v$-modules $\cG\to (\cG^D)^D$.
\end{enumerate}
\end{thm}
\begin{proof}
For the convenience of the reader, we give a simpler proof than in \cite[\S4]{Tag}. Consider the linear dual
\[
\cE_\cG^\vee=\sHom_{\cO_S}(\cE_\cG,\cO_S)
\]
and the dual maps
\[
(\psi_t^\cG)^\vee: \cE_\cG^\vee\to \cE_\cG^\vee,\quad  v_\cG^\vee: (\cE_\cG^\vee)^{(q)}\to \cE_\cG^\vee,\quad 
\varphi_\cG^\vee:  \cE_\cG^\vee \to (\cE_\cG^\vee)^{(q)}
\]
of $\psi_t^\cG$, $v_\cG$ and $\varphi_\cG$, respectively. We define a finite $v$-module $\cG^D$ over $S$ by $\cG^D=\Gr(\cE_\cG^\vee, v_\cG^\vee)$ with $t$-action $(\psi_t^\cG)^\vee$ and $v$-structure $\varphi_\cG^\vee$.

To see that it represents the functor in the theorem, let $\pi:T\to S$ be any morphism. Since $v_C$ is surjective, to give a map of $v$-sheaves $g:\cE_C|_T\to \cE_{\cG}|_T$ is the same as to give an $\cO_T$-linear map which is compatible with $t$-actions and $v$-structures. Since $\psi^C_{t^{i+1}}(Z)=\psi_t^C(\psi^C_{t^{i}}(Z))$, to give an $\cO_T$-linear map $g:\cE_C|_T\to \cE_{\cG}|_T$ compatible with $t$-actions is the same as to give an element $x=g(Z)$ of $\cE_{\cG}|_T(T)$. As for the compatibility with $v$-structures, we see that if $g$ is compatible with $t$-actions, then the relation $(g\otimes 1)(v_C(\psi^C_{t^{i}}(Z)))=v_\cG(g(\psi^C_{t^{i}}(Z)))$ implies a similar relation for $\psi^C_{t^{i+1}}(Z)=\psi_t^C(\psi^C_{t^{i}}(Z))$. 
Thus we only need to impose on $x$ the condition for $i=0$. Namely, we have
\begin{equation}\label{setGT}
\Hom_{v,T}(\cG|_T,C|_T)=\{x\in \cE_\cG|_T(T) \mid x\otimes 1=v_\cG(x)\},
\end{equation}
where $x\otimes 1 \in (\cE_\cG|_T)^{(q)}(T)$ is the pull-back of $x$ by the Frobenius map $F_T$. On the other hand, by 
(\ref{GrDef}) the set $\Gr(\cE_\cG^\vee,v_\cG^\vee)(T)$ can be identified with the set of $\cO_T$-linear homomorphisms 
$\chi: \cE_\cG^\vee|_T \to\cO_T$ satisfying 
\[
\chi\circ v_\cG^\vee=f_{\cO_T}\circ (F_T^*(\chi)).
\]
Via the natural isomorphisms
\[
\sHom_{\cO_T}(\cE_\cG^\vee|_T, \cO_T)\simeq \cE_\cG|_T,\quad f_{\cO_T}:\cO_T^{(q)}=F_T^*(\cO_T)\simeq \cO_T,
\]
we can easily show that it agrees with (\ref{setGT}). Thus we obtain a natural isomorphism
\[
\Hom_{v,T}(\cG|_T,C|_T)\simeq \Gr(\cE_\cG^\vee,v_\cG^\vee)(T)
\]
and we can check that it is compatible with $A$-actions. The assertion on the exactness follows from the agreement of the exactness and the $\Shv$-exactness for the category $\vMod_S^f$. The other assertions follow from the construction. 
\end{proof}

\begin{lem}\label{ConstDual}
Let $S$ be any scheme over $A$. Let $a\in A$ be any monic polynomial. Consider the finite $t$-module $C[a]$ endowed with the natural $v$-structure as the kernel of the isogeny $a:C\to C$. Then the Taguchi dual $C[a]^D$ of $C[a]$ is isomorphic as a $v$-module to the constant $A$-module scheme $\underline{A/(a)}$ endowed with the unique $v$-structure of Lemma \ref{SredV} (\ref{SredV-Et}).
\end{lem}
\begin{proof}
Let $\iota:C[a]\to C$ be the natural closed immersion. From the definition of the $v$-structure on $C[a]$, it is 
compatible with $v$-structures. Thus we have a morphism of $t$-modules over $S$
\[
\underline{A/(a)}\to C[a]^D=\sHom_{v,S}(C[a],C),\quad 1\mapsto \iota.
\]
We claim that it is a closed immersion. Indeed, by Nakayama's lemma we may assume $S=\Spec(k)$ for some field $k$. Suppose that $b\iota=0$ 
for some $b\in A$. Write as $b=sa+r$ with $s,r\in A$ satisfying $\deg(r)<\deg(a)$. Then we have $\Phi^C_a(Z)\mid 
\Phi^C_r(Z)$, which is a contradiction unless $r=0$. This implies that the kernel of the above morphism is zero and the claim follows. Since both sides have the same rank over 
$S$, it is an isomorphism. Since both sides are 
etale, it is compatible with unique $v$-structures. 
\end{proof}



\subsection{Duality for Drinfeld modules of rank two}

Let $S$ be a scheme over $A$. 
Recall that for any $S$-scheme $T$, both of the categories $\DM_T$ of Drinfeld modules over $T$ and $\vMod_T^f$ of finite $v$-modules over $T$ are full subcategories of $\vMod_T$, and $\vMod_T$ is anti-equivalent to $\vShv_T$.
For any $v$-modules $\cH,\cH'$ over $T$, we denote by $\Ext^1_{v,T}(\cH,\cH')$ the $A$-module of isomorphism classes of Yoneda extensions of $\cH$ by $\cH'$ in the category $\vMod_T$ with $\Shv$-exactness. We identify this $A$-module with the $A$-module $\Ext^1_{v,T}(\cE_{\cH'},\cE_{\cH})$ of isomorphism classes of Yoneda extensions of $\cE_{\cH'}=\Sh(\cH')$ by $\cE_{\cH}=\Sh(\cH)$ in the exact category $\vShv_T$. We also define a big Zariski sheaf $\sExt^1_{v,S}(\cH,\cH')$ as the sheafification of $T\mapsto \Ext^1_{v,T}(\cH|_T,\cH'|_T)$.

Let $E$ be a Drinfeld module over $S$ and put $G=\Ga$ or $C$ over $S$. We write as $G=\Spec_S(\cO_S[Z])$. Let us describe the isomorphism class of any extension
\[
\xymatrix{
0 \ar[r] & G \ar[r] & L \ar[r] & E \ar[r] & 0
}
\]
in the category $\vMod_S$. Consider the associated exact sequence  
\[
\xymatrix{
0 \ar[r] & \cE_E \ar[r] & \cE_L \ar[r] & \cE_{G} \ar[r] & 0
}
\]
in the category $\vShv_S$. 
Since $\cE_{G}$ is a free $\cO_S$-module, this sequence splits as $\cO_S$-modules if $S$ is affine. In this case, using $\varphi_G(Z^{q^i})=Z^{q^{i+1}}$, we can show that there exists a splitting $s:\cE_G\to \cE_L$ of the above sequence which is compatible with $\varphi$-structures. 

We assume that $S$ is affine and fix such a $\varphi$-compatible splitting $s$ for a while. Then the $a$-action on $L$ for any $a\in A$ is given by
\[
\Phi_{a}^L=(\Phi_{a}^G,\Phi_{a}^E+\delta_a)
\]
with some $\bF_q$-linear homomorphism
\[
\delta:A\to \Hom_{\bF_q,S}(E,G),\ a\mapsto \delta_a.
\]
Here $\delta_a$ is associated to the map $\psi^L_a\circ s-s\circ \psi_a^G:\cE_G\to \cE_E$ and satisfies
\begin{enumerate}[label=(\roman{enumi}), ref=(\roman{enumi})]
\item\label{CondBiderFq} $\delta_{\lambda}=0$ for any $\lambda\in \bF_q$,
\item\label{CondBiderProd} $\delta_{ab}=\Phi_{a}^G \circ \delta_b+\delta_a\circ \Phi_{b}^E$ for any $a,b\in A$. 
\end{enumerate}
Since $\varphi_L:\cE_L^{(q)}\to \cE_L$ is injective and by Definition \ref{DfnTMod} the map $\psi^L_{t}-\theta$ kills $\Coker(\varphi_L)$, the $v$-structure on $L$ is uniquely determined by the data $\delta_t$.

\begin{dfn}[\cite{Gek_dR}, \S3 and \cite{PapRam}, \S2]
Let $S$ be an affine scheme over $A$, $E$ a Drinfeld module over $S$ and $G=\Ga$ or $C$ as above.
\begin{enumerate}
\item An $(E,G)$-biderivation is an $\bF_q$-linear homomorphism $\delta:A\to \Hom_{\bF_q,S}(E,G),\ a\mapsto \delta_a$ satisfying the above conditions \ref{CondBiderFq} and \ref{CondBiderProd}.
The module of $(E,G)$-biderivations is denoted by $\Der(E,G)$, which admits two natural $A$-module structures defined by
\[
(\delta*c)_a=\delta_a\circ \Phi_{c}^E,\quad (c*\delta)_a=\Phi_{c}^G\circ \delta_a \text{ for any }c\in A.
\]
Note that we have a natural isomorphism
\begin{equation}\label{EqnEvalDer}
\ev_t: \Der(E,G)\to \Hom_{\bF_q,S}(E,G),\quad \delta\mapsto \delta_t.
\end{equation}
\item An $(E,G)$-biderivation $\delta$ is said to be inner if there exists $f\in \Hom_{\bF_q,S}(E,G)$ satisfying $\delta=\delta_f$, where the $(E,G)$-biderivation $\delta_f$ is defined by
\[
\delta_{f,a}=f\circ \Phi_{a}^E-\Phi_{a}^G\circ f \text{ for any }a\in A.
\] 
The submodule of $\Der(E,G)$ consisting of inner $(E,G)$-biderivations is denoted by $\Derin(E,G)$, which is stable under two natural $A$-actions.

\item We denote by $\Der_0(E,G)$ the submodule of $\Der(E,G)$ consisting of $(E,G)$-biderivations $\delta$ such that the induced map on sheaves of invariant differentials
\[
\Cot(\delta_t):\omega_G\to \omega_E
\]
is the zero map. We have $\Derin(E,G)\subseteq \Der_0(E,G)$.

\item An inner $(E,G)$-biderivation $\delta_f$ is said to be strictly inner if $\Cot(f)=0$. We denote by $\Dersi(E,G)$ the submodule of $\Der(E,G)$ consisting of strictly inner $(E,G)$-biderivations.
\end{enumerate}
\end{dfn}
Then the two natural $A$-actions on $\Der(E,G)$ agree with each other on the quotient $\Der(E,G)/\Derin(E,G)$ \cite[p. 412]{PapRam} and we have natural isomorphisms of $A$-modules
\begin{equation}\label{EqnDefExtDerMap}
\begin{aligned}
\Ext^1_{v,S}(E,G)&\to \Der(E,G)/\Derin(E,G)\\
&\overset{\ev_t}{\to} \Hom_{\bF_q,S}(E,G)/\ev_t(\Derin(E,G)).
\end{aligned}
\end{equation}
We define an $A$-submodule 
\[
\Ext^1_{v,S}(E,G)^0 
\]
of $\Ext^1_{v,S}(E,G)$ as the inverse image of $\Der_0(E,G)/\Derin(E,G)$ by the above isomorphism. Since another choice of a $\varphi$-compatible splitting gives the same biderivation modulo inner ones, the first map of (\ref{EqnDefExtDerMap}) is independent of the choice of a $\varphi$-compatible splitting, and so is the the $A$-submodule $\Ext^1_{v,S}(E,G)^0$.

Suppose that $E=\bV_*(\cL)$ is a Drinfeld module of rank two over the affine scheme $S$. We have a natural isomorphism
\[
\bigoplus_{m\geq 0} \cL^{\otimes -q^m}\to \sHom_{\bF_q,S}(E,G),\quad b\mapsto (Z\mapsto b),
\]
by which we identify both sides.
Then $\Der(E,G)$, $\Der(E,G)^0$ and $\Derin(E,G)$ are locally free $\cO_S(S)$-modules, and we can show that 
\[
T\mapsto \Ext^1_{v,T}(E|_T,G|_T) ,\quad T\mapsto \Ext^1_{v,T}(E|_T,G|_T)^0 
\]
satisfy the axiom of sheaves on affine open subsets of $S$. This implies that, for the case where $S$ is not necessarily affine, we have a subsheaf of $A$-modules
\[
\sExt^1_{v,S}(E,G)^0\subseteq \sExt^1_{v,S}(E,G)
\]
such that, for any affine scheme $T$ over $S$ and $\bullet\in\{\emptyset,0\}$, we have
\[
\sExt^1_{v,S}(E,G)^\bullet(T)=\Ext^1_{v,T}(E|_T,G|_T)^\bullet.
\]
Moreover, we have a natural isomorphism of big Zariski sheaves
\begin{equation}\label{EqnDualBdl}
\cL^{\otimes -q} \to \sExt^1_{v,S}(E,G)^0
\end{equation}
sending, for any affine scheme $T$ over $S$, any element $b\in \cL^{\otimes -q}(T)$ to the unique extension class such 
that, for the associated $(E|_T,G|_T)$-biderivation $\delta$, the map $\delta_t:E|_T\to G|_T$ is given by 
\[
\delta_t^*:\cO_T[Z]\to \Sym(\cL^{\otimes -1}|_T),\quad Z\mapsto b.
\]
Thus, taking $G=C$, we have the following theorem, which is due to Taguchi \cite[\S5]{Tag}. The interpretation of his duality using biderivations obtained here is a generalization of \cite[Theorem 1.1]{PapRam} to general base schemes.

\begin{thm}\label{DualDM}
	Let $S$ be any scheme over $A$.
\begin{enumerate}
\item\label{DualDMRep} Let $E=(\cL,\Phi^E)$ be any Drinfeld module of rank two over $S$ with
\[
\Phi_{t}^E=\theta +a_1 \tau +a_2 \tau^2,\quad a_i\in \cL^{\otimes 1-q^i}(S).
\]
Then the functor
\[
\sExt^1_{v,S}(E,C)^0: (S\text{-schemes})\to (A\text{-modules})
\] 
is represented by a Drinfeld module $E^D$ of rank two over $S$ defined by
\[
E^D=\bV_*(\cL^{\otimes -q}),\quad \Phi_{t}^{E^D}=\theta -a_1 \otimes a_2^{\otimes -1} \tau +a_2^{\otimes -q} \tau^2.
\]
\item\label{DualDMBC} The formation of $E^D$ commutes with any base change.

\item\label{DualDMLin} Let $F=(\cM,\Phi^F)$ be any Drinfeld module of rank two over $S$. Then any morphism $f:E\to F$ of the category $\DM_S$ induces a morphism $f^D:F^D\to E^D$ of this category. If $f$ is induced by an $\cO_S$-linear map $f:\cL\to \cM$, then the dual map $f^D: F^D\to E^D$ is given by the $q$-th tensor power $(f^\vee)^{\otimes q}$ of the linear dual $f^\vee: \cM^\vee\to \cL^\vee$.
\item\label{DualDMIsog} If $f$ is an isogeny, then $f^D$ is also an isogeny of the same degree as $f$, $\Ker(f)$ has a 
natural structure of a finite $v$-module over $S$ and there exists a natural isomorphism of $A$-module schemes over $S$
\[
(\Ker(f))^D\to \Ker(f^D).
\]
\end{enumerate}
\end{thm}
\begin{proof}
The assertions (\ref{DualDMRep}) and (\ref{DualDMBC}) follow easily from the construction. 
The assertion (\ref{DualDMLin}) follows from the functoriality of $\sExt^1_{v,S}(-,C)^0$ and the isomorphism (\ref{EqnDualBdl}).

Let us show the assertion (\ref{DualDMIsog}). Put $\cG=\Ker(f)$. Corollary \ref{DMIFF} implies that the exact sequence of group schemes
\[
\xymatrix{
	0 \ar[r] & \cG\ar[r] &E\ar[r]^{f}& F\ar[r] & 0
}
\]
is also $\Shv$-exact and thus $\cG$ has a natural structure of a finite $v$-module such that this sequence is compatible with $v$-structures.
Since $E$ and $C$ have different ranks, the long exact sequence of $\sHom_{v,S}$ yields an exact sequence
\[
\xymatrix{
0 \ar[r] & \cG^D\ar[r] &\sExt^1_{v,S}(F, C)\ar[r]&\sExt^1_{v,S}(E, C).
}
\]
From a description of the connecting homomorphism using Yoneda extension, we can show that it factors through the subsheaf $\sExt^1_{v,S}(F, C)^0$.
Therefore we have an exact sequence of $A$-module schemes over $S$
\[
\xymatrix{
0 \ar[r] & \cG^D\ar[r] &F^D\ar[r]^{f^D}&E^D,
}
\]
from which we obtain a natural isomorphism $\cG^D\to \Ker(f^D)$. 
To see that $f^D$ is faithfully flat, by a base change we may assume $S=\Spec(k)$ for some field $k$. Then the group schemes $F^D$ and $E^D$ are isomorphic to $\Ga$ and $f^D$ is defined by an additive polynomial. Since $\Ker(f^D)=\cG^D$ is finite over $S$, this polynomial is non-zero and thus $f^D$ is faithfully flat. Since the ranks of $\cG$ and $\cG^D$ are the same, the assertion on $\deg(f^D)$ also follows.
\end{proof}

\begin{rmk}\label{RmkAutodual}
Suppose that there exists a section $h\in \cL^{\otimes -(q+1)}(S)$ satisfying $h^{\otimes q-1}=-a_2$. Then the map $h:\cL\to \cL^{\otimes -q}$ gives an autoduality for Drinfeld modules of rank two. In the classical setting on the Drinfeld upper half plane, this is the case because of the existence of Gekeler's $h$-function \cite[Theorem 9.1 (c)]{Gek_Coeff}. In general, we only have a weaker version of autoduality: the map 
\[
\cL^{\otimes q-1}\to \cL^{\otimes -q(q-1)},\quad l \mapsto l\otimes a_2
\]
is an isomorphism of invertible sheaves. This is enough for our purpose.
\end{rmk}

For a Drinfeld module $E$ over $S$, we have analogues of the first de Rham cohomology group and the Hodge filtration for an abelian variety \cite[\S5]{Gek_dR}. First we show the following lemma.

\begin{lem}\label{DualLieCot}
For any Drinfeld module $E$ of rank two over an affine scheme $S$, we have natural isomorphisms
\[
\Lie(E^D)\to \Ext^1_{v,S}(E,\Ga)^0,\quad\Derin(E,\Ga)/\Dersi(E,\Ga)\to \Lie(E)^\vee.
\]
\end{lem}
\begin{proof}
For the former one, we put $S_\varepsilon=\Spec_S(\cO_S[\varepsilon]/(\varepsilon^2))$. Then we have
\[
\Lie(E^D)=\Ker(E^D(S_\varepsilon)\to E^D(S)).
\]
For any $\bF_q$-linear homomorphism $\delta:A\to \Hom_{\bF_q,S_\varepsilon}(E|_{S_\varepsilon},C|_{S_\varepsilon})$, we can write as
\[
\delta_a=\delta^0_a+\varepsilon \delta^1_a,\quad \delta^i_a\in\Hom_{\bF_q,S}(E,C).
\]
Then $\delta\in \Der_0(E|_{S_\varepsilon},C|_{S_\varepsilon})$ if and only if
\[
\delta^0\in\Der_0(E,C),\quad \delta^1 \in \Der_0(E,\Ga).
\]
On the other hand, for any $g=g^0+\varepsilon g^1\in \Hom_{\bF_q,S_\varepsilon}(E|_{S_\varepsilon},C|_{S_\varepsilon})$, the associated inner biderivation $\delta_g$ is written as
\[
\delta_g=\delta_{g^0}+\varepsilon(g^1\circ \Phi^E-\Phi^{\Ga}\circ g^1).
\]
From this, we see that the map sending $\delta$ to the class of $\delta^1$ gives a natural isomorphism $\Lie(E^D)\to 
\Ext^1_{v,S}(E,\Ga)^0$. The latter one is given by the natural map
\[
\Derin(E,\Ga)\to \Hom_{\cO_S}(\Lie(E),\Lie(\Ga)),\quad \delta_f\mapsto \Lie(f).
\]
\end{proof}

For any Drinfeld module $E$ over an affine scheme $S$, we put
\[
\DR(E,\Ga)=\Der_0(E,\Ga)/\Dersi(E,\Ga).
\]
From the proof of \cite[p. 412]{PapRam}, we see that the two natural $A$-actions on $\Der_0(E,\Ga)$ define the same $A$-action on $\DR(E,\Ga)$.
If $E$ is of rank two, then Lemma \ref{DualLieCot} yields an exact sequence of $A$-modules
\begin{equation}\label{HodgeFil}
\xymatrix{
0 \ar[r] & \Lie(E)^\vee \ar[r] & \DR(E,\Ga) \ar[r] & \Lie(E^D) \ar[r] & 0,
}
\end{equation}
which is functorial on $E$.

Finally, we recall the construction of the Kodaira-Spencer map for a Drinfeld module $E$ over an $A$-scheme $S$ 
\cite[\S 6]{Gek_dR}. We only treat the case where $S=\Spec(B)$ is affine and the underlying invertible sheaf of $E$ is 
trivial. Write as $E=\Spec(B[X])$ so that we identify as $\Hom_{\bF_q,S}(E,\Ga)=B\{\tau\}$. We define an action of 
$D\in \Der_A(B)$ on $B\{\tau\}$ by acting on coefficients. Then, via the isomorphism (\ref{EqnEvalDer}), the derivation 
$D$ induces a map $\nabla_D:\Der_0(E,\Ga)\to \Der_0(E,\Ga)$, which in turn defines
\[
\pi_D: \Lie(E)^\vee\to \DR(E,\Ga)\overset{\nabla_D}{\to}\DR(E,\Ga) \to \Lie(E^D),
\]
where the first and the last arrows are those of (\ref{HodgeFil}). Then the Kodaira-Spencer map for $E$ over $S$ is by definition
\[
\KS: \Der_A(B)\to \Hom_B(\Lie(E)^\vee,\Lie(E^D)),\quad D\mapsto \pi_D.
\]
Hence we also have the dual map 
\[
\KS^\vee: \omega_E\otimes_{\cO_S} \omega_{E^D}\to \Omega^1_{S/A}.
\]



\section{Canonical subgroups of ordinary Drinfeld modules}\label{Sec_Cansub}

Let $\wp$ be a monic irreducible polynomial of degree $d$ in $A=\bF_q[t]$. We denote by $\cO_K$ the complete local ring of $A$ at the prime ideal $(\wp)$, which is a complete discrete valuation ring with uniformizer $\wp$. We consider $\cO_K$ naturally as an $A$-algebra. The fraction field and the residue field of $\cO_K$ are denoted by $K$ and $k(\wp)=\bF_{q^d}$, respectively. We denote by $v_\wp$ the $\wp$-adic (additive) valuation on $K$ normalized as $v_\wp(\wp)=1$. For any $\cO_K$-algebra $B$ and any scheme $X$ over $B$, we put $\bar{B}=B/\wp B$ and $\bar{X}=X\times_B \Spec(\bar{B})$. 

We say an $\cO_K$-algebra $B$ is a $\wp$-adic ring if it is complete with respect to the $\wp$-adic topology. A $\wp$-adic ring $B$ is said to be flat if it is flat over $\cO_K$.



\subsection{Ordinary Drinfeld modules}

Let $\bar{S}$ be an $A$-scheme of characteristic $\wp$. Let $\bar{E}=(\bar{\cL},\Phi^{\bar{E}})$ be a Drinfeld module of rank two over 
$\bar{S}$. By \cite[Proposition 2.7]{Shastry}, we can write as
\begin{equation}\label{EqnWpAlpha}
\Phi^{\bar{E}}_{\wp}=(\alpha_d+\cdots+\alpha_{2d}\tau^d)\tau^d,\quad \alpha_i\in \bar{\cL}^{\otimes 1-q^i}(\bar{S}).
\end{equation}
We put 
\[
F_{d,\bar{E}}=\tau^d:\bar{E}\to \bar{E}^{(q^d)},\quad V_{d,\bar{E}}=\alpha_d+\cdots+\alpha_{2d}\tau^d:  \bar{E}^{(q^d)}\to \bar{E}.
\]
We also denote them by $F_d$ and $V_d$ if no confusion may occur.
We also define a homomorphism $F_d^n: \bar{E}\to \bar{E}^{(q^{dn})}$ by
\[
F_d^{1}=F_d,\quad F_d^{n}=(F_d^{n-1})^{(q^d)} \circ F_d.
\]
We define $V_d^n: \bar{E}^{(q^{dn})} \to \bar{E}$ similarly. They are isogenies of Drinfeld modules satisfying $V_d^n\circ F_d^n=\Phi_{\wp^n}^{\bar{E}}$ and $F_d^n\circ V_d^n=\Phi^{\bar{E}^{(q^{dn})}}_{\wp^n}$ \cite[\S2.8]{Shastry}. 
We also have exact sequences of $A$-module schemes over $\bar{S}$
\begin{gather*}
\xymatrix{
0 \ar[r] & \Ker(F_d^n) \ar[r] & \bar{E}[\wp^n] \ar[r] & \Ker(V_d^n) \ar[r] & 0,
}\\
\xymatrix{
0 \ar[r] & \Ker(F_d) \ar[r] & \Ker(F_d^n) \ar[r] & \Ker(F_d^{n-1})^{(q^d)} \ar[r] & 0,
}\\
\xymatrix{
0 \ar[r] & \Ker(V_d^{n-1})^{(q^d)} \ar[r] & \Ker(V_d^n) \ar[r] & \Ker(V_d) \ar[r] & 0.
}
\end{gather*}

\begin{dfn}
We say $\bar{E}$ is ordinary if $\alpha_d \in \bar{\cL}^{\otimes 1-q^d}(\bar{S})$ of (\ref{EqnWpAlpha}) is nowhere 
vanishing, and supersingular if $\alpha_d=0$.
\end{dfn}
By \cite[Proposition 2.14]{Shastry}, $\bar{E}$ is ordinary if and only if $\Ker(V_d)$ is etale if and only if $\Ker(V_d^n)$ is etale for any $n$.

We need a relation of the isogenies $F_d$ and $V_d$ with duality. For this, we first prove the following lemma.

\begin{lem}\label{CarlitzRed}
	Let $C$ be the Carlitz module over $A$. Then the polynomial $\Phi^C_\wp(Z)$ is a monic Eisenstein polynomial in $\cO_K[Z]$. In particular, we have
	\begin{equation}\label{CarlitzModP}
	\Phi^C_\wp(Z)\equiv Z^{q^d}\bmod \wp.
	\end{equation}
	\end{lem}
\begin{proof}
	Let $L$ be a splitting field of the polynomial $\Phi^C_\wp(Z)$ over $K$. Since the ring $A$ acts on $C[\wp](L)$ transitively, any non-zero root $\beta\in L$ of $\Phi^C_\wp(Z)$ satisfies $v_\wp(\beta)=1/(q^d-1)$ and thus the monic polynomial $\Phi^C_\wp(Z)$ is Eisenstein over $\cO_K$.
	\end{proof}

\begin{lem}\label{DualFV}
\[
F_{d,\bar{E}}^D=V_{d,\bar{E}^D},\quad V_{d,\bar{E}}^D=F_{d,\bar{E}^D}.
\]
\end{lem}
\begin{proof}
	First we prove the former equality. Since $F_{d,\bar{E}^D}$ is an isogeny, it is enough to show 
	$F_{d,\bar{E}}^D\circ F_{d,\bar{E}^D}=\Phi^{\bar{E}^D}_\wp$. Let $\bar{\cL}$ be the underlying invertible sheaf of 
	$\bar{E}$. Take any section $l$ of $\bar{\cL}^{\otimes -q}$. We have $F_{d,\bar{E}^D}(l)=l^{\otimes q^d}$. From 
	(\ref{EqnVectFrob}), we see that the map $F_{d,\bar{E}}^D$ sends it to the class of the biderivation $\delta$ 
	such that $\delta_t$ agrees with the homomorphism
	\[
	\bar{E}\to C=\Spec_{\bar{S}}(\cO_{\bar{S}}[Z]),\quad Z\mapsto l^{\otimes q^d} \in 
	\Sym(\bar{\cL}^{\otimes -1}).
	\]
	By (\ref{CarlitzModP}), this is equal to the class of
	$\wp\cdot (Z\mapsto l)$
	with respect to the $A$-module structure of $\sExt^1_{v,\bar{S}}(\bar{E},C)^0$. Since $l$ is a 
	section of $\bar{\cL}^{\otimes -q}$, the isomorphism (\ref{EqnDualBdl}) implies the 
	assertion. 
	
	For the latter equality, it is enough to show 
	$V_{d,\bar{E}}^D\circ V_{d,\bar{E}^D}=\Phi^{(\bar{E}^D)^{(q^d)}}_\wp$. By the former equality of the lemma, we have
	\[
	V_{d,\bar{E}}^D\circ V_{d,\bar{E}^D}=V_{d,\bar{E}}^D\circ F_{d,\bar{E}}^D=(F_{d,\bar{E}}\circ V_{d,\bar{E}})^D=(\Phi^{\bar{E}^{(q^d)}}_\wp)^D.
	\]
	By the definition of the $A$-module structure on $\sExt^1_{v,\bar{S}}(\bar{E}^{(q^d)},C)$, it is equal to 
	$\Phi_\wp^{(\bar{E}^D)^{(q^d)}}$ and we obtain the latter equality of the lemma.
\end{proof}

\begin{prop}\label{DualOrd}
Let $\bar{S}$ be an $A$-scheme of characteristic $\wp$ and $\bar{E}$ a Drinfeld module of rank two over $\bar{S}$. Consider the maps
\[
\Lie(V_{d,\bar{E}}): \Lie(\bar{E}^{(q^d)})\to \Lie(\bar{E}),\quad \Lie(V_{d,\bar{E}^D}): \Lie((\bar{E}^{D})^{(q^d)})\to \Lie(\bar{E}^D)
\]
and the linear dual $\Lie(V_{d,\bar{E}})^\vee$ of the former map. Then we have a natural isomorphism of $\cO_{\bar{S}}$-modules 
\[
\Coker(\Lie(V_{d,\bar{E}})^\vee)\simeq \Coker(\Lie(V_{d,\bar{E}^D})).
\]
In particular, $\bar{E}$ is ordinary if and only if $\bar{E}^D$ is ordinary.
\end{prop}
\begin{proof}
We follow the proof of \cite[Theorem 2.3.6]{Con}.
By gluing, we may assume that $\bar{S}$ is affine. By the exact sequence (\ref{HodgeFil}), we have a commutative diagram of $A$-modules
\[
\xymatrix{
0 \ar[r] & \Lie(\bar{E}^{(q^d)})^\vee \ar[r]\ar[d]_{\Lie(F_{d,\bar{E}})^\vee} & \DR(\bar{E}^{(q^d)},\Ga)\ar[r]\ar[d]^{F_{d,\bar{E}}^*} & \Lie((\bar{E}^{D})^{(q^d)}) \ar[r]\ar[d]^{\Lie(F^D_{d,\bar{E}})} & 0 \\
0 \ar[r] & \Lie(\bar{E})^\vee \ar[r]\ar[d]_{\Lie(V_{d,\bar{E}})^\vee} & \DR(\bar{E},\Ga)\ar[r]\ar[d]^{V_{d,\bar{E}}^*}  & \Lie(\bar{E}^{D}) \ar[r]\ar[d]^{\Lie(V_{d,\bar{E}}^D)}  & 0 \\
0 \ar[r] & \Lie(\bar{E}^{(q^d)})^\vee \ar[r] & \DR(\bar{E}^{(q^d)},\Ga)\ar[r] & \Lie((\bar{E}^{D})^{(q^d)}) \ar[r] & 0,
}
\]
where rows are exact and columns are complexes. Since $\Lie(F_{d,\bar{E}})=\Lie(F_{d,\bar{E}^D})=0$, Lemma \ref{DualFV} implies that the middle column of the diagram induces the complex
\[
\xymatrix{
0 \ar[r] & \Lie((\bar{E}^{D})^{(q^d)}) \ar[r]^{F_{d,\bar{E}}^*} &  \DR(\bar{E},\Ga) \ar[r]^{V_{d,\bar{E}}^*} & \Lie(\bar{E}^{(q^d)})^\vee  \ar[r] & 0.
}
\]
If it is exact, then as in the proof of \cite[Theorem 2.3.6]{Con}, by using \cite[Lemma 2.3.7]{Con} and Lemma \ref{DualFV} we obtain
\[
\Coker(\Lie(V_{d,\bar{E}})^\vee)\simeq \Coker(\Lie(F^D_{d,\bar{E}}))=\Coker(\Lie(V_{d,\bar{E}^D})).
\]

Let us show the exactness. Since it is a complex of locally free $\cO_{\bar{S}}$-modules of finite rank and its formation commutes with any base change of affine schemes, we may assume $\bar{S}=\Spec(k)$ for some field $k$. By comparing dimensions, it is enough to show that, for any Drinfeld module $\bar{E}$ of rank two over $k$, the maps
\begin{align*}
F_{d,\bar{E}}^*&: \Der_0(\bar{E}^{(q^d)},\Ga)/\Derin(\bar{E}^{(q^d)},\Ga) \to \Der_0(\bar{E},\Ga)/\Dersi(\bar{E},\Ga) \\
V_{d,\bar{E}}^*&: \Der_0(\bar{E},\Ga)/\Dersi(\bar{E},\Ga) \to \Derin(\bar{E}^{(q^d)},\Ga)/\Dersi(\bar{E}^{(q^d)},\Ga)
\end{align*}
are non-zero.

For the assertion on $F_{d,\bar{E}}^*$, we write as
\[
\Phi^{\bar{E}}_{t}=\theta+a_1 \tau+a_2 \tau^2, \quad \Phi^{\bar{E}}_{\wp}=(\alpha_d+\cdots+\alpha_{2 d}\tau^d)\tau^d
\]
with $a_2,\alpha_{2 d}\neq 0$.
Let $\delta$ be the element of $\Der_0(\bar{E}^{(q^d)},\Ga)$ satisfying $\delta_t=\tau$ and suppose that $F_{d,\bar{E}}^*(\delta)$ is an element of $\Dersi(\bar{E},\Ga)$. Namely, we have
\begin{equation}\label{Td1}
\tau^{d+1}=f\circ\Phi^{\bar{E}}_{t}-\Phi^{\Ga}_{t}\circ f
\end{equation}
for some $f\in \Hom_{\bF_q,k}(\bar{E},\Ga)$ satisfying $\Cot(f)=0$. We write $f$ as $f=b_r\tau^r+\cdots+b_s\tau^s$ with some $b_i\in k$ and $1\leq r\leq s$ satisfying $b_r,b_s\neq 0$. Then we have $s=d-1$ and the coefficient of $\tau^r$ in the right-hand side of (\ref{Td1}) is $(\theta^{q^r}-\theta)b_r$. Since $1\leq r\leq d-1$ and the element $\theta$ generates $k(\wp)=\bF_{q^d}$ over $\bF_q$, this term does not vanish and thus we have $r=d+1$, which is a contradiction.

Let us consider the assertion on $V_{d,\bar{E}}^*$. If $\alpha_d\neq 0$, then the map $\Lie(V_{d,\bar{E}})$ is an isomorphism and the claim follows from the above diagram. Otherwise, \cite[Lemma 2.5]{Shastry} yields $\alpha_i=0$ unless $i=2 d$. Let $\delta$ be the element of $\Der_0(\bar{E},\Ga)$ satisfying $\delta_t=\tau$ and suppose that $V_{d,\bar{E}}^*(\delta)$ is an element of $\Dersi(\bar{E}^{(q^d)},\Ga)$. We have
\[
\tau(\alpha_{2 d}\tau^d)=g\circ\Phi^{\bar{E}^{(q^d)}}_{t}-\Phi^{\Ga}_{t}\circ g
\]
for some $g\in \Hom_{\bF_q,k}(\bar{E}^{(q^d)},\Ga)$ satisfying $\Cot(g)=0$. Then we obtain a contradiction as in the above case.
\end{proof}




\subsection{Canonical subgroups}\label{SubsecCanSub}

Let $B$ be an $\okey$-algebra and $E$ a Drinfeld module of rank two over $B$. We say $E$ has ordinary reduction if $\bar{E}=E\times_B\Spec(\bar{B})$ is ordinary.

\begin{lem}\label{ExistCanSub}
	Let $B$ be a $\wp$-adic ring and $E$ a Drinfeld module of rank two over $B$ with ordinary reduction. Then, for any positive integer $n$, there exists a unique finite locally free closed $A$-submodule scheme $\cC_n(E)$ of $E[\wp^n]$ over $B$ satisfying $\overline{\cC_n(E)}=\Ker(F_{d,\bar{E}}^n)$. The formation of $\cC_n(E)$ commutes with any base change of $\wp$-adic rings. We refer to it as the canonical subgroup of level $n$ of the Drinfeld module $E$ with ordinary reduction.
	\end{lem}
\begin{proof}
	First note that, since $(B,\wp B)$ is a Henselian pair, the functor $X\mapsto \bar{X}$ gives an equivalence between the categories of finite etale schemes over $B$ and those over $\bar{B}$ \cite[\S1]{Gabber}. 
	
	Let us show the existence. The finite etale $A$-module scheme $\bar{\cH}=\Ker(V_{d,\bar{E}}^n)$ can be lifted to a finite etale $A$-module scheme $\cH$ over $B$. By the etaleness and \cite[Proposition (17.7.10)]{EGA4-4}, we can lift the map $\bar{E}[\wp^n]\to \bar{\cH}$ to a finite locally free morphism of $A$-module schemes $\pi: E[\wp^n]\to \cH$ over $B$. 
	Then $\cC_n(E)=\Ker(\pi)$ is a lift of $\Ker(F_{d,\bar{E}}^n)$.
	
	For the uniqueness, suppose that we have two subgroup schemes $\cC_{n,1},\cC_{n,2}$ of $E[\wp^n]$ as in the lemma. Put $\cH_i=E[\wp^n]/\cC_{n,i}$. Since they are lifts of $\bar{\cH}$, there exists an isomorphism $\theta: \cH_1\to \cH_2$ over $B$ reducing to $\id_{\bar{\cH}}$ over $\bar{B}$. Then the etaleness implies that $\theta$ is compatible with the quotient maps $E[\wp^n]\to \cH_i$. Therefore, $\cC_{n,1}$ and $\cC_{n,2}$ agree as $A$-submodule schemes of $E[\wp^n]$. Since the formation of $\Ker(F_{d,\bar{E}}^n)$ commutes with any base change, the commutativity of $\cC_n(E)$ with any base change follows from its uniqueness.
	\end{proof}

We refer to the natural isogeny
\[
\pi_{E,n}: E\to E/\cC_n(E)
\]
as the canonical isogeny of level $n$ for $E$. We have $\pi_{E,n}\bmod \wp=F_d^n$. 

On the other hand, since $E[\wp^n]/\cC_n(E)$ is etale both over $\bar{B}$ and $B\otimes_{\cO_K} K$, it is etale over $B$ and we have a natural isomorphism 
\[
\omega_{E[\wp^n]} \to \omega_{\cC_n(E)}.
\] 
Moreover, the map 
\[
\rho_{E,n}: E/\cC_n(E)\to (E/\cC_n(E))/(E[\wp^n]/\cC_n(E))\overset{\wp^n}{\simeq} E
\]
is an etale isogeny satisfying 
\[
\rho_{E,n}\circ \pi_{E,n}=\Phi^E_{\wp^n},\quad \pi_{E,n}\circ \rho_{E,n}=\Phi^{E/\cC_n(E)}_{\wp^n}.
\]
In particular, we have $\rho_{E,n}\bmod \wp=V_d^n$. We refer to $\rho_{E,n}$ as the canonical etale isogeny of level $n$ for $E$. The formation of $\pi_{E,n}$ and $\rho_{E,n}$ also commutes with any base change of $\wp$-adic rings.

Suppose that the $\wp$-adic ring $B$ is reduced and flat. Then by Corollary \ref{DMFV} the quotient $E/\cC_n(E)$ has a natural structure of a Drinfeld module of rank two. Moreover, Lemma \ref{DMQFt} implies that $E[\wp^n]$, $\cC_n(E)$ and $E[\wp^n]/\cC_n(E)$ are finite $t$-modules, and by Lemma \ref{SredV} (\ref{SredV-Red}) they have unique structures of finite $v$-modules, which make the natural exact sequence
\begin{equation}\label{EqnExactConnEt}
\xymatrix{
0 \ar[r] & \cC_n(E) \ar[r] & E[\wp^n] \ar[r]& E[\wp^n]/\cC_n(E) \ar[r] & 0
}
\end{equation}
compatible with $v$-structures. We also see that the formation of the $v$-structure on $\cC_n(E)$ also commutes with any base change of reduced flat $\wp$-adic rings.

\begin{lem}\label{DualEtale}
Let $B$ be a reduced flat $\wp$-adic ring. Let $E$ be a Drinfeld module of rank two over $B$ with ordinary reduction. Then the Taguchi dual $\cC_n(E)^D$ of the canonical subgroup $\cC_n(E)$ is etale over $B$. Moreover, it is etale locally isomorphic as a finite $v$-module to the constant $A$-module scheme $\underline{A/(\wp^n)}$ over $B$.
\end{lem}
\begin{proof}
By Proposition \ref{DualOrd}, the dual $E^D$ also has ordinary reduction. 
We claim that $E^D[\wp^n]/\cC_n(E^D)$ is not killed by $\wp^{n-1}$. Indeed, if it is killed by $\wp^{n-1}$, then we have $E^D[\wp]\subseteq \cC_n(E^D)$, which contradicts the fact that $\bar{E}^D[\wp]$ has an etale quotient. Since $E^D[\wp^n]/\cC_n(E^D)$ is etale, the claim implies that it is etale locally isomorphic to $\underline{A/(\wp^n)}$. Note that this identification is compatible with $v$-structures by Lemma \ref{SredV} (\ref{SredV-Et}).

Since Taguchi duality is exact, the exact sequence (\ref{EqnExactConnEt}) for $E^D$ yields an exact sequence of finite $v$-modules over $B$
\[
\xymatrix{
0 \ar[r] & (E^D[\wp^n]/\cC_n(E^D))^D \ar[r] & E^D[\wp^n]^D \ar[r]& \cC_n(E^D)^D \ar[r] & 0.
}
\]
By Theorem \ref{DualDM} (\ref{DualDMIsog}), we also have a natural isomorphism of $A$-module schemes $E[\wp^n]\simeq E^D[\wp^n]^D$, by which we identify both sides. Hence we reduce ourselves to showing the equality
\[
\cC_n(E)=(E^D[\wp^n]/\cC_n(E^D))^D.
\]
For this, by the uniqueness of the canonical subgroup it is enough to show that the reduction of $(E^D[\wp^n]/\cC_n(E^D))^D$ is killed by $F_d^n$. Since it can be checked after passing to a finite etale cover of $\Spec(B)$, we reduce ourselves to showing that the Taguchi dual $(\underline{A/(\wp^n)})^D$ of the constant $A$-module scheme $\underline{A/(\wp^n)}$ over $\bar{B}$ is killed by $F_d^n$. This follows from Lemma \ref{ConstDual} and (\ref{CarlitzModP}).
\end{proof}



\subsection{Hodge-Tate-Taguchi maps}

For any positive integer $n$, any $A$-algebra $B$ and any scheme $X$ over $A$, we put $B_n=B/(\wp^n)$ and $X_n=X\times_A \Spec(A_n)$. We identify a quasi-coherent module on the big fppf site of $X$ with a quasi-coherent $\cO_X$-module by descent.

Let $S$ be a scheme over $A$ and $\cG$ a finite $v$-module over $S$. For any scheme $T$ over $S$, Taguchi duality gives a natural homomorphism of $A$-modules
\begin{align*}
\cG^D(T)\simeq \Hom_{v,T}(\cG|_T, C|_T)&\to \omega_{\cG|_T}(T)\\
(g:\cG|_T\to C|_T)&\mapsto g^*(dZ),
\end{align*}
which defines a natural homomorphism of big fppf sheaves of $A$-modules over $S$
\[
\HTT_{\cG}: \cG^D\to \omega_{\cG}.
\]
We refer to it as the Hodge-Tate-Taguchi map for the finite $v$-module $\cG$ over $S$, and also denote it by $\HTT$ if no confusion may occur. The formation of the Hodge-Tate-Taguchi map commutes with any base change.

Suppose that the $A$-module scheme $\cG$ is killed by $\wp^n$. Then the Hodge-Tate-Taguchi map defines a natural $A$-linear homomorphism of big fppf sheaves on $S_n$
\[
\HTT: \cG^D|_{S_n}\otimes_{\underline{A_n}}\cO_{S_n}\to \omega_{\cG_n}.
\]
Note that, if $\cG^D$ is etale locally isomorphic to the constant $A$-module scheme $\underline{A_n}$ over $S$, then the $\cO_{S_n}$-module $\cG^D|_{S_n}\otimes_{\underline{A_n}}\cO_{S_n}$ is invertible. By Lemma \ref{DualEtale}, this is the case if $\cG=\cC_n(E)$ for any Drinfeld module $E$ of rank two over a reduced flat $\wp$-adic ring $B$ with ordinary reduction.

\begin{lem}\label{HTTCarlitz}
Let $S$ be any scheme over $A$. We give the finite $t$-module $C[\wp^n]$ the $v$-structure induced from that of $C$. Then the Hodge-Tate-Taguchi map for $C[\wp^n]$
\[
\HTT: \underline{A_n}\otimes_{\underline{A_n}} \cO_{S_n} \simeq C[\wp^n]^D|_{S_n}\otimes_{\underline{A_n}} \cO_{S_n}\to \omega_{C[\wp^n]_n}=\cO_{S_n}d Z
\]
is an isomorphism satisfying $\HTT(1)=d Z$.
\end{lem}
\begin{proof}
Let $\iota:C[\wp^n]\to C$ be the natural closed immersion, as in the proof of Lemma \ref{ConstDual}. The definition of the Hodge-Tate-Taguchi map gives $\HTT(1)=\iota^*(d Z)$, which yields the lemma.
\end{proof}

\begin{prop}\label{HTTIsom}
Let $B$ be a reduced flat $\wp$-adic ring. Let $E$ be a Drinfeld module of rank two over $B$ with ordinary reduction. Then the Hodge-Tate-Taguchi map
\[
\HTT: \cC_n(E)^D|_{B_n}\otimes_{\underline{A_n}}\cO_{\Spec(B_n)} \to \omega_{\cC_n(E)}\otimes_B B_n=\omega_{E}\otimes_B B_n
\]
is an isomorphism of invertible sheaves over $B_n$.
\end{prop}
\begin{proof}
It is enough to show that $\HTT$ is an isomorphism after passing to a finite etale cover $\Spec(B')$ of $\Spec(B)$. We may assume that the $A$-module scheme $\cC_n(E)^D|_{B'}=(\cC_n(E)|_{B'})^D$ over $B'$ is constant. In this case, the proposition follows from Lemma \ref{HTTCarlitz}.
\end{proof}



\section{Drinfeld modular curves and Tate-Drinfeld modules}\label{Sec_DMC}

\subsection{Drinfeld modular curves}

Let $\frn$ be a non-constant monic polynomial in $A=\bF_q[t]$ which is prime to $\wp$. Put $A_\frn=A[1/\frn]$. For any Drinfeld module $E$ of rank two over an $A$-scheme $S$ and a non-constant monic polynomial $\frem\in A$, a $\Gamma(\frem)$-structure on $E$ is an $A$-linear homomorphism $\alpha:(A/(\frem))^2\to E(S)$ inducing the equality of effective Cartier divisors of $E$
\[
\sum_{a\in (A/(\frem))^2}[\alpha(a)]=E[\frem].
\]
If $\frem$ is invertible in $S$, then it is the same as an isomorphism of $A$-module schemes $\alpha: \underline{(A/(\frem))^2}\to E[\frem]$ over $S$. If $\frem$ has at least two different prime factors, then the functor over $A$ sending $S$ to the set of isomorphism classes of such pairs $(E,\alpha)$ over $S$ is represented by a regular affine scheme $Y(\frem)$ of dimension two which is flat and of finite type over $A$. Over $A[1/\frem]$, this functor is always representable by an affine scheme $Y(\frem)$ which is smooth of relative dimension one over $A[1/\frem]$. The natural left action of $\mathit{GL}_2(A/(\frem))$ on $(A/(\frem))^2$ induces a right action of this group on $Y(\frem)$.

For any Drinfeld module $E$ of rank two over an $A_\frn$-scheme $S$,
we define a $\Gamma_1(\frn)$-structure on $E$ as a closed immersion of $A$-module schemes $\lambda: C[\frn]\to E$ over $S$. Since $C[\frn]$ is etale over $S$, we see that over a finite etale cover of $A_\frn$ a $\Gamma_1(\frn)$-structure on $E$ is identified with a closed immersion of $A$-module schemes $\underline{A/(\frn)}\to E$. Then \cite[Proposition 4.2 (2)]{Flicker} implies that $E$ has no non-trivial automorphism fixing $\lambda$. Note that the quotient $E[\frn]/\Img(\lambda)$ is a finite etale $A$-module scheme over $S$ which is etale locally isomorphic to $\underline{A/(\frn)}$, and thus the functor 
\[
\sIsom_{A,S}(\underline{A/(\frn)},E[\frn]/\Img(\lambda))
\]
is represented by a finite etale $(A/(\frn))^\times$-torsor $I_{(E,\lambda)}$ over $S$.

Consider the functor over $A_\frn$ sending an $A_\frn$-scheme $S$ to the set of isomorphism classes $[(E,\lambda)]$ of pairs $(E,\lambda)$ consisting of a Drinfeld module $E$ of rank two over $S$ and a $\Gamma_1(\frn)$-structure $\lambda$ on $E$. Then we can show that this functor is representable by an affine scheme $Y_1(\frn)$ which is smooth over $A_\frn$ of relative dimension one. 

Suppose that there exists a prime factor $\frq$ of $\frn$ such that its residue extension $k(\frq)/\bF_q$ is of degree 
prime to $q-1$. In this case, the inclusion $\bF_q^\times\to k(\frq)^\times$ splits and we can choose a subgroup 
$\Delta\subseteq (A/(\frn))^\times$ such that the natural map $\Delta\to (A/(\frn))^\times/\bF_q^\times$ is an 
isomorphism. For such $\Delta$, we define a $\Gamma_1^\Delta(\frn)$-structure on $E$ as a pair 
$(\lambda,[\mu])$ of a $\Gamma_1(\frn)$-structure $\lambda$ on $E$ and an element $[\mu]\in 
(I_{(E,\lambda)}/\Delta)(S)$. We have a fine moduli scheme $Y_1^\Delta(\frn)$ of the isomorphism classes of triples 
$(E,\lambda,[\mu])$, which is finite etale over $Y_1(\frn)$. The universal Drinfeld module over $Y_1^\Delta(\frn)$ is 
denoted by $E_\univ^\Delta=\bV_*(\cL_\univ^\Delta)$ and put
\[
\omega^\Delta_\univ:=\omega_{E^\Delta_\univ}=(\cL^\Delta_\univ)^\vee.
\]

For any Drinfeld module $E$ over an $A_\frn$-scheme $S$, a $\Gamma_0(\wp)$-structure on $E$ is a finite locally free closed $A$-submodule scheme $\cG$ of $E[\wp]$ of rank $q^d$ over $S$. Then we have a fine moduli scheme $Y_1^\Delta(\frn,\wp)$ classifying tuples $(E,\lambda,[\mu],\cG)$ consisting of a Drinfeld module $E$ of rank two over an $A_\frn$-scheme $S$, a $\Gamma_1^\Delta(\frn)$-structure $(\lambda,[\mu])$ and a $\Gamma_0(\wp)$-structure $\cG$ on $E$. From the theory of Hilbert schemes, we see that the natural map $Y_1^\Delta(\frn,\wp)\to Y_1^\Delta(\frn)$ is finite, and it is also etale over $A_\frn[1/\wp]$. For any $A_\frn$-algebra $R$, we write as $Y_1^\Delta(\frn)_R=Y_1^\Delta(\frn)\times_{A_\frn}\Spec(R)$ and similarly for other Drinfeld modular curves.

\begin{lem}\label{NormalNearInfty}
	$Y_1^\Delta(\frn,\wp)$ is smooth over $A_\frn$ outside finitely many supersingular points on the fiber over $(\wp)$.
	\end{lem}
\begin{proof}
	Let $B$ be an Artinian local $A_\frn$-algebra of characteristic $\wp$ and $J$ an ideal of $B$ satisfying $J^2=0$. Let $E$ be an ordinary Drinfeld module of rank two over $B/J$ and $\cG$ a $\Gamma_0(\wp)$-structure on $E$. Since $B$ is local, the underlying invertible sheaf of $E$ is trivial. It is enough to show that the isomorphism class of the pair $(E,\cG)$ lifts to $B$. 
	
	Since $E$ is ordinary and $B/J$ is Artinian local, we have either $\cG=\Ker(F_{d,E})$ or the composite $\cG\to E[\wp]\to \Ker(V_{d,E})$ is an isomorphism. In the former case, write as $\Phi^E_t=\theta+a_1\tau+a_2\tau^2$. For any lift $\hat{a}_i\in B$ of $a_i$, we can define a structure of a Drinfeld module of rank two over $B$ on $\hat{E}=\Spec(B[X])$ by putting $\Phi^{\hat{E}}_t=\theta+\hat{a}_1\tau+\hat{a}_2\tau^2$, which is also ordinary. Then $\cG$ lifts to $\Ker(F_{d,\hat{E}})$. In the latter case $\cG$ is etale and, by Lemma \ref{DMQFF} (\ref{DMQFF-Dr}), $E/\cG$ has a structure of a Drinfeld module of rank two. Moreover, it is also ordinary since $(E/\cG)[\wp]$ has the etale quotient $\cG$. Thus we have isomorphisms
	\[
	\xymatrix{
	(E/\cG)^{(q^d)} & (E/\cG)/\Ker(F_{d,E})\ar[r]_-{\sim}^-{\wp}\ar[l]^-{\sim}_-{F_{d,E/\cG}} & E	
	}
	\]
	sending $\Ker(V_{d,E/\cG})$ to $\cG$. Since the above argument shows that $E/\cG$ also lifts to an ordinary Drinfeld module $\hat{F}$ of rank two over $B$, the pair $(E,\cG)$ lifts to the pair $(\hat{F}^{(q^d)},\Ker(V_{d,\hat{F}}))$ over $B$.
	\end{proof}

Let $K_\infty$ be the completion of $\bF_q(t)$ with respect to the $(1/t)$-adic valuation and $\bC_\infty$ the $(1/t)$-adic completion of an algebraic closure of $K_\infty$. Let $\bA_f$ be the ring of finite adeles and $\hat{A}$ its subring of elements which are integral at all finite places. Let $\Omega$ be the Drinfeld upper half plane over $\bC_\infty$. Put
\begin{align*}
K_1^\Delta(\frn)&=\left\{ g\in\mathit{GL}_2(\hat{A})\ \middle |\ g\bmod \frn\hat{A} \in \begin{pmatrix}\Delta&A/(\frn) \\ 0 & 1\end{pmatrix}\right\},\\
\Gamma(\frn)&=\left\{ g\in\mathit{GL}_2(A)\ \middle |\ g\bmod (\frn) \in \begin{pmatrix}1&0 \\ 0 & 1\end{pmatrix}\right\}
\end{align*}
and $\Gamma_1^\Delta(\frn)=\mathit{GL}_2(A)\cap K_1^\Delta(\frn)$. Since $A^\times=\bF_q^\times$, we have $\Gamma_1^\Delta(\frn)\subseteq \mathit{SL}_2(A)$. This yields
\[
\Gamma_1^\Delta(\frn)=\left\{ g\in\mathit{SL}_2(A)\ \middle |\ g\bmod (\frn) \in \begin{pmatrix}1&A/(\frn) \\ 0 & 1\end{pmatrix}\right\}.
\]
In particular, the group $\Gamma_1^\Delta(\frn)$ is independent of the choice of $\Delta$.
Note that the natural right action of $g\in \mathit{GL}_2(A/(\frn))$ on $Y(\frn)_{\bC_\infty}$ corresponds to the left action of ${}^t\! g$ on $\Gamma(\frn)\backslash \Omega$ via the M\"{o}bius transformation.
Since $\bF_q^\times\det(K_1^\Delta(\frn))=\hat{A}^\times$, \cite[Proposition 6.6]{Dri} implies that the analytification of $Y_1^\Delta(\frn)_{\bC_\infty}$ is identified with
\[
\mathit{GL}_2(\bF_q(t))\backslash \Omega\times \mathit{GL}_2(\bA_f)/K_1^\Delta(\frn)=\Gamma^\Delta_1(\frn)\backslash \Omega,
\]
and thus the fiber $Y_1^\Delta(\frn)_{K_\infty}$ is geometrically connected. Similarly, we see that $Y_1^\Delta(\frn,\wp)_{K_\infty}$ is also geometrically connected.

For any Drinfeld module $E$ of rank two over $S$, we write the $t$-multiplication map of $E$ as $\Phi^{E}_t=\theta+ a_1\tau+a_2\tau^2$ and put 
\[
j_t(E)=a_1^{\otimes q+1}\otimes a_2^{\otimes -1}\in \cO_S(S). 
\]
Consider the finite flat map
\[
j_t: Y_1^\Delta(\frn) \to \bA^1_{A_\frn}=\Spec(A_\frn[j]), \quad j\mapsto j_t(E_\univ^\Delta)
\]
and a similar finite map for $Y_1^\Delta(\frn,\wp)$.
We define the compactifications $X_1^\Delta(\frn)$ and $X_1^\Delta(\frn,\wp)$ of $Y_1^\Delta(\frn)$ and $Y_1^\Delta(\frn,\wp)$ as the normalizations of $\bP^1_{A_\frn}$ in $Y_1^\Delta(\frn)$ and $Y_1^\Delta(\frn,\wp)$ via this map, respectively. As in \cite[\S7.2]{Shastry}, we see that $X_1^\Delta(\frn)$ is smooth over $A_\frn$ and $X_1^\Delta(\frn,\wp)$ is smooth over $A_\frn[1/\wp]$. By a similar argument to the proof of \cite[Corollary 10.9.2]{KM}, Zariski's connectedness theorem implies that each fiber of the map $X^\Delta_1(\frn)\to \Spec(A_\frn)$ is geometrically connected, and so is $X^\Delta_1(\frn,\wp)\to \Spec(A_\frn[1/\wp])$. For any $A_\frn$-algebra $R$ which is Noetherian, excellent and regular, we also have the compactifications $X_1^\Delta(\frn)_R$ and $X_1^\Delta(\frn,\wp)_R$ of $Y_1^\Delta(\frn)_R$ and $Y_1^\Delta(\frn,\wp)_R$. From the smoothness of $X_1^\Delta(\frn)$, we have $X_1^\Delta(\frn)_R=X_1^\Delta(\frn)\times_{A_\frn}\Spec(R)$. The base change compatibility also holds for $X_1^\Delta(\frn,\wp)_R$ if $\wp$ is invertible in $R$.

On the other hand, the maps
\[
[(E,\lambda,[\mu])]\mapsto [(E,a\lambda,[\mu])],\quad [(E,\lambda,[\mu])]\mapsto [(E,\lambda,c[\mu])]
\]
induce actions of the groups $(A/(\frn))^\times$ and $(A/(\frn))^\times/\Delta=\bF_q^\times$ on $X_1^\Delta(\frn)_R$. We denote them by $\langle a \rangle_{\frn}$ and $\langle c \rangle_{\Delta}$, respectively.


\begin{lem}\label{DMAut}
Let $S$ be a scheme over $A$ and $E$ a Drinfeld module of rank two over $S$. If $j_t(E)\in \cO_S(S)$ is invertible, 
then for the big fppf sheaf $\sAut_{A,S}(E)$ defined by
\[
T\mapsto \Aut_{A,T}(E|_T),
\]
the natural map $\underline{\bF_q^\times}\to \sAut_{A,S}(E)$ is an isomorphism.
\end{lem}
\begin{proof}
We may assume that $S=\Spec(B)$ is affine and the underlying invertible sheaf of $E$ is trivial. By \cite[Proposition 4.2 (2)]{Flicker}, any automorphism of $E=\Spec(B[X])$ is linear, namely it is given by $X\mapsto b X$ for some $b\in B^\times$. Write as $\Phi^E_t=\theta+ a_1 \tau +a_2\tau^2$.
From the assumption, we have $a_1\in B^\times$ and the equality $\Phi^E_t(bX)=b\Phi_t^E(X)$ yields $b^{q-1}=1$. Since the group scheme $\mu_{q-1}$ over $\bF_q$ is isomorphic to the constant group scheme $\underline{\bF_q^\times}$, so is $\mu_{q-1}|_B$ over the $\bF_q$-algebra $B$. This concludes the proof.
\end{proof}

\begin{lem}\label{DMIsog}
Let $S$ be a scheme over $A$. Let $E$ and $E'$ be Drinfeld modules of rank two over $S$ satisfying $j_t(E)=j_t(E')\in 
\cO_S(S)^\times$. Then the big fppf sheaf $\sIsom_{A,S}(E,E')$ over $S$ defined by
\[
T\mapsto \Isom_{A,T}(E|_T, E'|_T)
\]
is represented by a Galois covering of $S$ with Galois group $\bF_q^\times$.
\end{lem}
\begin{proof}
By gluing, we reduce ourselves to the case where $S=\Spec(B)$ is affine and the underlying line bundles of $E$ and $E'$ are trivial.
We write the $t$-multiplication maps of $E$ and $E'$ as
\[
\Phi^E_t=\theta+ a_1\tau+ a_2 \tau^2,\quad \Phi^{E'}_t=\theta+ a'_1\tau+ a'_2 \tau^2
\]
with some $a_1,a'_1\in B$ and $a_2,a'_2\in B^\times$.
By assumption, we have $a_1^{q+1}/a_2=(a'_1)^{q+1}/a'_2 \in B^\times$ and thus $a_1, a'_1 \in B^\times$. Hence the scheme
\[
J=\Spec(B[Y]/(Y^{q-1}-a_1/a'_1))
\]
is a finite etale $\bF_q^\times$-torsor over $B$. By $Y\mapsto (X\mapsto YX)$, we obtain a map of functors $J\to 
\sIsom_{A,S}(E,E')$. To show that it is an isomorphism, we may prove it over $J$. In this case, it follows from Lemma 
\ref{DMAut}.
\end{proof}



\subsection{Tate-Drinfeld modules}

To investigate the structure around cusps of Drinfeld modular curves and extend the sheaf $\omega_\univ^\Delta$, we 
need to introduce the Tate-Drinfeld module. Let $R_0$ be a flat $A_\frn$-algebra which is an excellent Noetherian domain with fraction 
field $K_0$. Let $R_0((x))$ and $K_0((x))$ be the Laurent power series rings over $R_0$ and $K_0$, respectively. Put 
$T_0=\Spec(R_0((x)))$. We denote the normalized $x$-adic valuation on $K_0((x))$ by $v_x$. We also denote the ring of 
entire series over $K_0((x))$ by $K_0((x))\{\{X\}\}$; it is the subring of $K_0((x))[[X]]$ consisting of elements 
$\sum_{i\geq 0} a_i X^i$ satisfying
\[
\lim_{i\to \infty} (v_x(a_i)+ i\rho) =+\infty\text{ for any }\rho\in \bR.
\]
We put $R_0[[x]]\{\{X\}\}=K_0((x))\{\{X\}\}\cap R_0[[x]][[X]]$.

Let $(C,\Phi^C)$ be the Carlitz module over $R_0$. For any non-zero element $f\in A$, put
\begin{gather}
f\Lambda=\left\{\Phi_{fa}^C\left(\frac{1}{x}\right)\ \middle |\ a\in A\right\}\subseteq R_0((x)),\label{EqnDefLambda}\\
e_{f\Lambda}(X)=X\prod_{\alpha\neq 0\in f\Lambda}\left(1-\frac{X}{\alpha}\right)\in X+x X^2 R_0[[x]][[X]]\label{EqnDefExp}
\end{gather}
as in \cite[Ch.~5, \S2]{Lehmkuhl}. Note that any non-zero element of $f\Lambda$ is invertible in $R_0((x))$. We 
consider $f\Lambda$ as an $A$-module via $\Phi^C$. 
Then it is a free $A$-module of rank one, and it is also discrete inside $K_0((x))$. Hence the power series $e_{f\Lambda}(X)$ is entire, and it is an element of $R_0[[x]]\{\{X\}\}$. 

Put
\[
F_f(x)=\frac{1}{\Phi^C_f\left(\frac{1}{x}\right)}\in x^{q^{\deg(f)}}\bF_q^\times (1+ x R_0[[x]]).
\]
Then $x\mapsto F_f(x)$ defines an $R_0$-algebra homomorphism $\nu_f^\sharp:R_0((x))\to R_0((x))$ and a map $\nu_f:T_0\to T_0$. For any element $h(X)=\sum_i a_i X^i\in R_0((x))[[X]]$, we put $\nu_f^*(h)(X)=\sum_i \nu_f^\sharp(a_i)X^i$. Then we have $\nu_f^\sharp(\Lambda)=f\Lambda$ and $\nu_f^*(e_\Lambda)(X)=e_{f\Lambda}(X)$.

For any element $a\in A$, consider the power series
\begin{equation}\label{EqnDefTDa}
\Phi^{f\Lambda}_a(X)=e_{f\Lambda}(\Phi^C_a(e_{f\Lambda}^{-1}(X))) \in R_0[[x]][[X]].
\end{equation}
Note that (\ref{EqnDefExp}) yields
\begin{equation}\label{EqnTDRed}
\Phi^{f\Lambda}_a(X)\equiv \Phi^C_a(X) \bmod x R_0[[x]] \text{ for any }a\in A.
\end{equation}
Let $K_0((x))^{\alg}$ be an algebraic closure of $K_0((x))$. For any $a\in A$, put
\[
(\Phi^C_a)^{-1}(f\Lambda)=\{y\in K_0((x))^{\alg}\mid \Phi^C_a(y)\in f\Lambda\},
\]
which is an $A$-module, and let $\Sigma_a\subseteq (\Phi^C_a)^{-1}(f\Lambda)$ be a representative of the set 
\[
((\Phi^C_a)^{-1}(f\Lambda)/f\Lambda)\setminus\{0\}.
\]
Since $R_0$ is flat over $A$, we have
\begin{equation}\label{EqnTDBoe}
\Phi^{f\Lambda}_a(X)=a X \prod_{\beta \in \Sigma_a} \left(1-\frac{X}{e_{f\Lambda}(\beta)}\right)
\end{equation}
(see for example the proof of \cite[Proposition 2.9]{Boeckle}). In particular, it is an $\bF_q$-linear additive polynomial of degree $q^{2 \deg(a)}$.

\begin{lem}\label{TDCoeff}
If we write as $\Phi^{\Lambda}_t=\theta+ a_1 \tau+a_2 \tau^2$ for some $a_i\in R_0[[x]]$,
then we have 
\[
a_1\in 1+x R_0[[x]],\quad a_2\in x^{q-1} R_0[[x]]^\times.
\]
\end{lem}
\begin{proof}
The assertion on $a_1$ follows from (\ref{EqnTDRed}). That on $a_2$ is proved by the computation in the proof of \cite[Lemma 2.10]{Boeckle}. Indeed, we choose a root $\eta\in K_0((x))^{\alg}$ of the equation 
\[
\Phi^C_t(X)=\theta X+ X^q=\frac{1}{x}.
\]
Put $\tilde{\Sigma}=\{c\eta\mid c\in \bF_q^\times \}$, $\Sigma_0=\{\zeta \in K_0((x))^{\alg}\mid \Phi^C_t(\zeta)=0\}$ and $\Sigma_t=(\tilde{\Sigma}+\Sigma_0)\cup (\Sigma_0\setminus\{0\})$. By (\ref{EqnTDBoe}), we have $a_2=\theta/(\prod_{\beta\in \Sigma_t} e_\Lambda(\beta))$. The denominator $\prod_{\beta\in \Sigma_t} e_\Lambda(\beta)$ is equal to
\[
\prod_{\beta\in \tilde{\Sigma}}\prod_{\zeta\in \Sigma_0}(\beta+\zeta)\prod_{\alpha\neq 0\in \Lambda}\left(\frac{\alpha-(\beta+\zeta)}{\alpha}\right)\cdot \prod_{\zeta\in \Sigma_0\setminus \{0\}}\zeta\prod_{\alpha\neq 0\in \Lambda}\left(\frac{\alpha-\zeta}{\alpha}\right).
\]
The first term is equal to 
\[
\prod_{\beta\in\tilde{\Sigma}}\Phi^C_t(\beta)\prod_{\alpha\neq 0\in \Lambda}\left(\frac{\Phi^C_t(\alpha-\beta)}{\alpha^q}\right)=\frac{\left(\prod_{c\in \bF_q^\times} c\right)}{x^{q-1}}\cdot \prod_{c\in \bF_q^\times}\prod_{\alpha\neq 0\in \Lambda}\frac{\theta \alpha +\alpha^q-\frac{c}{x}}{\alpha^q}.
\]
By the definition (\ref{EqnDefLambda}) of $\Lambda$, any $\alpha\neq 0\in \Lambda$ can be written as $\alpha=\Phi^C_a(1/x)$ for some $a\neq 0\in A$. Thus we have $\alpha=x^{-q^r}h$ with $r=\deg(a)$ and $h\in R_0[[x]]^\times$, which yields $(\theta \alpha +\alpha^q-c/x)/\alpha^q\in 1+x R_0[[x]]$. By a similar computation, the second term is equal to
\[
\theta \prod_{\alpha\neq 0\in \Lambda}\frac{\theta  +\alpha^{q-1}}{\alpha^{q-1}}\in \theta(1+x R_0[[x]]).
\]
Hence we obtain the assertion on $a_2$.
\end{proof}

Using Lemma \ref{TDCoeff} and the map $\nu_f$, we see that 
the polynomials $\Phi^{f\Lambda}_a$ define a structure of a Drinfeld module of rank two over $T_0$. We refer to it as the 
Tate-Drinfeld module $\TD(f\Lambda)$ over $T_0$.

\begin{lem}\label{TDKS}
Let $X$ be the parameter of $\TD(\Lambda)$ as above.
We trivialize the underlying invertible sheaf $\omega_{\TD(\Lambda)}^{\otimes q}$ of the dual $\TD(\Lambda)^D$ by $(dX)^{\otimes q}$, and we denote the corresponding parameter of $\TD(\Lambda)^D$ by $Y$.
Then
the dual of the Kodaira-Spencer map
\[
\KS^\vee: \omega_{\TD(\Lambda)}\otimes \omega_{\TD(\Lambda)^D}\to \Omega^1_{R_0((x))/R_0}
\]
satisfies $\KS^\vee(d X\otimes d Y)=l(x) d x$ with
\[
l(x)=\frac{d a_1}{d x}-\frac{a_1}{a_2}\frac{d a_2}{d x} \equiv \frac{1}{x} \bmod R_0[[x]].
\]
\end{lem}
\begin{proof}
We want to compute $\nabla_{\frac{d}{d x}}(d X)$. For this, first note that the inner biderivation $\delta_{\id}\in\Derin(\TD(\Lambda),\Ga)$ gives $d X$ via the second isomorphism of Lemma \ref{DualLieCot}. Then we have
\[
\delta_{\id,t}=\id\circ \Phi^{\TD(\Lambda)}_t-\Phi^{\Ga}_t\circ\id=a_1\tau+ a_2\tau^2
\]
and $\nabla_{\frac{d}{d x}}(d X)$ corresponds to the class of $\delta\in \Der_0(\TD(\Lambda),\Ga)$ satisfying $\delta_t=\frac{d 
a_1}{d x}\tau+\frac{d a_2}{d x} \tau^2$. Subtracting the inner biderivation $\delta_\beta$ for $\beta=a_2^{-1}\frac{d 
a_2}{d x}$, we may assume $\delta_t=l(x)\tau$. Hence, the element $\pi_{\frac{d}{d x}}(d X)\in \Lie(\TD(\Lambda)^D)$ is 
given by the biderivation $\delta'\in \Der_0(\TD(\Lambda)|_{T_{0,\varepsilon}},C|_{T_{0,\varepsilon}})$ satisfying 
$\delta'_t=\varepsilon l(x)\tau$, where $T_{0,\varepsilon}=\Spec_{T_0}(\cO_{T_0}[\varepsilon]/(\varepsilon^2))$. The map $\delta'_t$ is an element of 
$\Hom_{\bF_q,T_{0,\varepsilon}}(\TD(\Lambda)|_{T_{0,\varepsilon}},C|_{T_{0,\varepsilon}})$ defined by $Z\mapsto \varepsilon l(x) 
X^q$. Let $\cL=R_0((x)) \frac{d}{d X}$ be the underlying invertible sheaf of $\TD(\Lambda)$. Via the identification 
(\ref{EqnDualBdl}), the above homomorphism corresponds to the element 
\[
\varepsilon l(x) (d X)^{\otimes q}\in\Ker(\bV_*(\cL^{\otimes -q})(T_{0,\varepsilon})\to \bV_*(\cL^{\otimes -q})(T_0))
\]
and, with the parameter $Y$ of $\bV_*(\cL^{\otimes -q})$ in the lemma, it corresponds to $l(x)\frac{d}{d Y}$. This concludes the proof.
\end{proof}

\begin{lem}\label{TDGamma1}
For any monic polynomial $\frem\in A$, there exists a natural $A$-linear closed immersion $\lambda_{\infty,\frem}^{f\Lambda}:C[\frem]\to \TD(f\Lambda)$ over $T_0$ satisfying $\nu_f^*(\lambda_{\infty,\frem}^{\Lambda})=\lambda_{\infty,\frem}^{f\Lambda}$. In particular, the Tate-Drinfeld module $\TD(f\Lambda)$ is endowed with a natural $\Gamma_1(\frn)$-structure $\lambda_{\infty,\frn}^{f\Lambda}$ over $T_0$.
\end{lem}
\begin{proof}
Let $R_0[[x]]\langle Z\rangle$ be the $x$-adic completion of the ring $R_0[[x]][Z]$.
We have a natural map
\[
i:R_0[[x]][Z]/(\Phi_\frem^C(Z))\to R_0[[x]]\langle Z\rangle/(\Phi_\frem^C(Z)).
\]
Since $\Phi_\frem^C(Z)\in R_0[Z]$ is monic, the ring on the left-hand side is finite over the $x$-adically complete Noetherian ring $R_0[[x]]$. Hence this ring is also $x$-adically complete and the map $i$ is an isomorphism.
Since $R_0[[x]]\{\{Z\}\}\subseteq R_0[[x]]\langle Z\rangle$, the map 
\[
R_0[[x]][X] \to R_0[[x]]\{\{Z\}\},\quad X\mapsto e_{f\Lambda}(Z)
\]
induces a homomorphism of Hopf algebras
\[
R_0((x))[X]\to R_0[[x]]\langle Z \rangle[1/x]/(\Phi_\frem^C(Z))\overset{i^{-1}}{\to} R_0((x))[Z]/(\Phi_\frem^C(Z)),
\]
which we denote by $(\lambda_{\infty,\frem}^{f\Lambda})^*$.
In the ring $R[[x]]\langle Z\rangle$, we have $\Phi_a^{f\Lambda}(e_{f\Lambda}(Z))=e_{f\Lambda}(\Phi_a^C(Z))$ for any 
$a\in A$ and this implies that the map $(\lambda_{\infty,\frem}^{f\Lambda})^*$ is compatible with $A$-actions.
Thus we obtain a homomorphism of finite locally free $A$-module schemes over $T_0$
\[
\lambda_{\infty,\frem}^{f\Lambda}: C[\frem]\to \TD(f\Lambda)[\frem]
\]
which is compatible with the map $\nu_f$.

To prove that it is a closed immersion, it is enough to show that the map $R_0[[x]][X]\to R_0[[x]]\langle Z \rangle/(\Phi_\frem^C(Z))$ defined by $X\mapsto e_{f\Lambda}(Z)$ is surjective. Since the right-hand side is $x$-adically complete, it suffices to show the surjectivity modulo $x$, which follows from (\ref{EqnDefExp}). 
\end{proof}

\begin{lem}\label{ELambEval}
Let $D$ be any finite flat $R_0((x))$-algebra whose restriction to $\Frac(R_0((x)))$ is etale, and $\delta$ any element of $D$. Let $\cD$ be the integral closure of $R_0[[x]]$ in $D$. We consider $D$ as a topological ring by taking $\{x^l \cD\}_{l\in \bZ_{\geq 0}}$ as a fundamental system of neighborhoods of $0\in D$. Then, for any $F(X)\in R_0((x))\{\{X\}\}$, the evaluation $F(\delta)$ converges for any $\delta\in D$. In particular, we have an $\bF_q$-linear map $e_{f\Lambda}: D\to D$ which is functorial on $D$.
\end{lem}
\begin{proof}
We have $\cD[1/x]=D$. Since $R_0$ is excellent, so is the power series ring $R_0[[x]]$. Thus $\cD$ is finite over 
$R_0[[x]]$ and $x$-adically complete. (Here the fact that $R_0[[x]]$ is excellent follows from an unpublished work of Gabber \cite[Main Theorem 2]{KS}. If we assume that $R_0$ is regular, then the finiteness of $\cD$ follows from \cite[Proposition (31.B)]{Matsumura}. This is the only case we need.)
This implies that the evaluation $F(\delta)$ converges and defines an element of $D$.
\end{proof}

Put $\cH_{\infty,\frem}^{f\Lambda}=\TD(f\Lambda)[\frem]/\Img(\lambda_{\infty,\frem}^{f\Lambda})$ and 
\begin{equation}\label{EqnDefBn}
B_{0,\frem}^{f\Lambda}=R_0((x))[\eta]/(\Phi^C_\frem(\eta)-\Phi_f^C(1/x)).
\end{equation}
Then $\Spec(B_{0,\frem}^{f\Lambda})$ is a finite flat $C[\frem]$-torsor over $T_0$. Since $\frem$ is invertible in $K_0$, it is etale over $\Frac(R_0((x)))$. 

\begin{lem}\label{TDCokerTriv}
For any monic polynomial $\frem\in A$, there exists an $A$-linear isomorphism 
$\mu_{\infty,\frem}^{f\Lambda}: \underline{A/(\frem)}\to \cH_{\infty,\frem}^{f\Lambda}$ which is compatible with the 
map $\nu_f$ such that the image of $\mu_\infty^{f\Lambda}(1)\in \cH_{\infty,\frem}^{f\Lambda}(T_0)$ in 
$\cH_{\infty,\frem}^{f\Lambda}(B_{0,\frem}^{f\Lambda})$ is equal to the image $\overline{e_{f\Lambda}(\eta)}$ of the 
element $e_{f\Lambda}(\eta)\in \TD(f\Lambda)[\frem](B_{0,\frem}^{f\Lambda})$. 
In particular, we have an exact sequence of $A$-module schemes over $T_0$ 
\begin{equation}\label{TDfrnExact}
\xymatrix{
0 \ar[r] & C[\frem]\ar[r]^-{\lambda_{\infty,\frem}^{f\Lambda}} & \TD(f\Lambda)[\frem] \ar[r]^-{\pi_{\infty,\frem}^{f\Lambda}} & \underline{A/(\frem)} \ar[r] & 0.
}
\end{equation}
\end{lem}
\begin{proof}
By Lemma \ref{ELambEval}, we have an element $e_{f\Lambda}(\eta)\in \TD(f\Lambda)[\frem](B_{0,\frem}^{f\Lambda})$. 
Since its image $\overline{e_{f\Lambda}(\eta)}$ in $\cH_{\infty,\frem}^{f\Lambda}(B_{0,\frem}^{f\Lambda})$ is invariant 
under the action of $C[\frem]$ on $B_{0,\frem}^{f\Lambda}$, we obtain $\overline{e_{f\Lambda}(\eta)}\in 
\cH_{\infty,\frem}^{f\Lambda}(T_0)$. This yields an $A$-linear homomorphism $\underline{A/(\frem)}\to 
\cH_{\infty,\frem}^{f\Lambda}$ over $T_0$ which is compatible with the map $\nu_f$.  

To see that it is an isomorphism, using the map $\nu_f$ we reduce ourselves to the case of $f=1$. Since the element $\frem$ is invertible in $K_0$, using co-Lie complexes we obtain the exact sequence
\[
\xymatrix{
0 \ar[r] & 	\omega_{\cH_{\infty,\frem}^\Lambda} \ar[r] & \omega_{\TD(\Lambda)[\frem]} \ar[r]^-{(\lambda_{\infty,\frem}^{\Lambda})^*} & \omega_{C[\frem]}\ar[r] & 	0.		
}
\]
We also see that the natural sequence
\[
\xymatrix{
	0 \ar[r] & 	\omega_{\TD(\Lambda)} \ar[r]^{\frem} & \omega_{\TD(\Lambda)} \ar[r] & \omega_{\TD(\Lambda)[\frem]}\ar[r] & 	0
}
\]
is exact and similarly for $C[\frem]$. Since we have $d(e_\Lambda(Z))=dZ$, the map $(\lambda_{\infty,\frem}^{\Lambda})^*$ is an isomorphism. Hence $\omega_{\cH_{\infty,\frem}^\Lambda}=0$ and $\cH_{\infty,\frem}^\Lambda$ is etale.

Now it is enough to show $a\overline{e_{\Lambda}(\eta)}\neq 0$ in 
$\cH_{\infty,\frem}^{\Lambda}(B_{0,\frem}^{f\Lambda})$ for any non-zero element $a\in A/(\frem)$. For this, we may 
assume $R_0=K_0$. In this case, note that the polynomial $\Phi^C_\frem(X)-1/x$ is irreducible over $K_0((x))$, since 
the equation $\Phi^C_\frem(1/X)=1/x$ gives an Eisenstein extension over $K_0[[x]]$. Hence we may consider the ring 
$B^\Lambda_{0,\frem}$ as a subfield of $K_0((x))^{\alg}$. 
Let $\hat{a}\in A$ be a lift of $a$ satisfying $\deg(\hat{a})<\deg(\frem)$. The condition $a\overline{e_{\Lambda}(\eta)}= 0$ implies $\Phi^C_{\hat{a}}(\eta)\equiv \zeta\bmod \Lambda$ for some root $\zeta$ of $\Phi^C_\frem(X)$ in $K_0((x))^{\alg}$. By inspecting $x$-adic valuations it forces $\zeta=0$, and the irreducibility of $\Phi^C_\frem(X)-1/x$ implies $\hat{a}=0$. This concludes the proof.
\end{proof}
We often write $\lambda_{\infty,\frn}^{\Lambda}$, $\cH_{\infty,\frn}^{\Lambda}$, $B_{0,\frn}^{\Lambda}$, $\mu_{\infty,\frn}^{\Lambda}$ and $\pi_{\infty,\frn}^{\Lambda}$ as $\lambda_\infty$, $\cH_\infty$, $B_0$, $\mu_\infty$ and $\pi_\infty$, respectively.

Put $S_0=\Spec(R_0((y)))$ and consider the morphism 
\[
\sigma_{q-1}:T_0\to S_0
\] 
defined by $y\mapsto x^{q-1}$. The $S_0$-scheme $T_0$ is a finite etale $\bF_q^\times$-torsor, where $c\in \bF_q^\times$ acts on it by the $R_0$-linear map
\[
g_c:R_0((x))\to R_0((x)),\quad x\mapsto c^{-1} x.
\]
Since $\Lambda$ is stable under this $\bF_q^\times$-action, we see that the coefficients of $e_\Lambda(X)$ and $\Phi^\Lambda_a(X)$ are in $R_0[[x^{q-1}]]$ for any $a\in A$ \cite[\S 5C1]{Armana}. This means that there exists a unique pair of a Drinfeld module and its $\Gamma_1(\frn)$-structure over $S_0$ 
\[
(\TD^\triangledown(\Lambda),\lambda^\triangledown_\infty)
\] 
satisfying $\sigma_{q-1}^*(\TD^\triangledown(\Lambda),\lambda^\triangledown_\infty)=(\TD(\Lambda),\lambda_\infty)$. 
Over $T_0$, the Tate-Drinfeld module $\TD^\triangledown(\Lambda)|_{T_0}=\TD(\Lambda)$ has a $\Gamma_1^\Delta(\frn)$-structure 
\[
(\TD(\Lambda),\lambda_\infty,[\mu_\infty])
\]
with the element $[\mu_\infty]\in (I_{(\TD(\Lambda),\lambda_\infty)}/\Delta)(T_0)$ defined by $\mu_\infty$.
We also put
\[
\cH^\triangledown_\infty = \TD^\triangledown(\Lambda)[\frn]/\Img(\lambda^\triangledown_\infty),\quad I^\triangledown_\infty=\sIsom_{A,S_0}(\underline{A/(\frn)},\cH_\infty^\triangledown).
\]

\begin{lem}\label{TDMuInfty}
There exists an isomorphism of finite etale $\bF_q^\times$-torsors over $S_0$
\[
T_0\to I^\triangledown_\infty/\Delta.
\]
\end{lem}
\begin{proof}
It is enough to give an $\bF_q^\times$-equivariant morphism $T_0\to I_{\infty}^\triangledown$ over $S_0$, which amounts 
to giving an $A$-linear isomorphism $\mu: \underline{A/(\frn)}\to \cH_\infty$ over $T_0$ satisfying $c\mu=g_c^*(\mu)$ 
for any $c\in \bF_q^\times$. The map $g_c$ extends to a similar $R_0((x))$-linear isomorphism on $B_0$ via $\eta\mapsto 
c\eta$, which we denote by $\tilde{g}_c$. Then the inclusion $\cH_\infty(R_0((x)))\to \cH_\infty(B_0)$ is compatible 
with $g_c$ and $\tilde{g}_c$. 
Consider the isomorphism $\mu_\infty$ of Lemma \ref{TDCokerTriv}. We have $\tilde{g}_c(e_\Lambda(\eta))=e_\Lambda(c \eta)$ in $B_0$ and this yields $c\mu_\infty=g_c^*(\mu_\infty)$.
\end{proof}



\subsection{Structure around cusps I}\label{SubsecCuspI}

Suppose moreover that $R_0$ is regular. Note that Lemma \ref{TDCoeff} implies
\begin{equation}\label{EqnTDdj}
j_t(\TD^\triangledown(\Lambda))\in y^{-1} R_0[[y]]^\times.
\end{equation}
We define a scheme $\widehat{\Cusps}_{R_0}^\Delta$ by the cartesian diagram
\[
\xymatrix{
\widehat{\Cusps}_{R_0}^\Delta \ar[r]\ar[d] & X_1^\Delta(\frn)_{R_0}\ar[d] \\
\Spec(R_0[[\frac{1}{j}]]) \ar[r] & \bP^1_{R_0}
}
\]
and put $\Cusps_{R_0}^\Delta=(\widehat{\Cusps}_{R_0}^\Delta|_{V(1/j)})_\red$.
Since $Y_1^\Delta(\frn)_{R_0}$ is regular and (\ref{EqnTDdj}) implies that the map $j_t$ induces an isomorphism 
\[
y^\triangledown :S_0=\Spec(R_0((y)))\to \Spec(R_0((1/j))),
\]
we see as in the proof of \cite[Lemma 8.11.4]{KM} that $\widehat{\Cusps}_{R_0}^\Delta$ is isomorphic to the normalization of $\cS_0=\Spec(R_0[[y]])$ in the scheme $Y_1^\Delta(\frn)_{S_0}$ defined by the cartesian diagram
\[
\xymatrix{
Y_1^\Delta(\frn)_{S_0} \ar[rr]\ar[d] & &Y_1^\Delta(\frn)_{R_0}\ar[d] \\
S_0 \ar[r]_-{y^\triangledown} & \Spec(R_0((1/j))) \ar[r]& \bA^1_{R_0}.
}
\]

For $\bullet\in\{\emptyset, \Delta\}$, let us consider the functor sending a scheme $S$ over $S_0$ to the set of $\Gamma_1^\bullet(\frn)$-structures on $\TD^\triangledown(\Lambda)|_S$, which is representable by a finite etale scheme $[\Gamma^\bullet_1(\frn)]_{\TD^\triangledown}$ over $S_0$. By Lemma \ref{DMAut} and Lemma \ref{DMIsog}, as in the proof of \cite[Corollary 8.4.4]{KM} we obtain a natural isomorphism
\[
[\Gamma_1^\Delta(\frn)]_{\TD^\triangledown}/\bF_q^\times \to Y_1^\Delta(\frn)_{S_0},
\]
where $\bF_q^\times$ acts as the automorphism group of $\TD^\triangledown(\Lambda)$.
Thus $\widehat{\Cusps}_{R_0}^\Delta$ is isomorphic to the quotient $\cZ_{R_0}^\Delta/\bF_q^\times$ of the normalization $\cZ_{R_0}^\Delta$ of $\cS_0$ in $[\Gamma_1^\Delta(\frn)]_{\TD^\triangledown}$ by the induced action of $\bF_q^\times$. Note that we have a natural identification 
\[
[\Gamma_1(\frn)]_{\TD^\triangledown}\times_{S_0}T_0=[\Gamma_1(\frn)]_{\TD},
\]
where the right-hand side is a similar finite etale scheme over $T_0$ for $\TD(\Lambda)$. We also put $\cT_0=\Spec(R_0[[x]])$. It is normal since $R_0$ is regular.

\begin{lem}\label{GammaDeltaNonCan}
There exists a natural isomorphism over $S_0$
\[
[\Gamma_1(\frn)]_{\TD}=[\Gamma_1(\frn)]_{\TD^\triangledown}\times_{S_0}T_0 \to [\Gamma_1^\Delta(\frn)]_{\TD^\triangledown}
\]
which is compatible with actions of $\bF_q^\times=\Aut_{A,S_0}(\TD^\triangledown(\Lambda))$. Here this group acts on the left-hand side diagonally.
\end{lem}
\begin{proof}
Let $\lambda$ be the universal $\Gamma_1(\frn)$-structure on $\TD^\trd(\Lambda)$ over $[\Gamma_1(\frn)]_{\TD^\triangledown}$. Taking the determinant of locally constant etale sheaves of locally free $A/(\frn)$-modules, we obtain a natural isomorphism of $A$-module schemes $\iota: \cH_\infty^\trd|_{[\Gamma_1(\frn)]_{\TD^\triangledown}}\to \TDt(\Lambda)[\frn]/\Img(\lambda)$. Then, by Lemma \ref{TDMuInfty}, the map
\[
(I_\infty^\trd/\Delta)|_{[\Gamma_1(\frn)]_{\TD^\triangledown}}\to  [\Gamma_1^\Delta(\frn)]_{\TD^\triangledown},\quad 
[(j:\underline{A/(\frn)}\to \cH_\infty^\trd)]\mapsto [\iota\circ j]
\]
gives the desired isomorphism.
\end{proof}

\begin{lem}\label{GammaDeltaModuli}
The scheme $\cZ^\Delta_{R_0}$ over $\cS_0$ is decomposed as
\[
\cZ^\Delta_{R_0}=\cZ_{R_0}^{\Delta,0}\sqcup \cZ_{R_0}^{\Delta,\neq 0}, \quad \cZ_{R_0}^{\Delta,0}=\coprod_{(A/(\frn))^\times} \cT_0.
\]
Moreover, the group $\bF_q^\times=\Aut_{A,S_0}(\TDt(\Lambda))$ induces free actions on the two components of the former decomposition.
\end{lem}
\begin{proof}
	First note that, for any scheme $S$ over $A_\frn$ and any finite etale $A$-module scheme $\cG$ over $S$, the big fppf sheaf $\sHom_{A,S}(C[\frn],\cG)$ is representable by a finite etale $A$-module scheme over $S$ and thus its zero section is a closed and open immersion. 
	
	Since $\cT_0$ is normal, Lemma \ref{GammaDeltaNonCan} implies that $\cZ^\Delta_{R_0}$ is identified with the normalization of $\cT_0$ in the finite etale scheme $[\Gamma_1(\frn)]_{\TD}$ over $T_0$. For any scheme $T$ over $T_0$, we have an exact sequence of finite etale $A$-module schemes over $T$
	\[
	\xymatrix{
		0 \ar[r] & C[\frn]|_T \ar[r]^-{\lambda_\infty} & \TD(\Lambda)[\frn]|_T \ar[r]^-{\pi_\infty}& \underline{A/(\frn)}|_T \ar[r] & 0.
	}
	\]
	Any $\Gamma_1(\frn)$-structure $\lambda: C[\frn]|_T\to\TD(\Lambda)[\frn]|_T$ over $T$ induces an $A$-linear homomorphism $\pi_\infty\circ\lambda: C[\frn]|_T\to \underline{A/(\frn)}|_T$. This gives a morphism over $T_0$
	\[
	[\Gamma_1(\frn)]_{\TD}\to \sHom_{A,T_0}(C[\frn],\underline{A/(\frn)})=T_0\sqcup U,
	\]
	where $U$ is the complement of the zero section. 
	Let $[\Gamma_1(\frn)]_{\TD}^{0}$ be the inverse image of $T_0$. It is isomorphic to $\sAut_{A,T_0}(C[\frn])=\underline{(A/(\frn))^\times}$.
	
	Since $\sHom_{A,\cT_0}(C[\frn],\underline{A/(\frn)})$ is also a finite etale $A$-module scheme over $\cT_0$, it agrees with the normalization of $\cT_0$ in $\sHom_{A,T_0}(C[\frn],\underline{A/(\frn)})$. Moreover, it is etale locally isomorphic to $\underline{A/(\frn)}$. Thus we obtain a map
	\[
	\cZ^\Delta_{R_0}\to \sHom_{A,\cT_0}(C[\frn],\underline{A/(\frn)})=\cT_0\sqcup \cU,
	\]
	where $\cU$ is the complement of the zero section. Since $\cU$ is etale locally isomorphic to $\underline{A/(\frn)\setminus\{0\}}$, the group $\bF_q^\times$ acts freely on $\cU$.
	
	Let $\cZ_{R_0}^{\Delta,0}$ and $\cZ_{R_0}^{\Delta,\neq 0}$ be the inverse images of $\cT_0$ and $\cU$, respectively. Since the component $\cZ_{R_0}^{\Delta, 0}$ is the normalization of $\cT_0$ in $[\Gamma_1(\frn)]_{\TD}^0$, the latter decomposition of the lemma follows. Hence we also obtain the freeness of the $\bF_q^\times$-actions as in the lemma.
\end{proof}

The tuple $(\TD(\Lambda),\lambda_\infty,[\mu_\infty])$ over $T_0$ gives a map $T_0\to Y^\Delta_1(\frn)_{R_0}$.
Since the ring $R_0[[x]]$ is normal, this extends to a map
\[
x^\Delta_\infty: \cT_0 \to X^\Delta_1(\frn)_{R_0}.
\]
The $R_0$-algebra homomorphism defined by $x\mapsto 0$ gives a point $P_\infty^\Delta\in X^\Delta_1(\frn)_{R_0}$, which we refer to as the $\infty$-cusp. We write the complete local ring at this point as $\hat{\cO}_{X_1^\Delta(\frn)_{R_0},P^\Delta_\infty}$.

\begin{lem}\label{GammaDeltaDesc}
Suppose that $R_0$ is a flat $A_\frn$-algebra which is an excellent regular domain. 
\begin{enumerate}
\item\label{GammaDeltaDescCompl} 
The map $x^\Delta_\infty$ induces an isomorphism of complete local rings
\[
(x^\Delta_\infty)^*: \hat{\cO}_{X_1^\Delta(\frn)_{R_0},P^\Delta_\infty} \to R_0[[x]].
\]
\item\label{GammaDeltaDescSheaf}
The invertible sheaf $\omega_\univ^\Delta$ on $Y_1^\Delta(\frn)_{R_0}$ extends to an invertible sheaf $\bar{\omega}_\univ^\Delta$ on $X_1^\Delta(\frn)_{R_0}$ satisfying
\[
(x_\infty^\Delta)^*(\bar{\omega}^\Delta_\univ)=R_0[[x]] d X,
\]
where $dX$ denotes the invariant differential of $\TDt(\Lambda)$ associated to its parameter $X$. 

\item\label{GammaDeltaDescBC}
The formation of $\bar{\omega}_\univ^\Delta$ is compatible with any base change $R_0\to R'_0$ of flat $A_\frn$-algebras which are excellent regular domains.

\item\label{GammaDeltaDescCompat}
The natural action of $\bF_q^\times$ on $\omega_\univ^\Delta$ via $c\mapsto [c]_\Delta$ extends to an action on $\bar{\omega}_\univ^\Delta$ covering its action on $X_1^\Delta(\frn)_{R_0}$.
\end{enumerate}
\end{lem}
\begin{proof}
The assertion (\ref{GammaDeltaDescCompl}) follows from Lemma \ref{GammaDeltaModuli}.
Moreover, Lemma \ref{GammaDeltaModuli} also implies that the trivial invertible sheaf $\cO_{\cZ_{R_0}^\Delta}dX$, with the natural $\bF_q^\times$-action via $X\mapsto cX$ which covers the action on $\cZ_{R_0}^\Delta$, descends to the quotient $\cZ_{R_0}^\Delta/\bF_q^\times\simeq \widehat{\Cusps}_{R_0}^\Delta$ and we obtain $\bar{\omega}_\univ^\Delta$ by gluing. 
(\ref{GammaDeltaDescBC}) follows from the uniqueness of the descended sheaf.

For (\ref{GammaDeltaDescCompat}), Lemma \ref{GammaDeltaNonCan} implies that $[c]_\Delta$ acts on 
\[
\cP:=[\Gamma_1^\Delta(\frn)]_{\TD^\triangledown}\simeq [\Gamma_1(\frn)]_{\TD^\triangledown}\times_{S_0}T_0
\]
via $1\times g_c^*$. Thus, for the universal $\Gamma_1(\frn)$-structure $\lambda^\trd_\univ$ on $\TDt(\Lambda)$ over $[\Gamma_1(\frn)]_{\TD^\triangledown}$,
we have
\[
[c]_\Delta^*(\TDt(\Lambda)|_{\cP}, \lambda^\trd_\univ|_{\cP})=(\TDt(\Lambda)|_{\cP}, \lambda^\trd_\univ|_{\cP}).
\]
Since any $\Gamma_1(\frn)$-structure has no non-trivial automorphism, the natural action of $[c]_\Delta$ on $\omega_\univ^\Delta|_{\cP/\bF_q^\times}$ is the descent of the map given by
\[
[c]_\Delta^*(\cO_\cP dX)\to \cO_\cP dX,\quad dX\otimes 1\mapsto dX.
\]
Hence it extends to the sheaf $\cO_{\cZ_{R_0}^\Delta} dX$, and thus to $\bar{\omega}_\univ^\Delta$.
\end{proof}

\subsection{Structure around cusps II}\label{SubsecCuspII}

Let $W_\frn(X)$ be the unique monic prime factor of $\Phi^C_\frn(X)$ in $A[X]$ which does not divide $\Phi^C_\frem(X)$ for any non-trivial divisor $\frem$ of $\frn$ \cite[\S3]{Carlitz}.
Then
\[
I=\sIsom_{A,R_0}(\underline{A/(\frn)},C[\frn])
\]
is represented by $\Spec(R_0[X]/(W_\frn(X)))$, which is finite etale over $R_0$. For any scheme $S$ over $R_0$, we put $I_S=I\times_{R_0}S$.

Let $R_\frn$ be the affine ring of a connected component of $I$, which is a finite etale domain over $R_0$.
We denote by $\zeta$ the image of $X$ in $R_\frn$.
In the sequel, we also need an explicit description of the scheme $[\Gamma_1^\Delta(\frn)]_{\TD^\triangledown}$ over $S_\frn=\Spec(R_\frn((y)))$. 

Put $T_\frn=\Spec(R_\frn((x)))$. By Lemma \ref{GammaDeltaNonCan}, it is enough to describe the restriction
\[
[\Gamma_1(\frn)]_{\TD|_{T_\frn}}=[\Gamma_1(\frn)]_{\TD}\times_{T_0} T_\frn.
\]
For this, we denote by $\sH$ the set of $A$-linear surjections $(A/(\frn))^2\to A/(\frn)$. 
By the map $(a,b)\mapsto ({}^t\! (u,v)\mapsto (a,b){}^t\! (u,v))$, we identify the set $\sH$ with $\{(a,b)\in (A/(\frn))^2\mid (a,b)=(1)\}$.
As in \cite[Proposition 10.2.4]{KM}, for any $\Xi\in \sH$ we denote by $k_\Xi$ the unique generator of $\Ker(\Xi)$ satisfying $\Xi(l)=\det(k_\Xi,l)$ for any $l\in (A/(\frn))^2$. We also choose $l_\Xi\in (A/(\frn))^2$ satisfying $\Xi(l_\Xi)=1$. Then, for any $g\in \mathit{GL}_2(A/(\frn))$ there exists a unique $n(g,\Xi)\in A/(\frn)$ satisfying
\begin{equation}\label{EqnDefNxi}
l_{\Xi\circ g}=g^{-1}(l_{\Xi})+ n(g,\Xi)g^{-1}(k_\Xi).
\end{equation}
Put $\Fix(\Xi)=\{g\in \mathit{GL}_2(A/(\frn))\mid \Xi\circ g=\Xi\}$. Considering the representing matrix for $g$ with respect to the ordered basis $(k_\Xi,l_\Xi)$, we have an isomorphism
\begin{equation}\label{FixGpIsom}
\Fix(\Xi)\to \left\{\begin{pmatrix}\det(g)& n(g,\Xi)\\0 & 1\end{pmatrix}\ \middle |\ g\in \Fix(\Xi)\right\}.
\end{equation}

We denote by $[\Gamma(\frn)]_{\TD|_{T_\frn}}$ the scheme representing the functor over $T_\frn$ sending a $T_\frn$-scheme $T$ to the set of $\Gamma(\frn)$-structures on $\TD(\Lambda)|_T$. It is finite etale over $T_\frn$.
By (\ref{TDfrnExact}), to give $\alpha\in[\Gamma(\frn)]_{\TD|_{T_\frn}}(T)$ satisfying $\pi_\infty\circ\alpha=\Xi$ is the same as to give $\alpha(k_\Xi)\in C[\frn](T)$ inducing an $A$-linear isomorphism $\underline{A/(\frn)}\to C[\frn]$ and $\alpha(l_\Xi)\in \pi_\infty^{-1}([1])(T)$, where $[1]$ is the section $T_\frn\to \underline{A/(\frn)}$ corresponding to $1\in A/(\frn)$.

By taking the determinant, we have an $A$-linear isomorphism of etale sheaves of locally free $A/(\frn)$-modules
\[
\omega: \bigwedge^2 \TD(\Lambda)[\frn]\to C[\frn],
\]
which defines a map $[\Gamma(\frn)]_{\TD|_{T_\frn}}\to I$ by $(\alpha\mapsto \omega\circ \wedge^2\alpha)$. For any scheme $T$ over $T_\frn$, we say an element $\alpha\in [\Gamma(\frn)]_{\TD|_{T_\frn}}(T)$ is canonical if the map $\omega\circ\wedge^2\alpha: T\to I$ is equal to the structure map $T\to T_\frn\to I$. 
The subfunctor of canonical elements is represented by a finite etale scheme $[\Gamma(\frn)]_{\TD|_{T_\frn}}^\can$ over $T_\frn$.

\begin{lem}\label{TD1Torsor}
Put $B_\frn=R_\frn((x))[\eta]/(\Phi^C_\frn(\eta)-1/x)$. Then the map over $T_\frn$
\[
\coprod_{\Xi\in \sH} \Spec(B_\frn)\to [\Gamma(\frn)]_{\TD|_{T_\frn}}^\can,
\]
which is defined on the $\Xi$-component by the canonical $\Gamma(\frn)$-structure $(k_\Xi,l_\Xi)\mapsto (e_\Lambda(\zeta),e_\Lambda(\eta))$ over $B_\frn$, is an isomorphism.
\end{lem}
\begin{proof}
The element $e_\Lambda(\eta)\in B_\frn$ defines a map $\Spec(B_\frn)\to \pi_\infty^{-1}([1])$. Since it is $C[\frn]$-equivariant, it is an isomorphism of $C[\frn]$-torsors over $T_\frn$ and the lemma follows.
\end{proof}

Put $\bar{\Gamma}_1=\left\{\begin{pmatrix}* & 0\\ *& 1\end{pmatrix}\in \mathit{GL}_2(A/(\frn)) \right\}$ and $\bar{\Gamma}_1^1=\bar{\Gamma}_1\cap \mathit{SL}_2(A/(\frn))$.
For any element $f\neq 0\in A$, we define
\[
G_f(w)=w^{q^{\deg(f)}}-x w^{q^{\deg(f)}}\Phi^C_f\left(\frac{1}{w}\right) \in R_0[[x]][w].
\]
Then we have natural maps
\[
\xymatrix{
R_0[[w]] \ar[r] & R_0[[x]][w]/(G_f(w))\ar[r]& R_0[[w]]
}
\]
which are isomorphisms.
Moreover, for any $b\in A/(\frn)$, let $f_b$ be the monic generator of the ideal $\Ann_A(b(A/(\frn)))$. Then $f_b$ divides $\frn$ and $f_b\in A_\frn^\times$.

\begin{lem}\label{GammaDeltaExpDesc}
The scheme $[\Gamma_1^\Delta(\frn)]_{\TDt}$ over $S_\frn$ is decomposed as
\begin{align*}
[\Gamma_1^\Delta(\frn)]_{\TDt}&=\coprod_{(a,b)} \Spec(R_\frn((x))[w]/(G_{f_{b}}(w)))\simeq \coprod_{(a,b)} \Spec(R_\frn((w))).
\end{align*}
Here the direct sum is taken over a complete representative of the set
\[
\{(a,b)\in (A/(\frn))^2\mid (a,b)=(1)\}/\bar{\Gamma}^1_1.
\]
\end{lem}
\begin{proof}
	For any scheme $T$ over $T_\frn$, any $\Gamma(\frn)$-structure $\alpha$ on $\TD(\Lambda)|_T$ defines a $\Gamma_1(\frn)$-structure $\zeta\mapsto \alpha({}^t\!(0,1))$.
Since we have $\mathit{SL}_2(A/(\frn))/\bar{\Gamma}^1_1=\mathit{GL}_2(A/(\frn))/\bar{\Gamma}_1$, Lemma \ref{TD1Torsor} yields
\[
[\Gamma_1(\frn)]_{\TD|_{T_\frn}}=[\Gamma(\frn)]^\can_{\TD|_{T_\frn}}/\bar{\Gamma}^1_1=\coprod_{\Xi\in \sH/\bar{\Gamma}^1_1}\Spec(B_\frn^{\bar{\Gamma}^1_1\cap \Fix(\Xi)}).
\]
Note that, via the isomorphism of Lemma \ref{TD1Torsor}, any element $g\in \bar{\Gamma}^1_1\cap \Fix(\Xi)$ acts on $B_\frn$ of the $\Xi$-component by 
\[
\eta\mapsto \eta+ \Phi^C_{n(g,\Xi)}(\zeta).
\]

For $\Xi=(a,b)$, we have $k_\Xi={}^t\!(b,-a)$ and
\[
\bar{\Gamma}^1_1\cap \Fix(\Xi)=\left\{ \begin{pmatrix}1 &0 \\
 (f_b)/(\frn)& 1\end{pmatrix}\right\}.
\]
By the isomorphism (\ref{FixGpIsom}), the additive subgroup
\[
n(\Xi)=\{n(g,\Xi)\in A/(\frn)\mid g\in \bar{\Gamma}^1_1\cap \Fix(\Xi)\}
\]
is isomorphic to $(f_b)/(\frn)$. In particular, they have the same cardinality. On the other hand, for any $g\in \bar{\Gamma}^1_1\cap \Fix(\Xi)$, (\ref{EqnDefNxi}) yields $b n(g,\Xi)=0$ and thus $n(\Xi)\subseteq (f_b)/(\frn)$. Hence they are equal. 

Put $h_b=\frn/f_b$. Consider the map
\[
R_\frn((x))[\eta']/(\Phi^C_{f_b}(\eta')-1/x)\to B_\frn,\quad \eta'\mapsto \Phi^C_{h_b}(\eta).
\]
Note that the left-hand side is isomorphic to $R_\frn((1/\eta'))$ and thus normal. Hence this map identifies the left-hand side with $B_\frn^{\bar{\Gamma}^1_1\cap \Fix(\Xi)}$. By changing the variable as $w=1/\eta'$, Lemma \ref{GammaDeltaNonCan} yields the decomposition as in the lemma.
\end{proof}

\begin{cor}\label{GammaDeltaCusp}
Suppose that $R_0$ is a flat $A_\frn$-algebra which is an excellent regular domain. 
\begin{enumerate}
\item\label{GammaDeltaCuspCompl} We have a natural isomorphism over $R_\frn[[y]]$
\[
\widehat{\Cusps}_{R_\frn}^\Delta=\coprod_{(a,b)} \Spec(R_\frn[[x]][w]/(G_{f_{b}}(w)))\simeq \coprod_{(a,b)} \Spec(R_\frn[[w]]).
\]
Here the direct sum is taken over a complete representative of the set
\[
\bF^\times_q\backslash \{(a,b)\in (A/(\frn))^2\mid (a,b)=(1)\}/\bar{\Gamma}^1_1.
\]
\item\label{GammaDeltaCuspRed} $\Cusps_{R_0}^\Delta$ is finite etale over $R_0$. In particular, it defines an effective Cartier divisor of $X_1^\Delta(\frn)_{R_0}$ over $R_0$.
\item\label{GammaDeltaCuspDlog} At each point of $\Cusps_{R_0}^\Delta$, the invertible sheaf 
\[
\Omega^1_{X_1^\Delta(\frn)_{R_0}/R_0}(2\Cusps_{R_0}^\Delta)
\]
is locally generated by the section $dx/x^2$.
\end{enumerate}
\end{cor}
\begin{proof}
Note that the ring $R_\frn[[w]]$ is normal.
Since the group $\bF_q^\times$ acts freely on the index set of the decomposition of Lemma \ref{GammaDeltaExpDesc}, we obtain the assertion (\ref{GammaDeltaCuspCompl}), which implies the assertion (\ref{GammaDeltaCuspRed}) since we have $\Cusps_{R_\frn}^\Delta=\Cusps_{R_0}^\Delta\times_{R_0}\Spec(R_\frn)$.

For the assertion (\ref{GammaDeltaCuspDlog}), by a base change it is enough to show it over $R_\frn$. Put $e=\deg(f_b)$ and $G_{f_b}(w)=w^{q^e}-x H(w)$. 
Then we have $H(w)dx=x f_b w^{q^e-2} dw$ in $\Omega^1_{R_\frn[[w]]/R_\frn}$ and
\[
\frac{dw}{w^2}=\frac{H(w)}{f_b}\frac{dx}{xw^{q^e}}=\frac{1}{f_b}\frac{dx}{x^2},
\]
which concludes the proof.
\end{proof}

On the component of $\widehat{\Cusps}_{R_\frn}^\Delta$ corresponding to $\Xi=(a,b)$, the pull-back of $\TDt(\Lambda)$ agrees with $\TD(f_b \Lambda)$ over $R_\frn((w))$ with a universal $\Gamma_1^\Delta(\frn)$-structure $(\lambda,[\mu])$. Let us describe them explicitly. We set $T'_\frn=\Spec(R_\frn((w)))$, and consider the ring $R_\frn((w))$ as a subring of $B_\frn$ as in the proof of Lemma \ref{GammaDeltaExpDesc}. Put
\[
(P_\Xi,Q_\Xi)=(e_{f_b\Lambda}(\zeta),e_{f_b\Lambda}(\eta))(k_\Xi,l_\Xi)^{-1}.
\]
Then we have $Q_\Xi\in \TD(f_b\Lambda)[\frn](T'_\frn)$ and
\begin{equation}\label{EqnUnivLambda}
\lambda: C[\frn](T'_\frn)\to \TD(f_b\Lambda)[\frn](T'_\frn),\quad \zeta\mapsto Q_\Xi.
\end{equation}
On the other hand, taking the determinant as in the proof of Lemma \ref{GammaDeltaNonCan} yields
\begin{align*}
C[\frn]\otimes (\TD(f_b\Lambda)[\frn]/\Img(\lambda)) &\to \bigwedge^2 \TD(f_b\Lambda)[\frn] \\
\zeta \otimes (P_\Xi\bmod \Img(\lambda))&\mapsto Q_\Xi\wedge P_\Xi
\end{align*}
and similarly for $\lambda_{\infty,\frn}^{f_b\Lambda}$. Since $\det(k_\Xi,l_\Xi)=1$, we obtain an isomorphism
\[
\iota: \cH_\infty|_{T'_\frn}\to \TD(f_b\Lambda)[\frn]/\Img(\lambda)
\]
defined by $e_{f_b\Lambda}(\eta)\bmod \Img(\lambda_{\infty,\frn}^{f_b\Lambda})\mapsto -P_\Xi\bmod \Img(\lambda)$.
Then we have $\mu=\iota\circ\mu_{\infty,\frn}^{f_b\Lambda}$, which is given by
\begin{equation}\label{EqnUnivMu}
\mu: \underline{A/(\frn)}\to \TD(f_b\Lambda)[\frn]/\Img(\lambda),\quad 1\mapsto -P_\Xi\bmod \Img(\lambda).
\end{equation}

\begin{cor}\label{Ample}
Suppose that $R_0$ is a flat $A_\frn$-algebra which is an excellent regular domain.
Let $g$ be the common genus of the fibers of $X_1^\Delta(\frn)_{R_0}$ over $R_0$. Then, on each fiber, the invertible sheaf $(\bar{\omega}^\Delta_\univ)^{\otimes 2}$ has degree no less than $2g$.
\end{cor}
\begin{proof}
Since the map $Y(\frn)\to Y_1^\Delta(\frn)$ is etale, \cite[Theorem 6.11]{Gek_dR} implies that the dual of the Kodaira-Spencer map for the universal Drinfeld module $E^\Delta_\univ$ over $Y_1^\Delta(\frn)_{R_0}$
\[
\KS^\vee: \omega_{E^\Delta_\univ}\otimes \omega_{(E^\Delta_\univ)^D}\to \Omega^1_{Y_1^\Delta(\frn)_{R_0}/R_0}
\]
is an isomorphism. We write as $\Phi^{E^\Delta_\univ}_t=\theta+ A_1 \tau+ A_2 \tau^2$.
Since we have the isomorphism
\begin{equation}\label{EqnWeakAutoDual}
\omega_{E^\Delta_\univ}^{\otimes q-1}\to \omega_{(E^\Delta_\univ)^D}^{\otimes q-1},\quad l\mapsto l\otimes A_2^{\otimes -1},
\end{equation}
the map $(\KS^\vee)^{\otimes q-1}$ induces an isomorphism
\begin{equation}\label{EqnWeakAutoDualPower}
\omega_{E^\Delta_\univ}^{\otimes q-1}\otimes \omega_{E^\Delta_\univ}^{\otimes q-1} \to (\Omega^1_{Y_1^\Delta(\frn)_{R_0}/R_0})^{\otimes q-1}.
\end{equation}

Consider the cusp corresponding to $\Xi=(a,b)$ and the pull-back of this map to $R_\frn((w))$. Since $R_\frn((w))$ is a domain, the isomorphism $E_\univ^\Delta|_{R_\frn((w))}\to \TD(f_b\Lambda)$ is $R_\frn((w))$-linear.
Using Theorem \ref{DualDM} (\ref{DualDMLin}) and the functoriality of $\KS$, we can show that the pull-back of (\ref{EqnWeakAutoDualPower}) is identified with a similar map 
\[
\omega_{\TD(f_b\Lambda)}^{\otimes q-1}\otimes \omega_{\TD(f_b\Lambda)}^{\otimes q-1} \to (\Omega^1_{R_\frn((w))/R_\frn})^{\otimes q-1}
\]
induced by $\KS^\vee$ for $\TD(f_b\Lambda)$ over $R_\frn((w))$.
By Lemma \ref{TDKS} and (\ref{EqnWeakAutoDual}), this map is given by
\[
(d X)^{\otimes q-1}\otimes (d X)^{\otimes q-1}\mapsto  a_2^{-1}(\frac{d a_1}{d x}-\frac{a_1}{a_2}\frac{d a_2}{d x})^{q-1} (d x)^{\otimes q-1}.
\]
Since the right-hand side is an element of $R_\frn[[w]]^\times (\frac{d x}{x^2})^{\otimes q-1}$, Corollary \ref{GammaDeltaCusp} (\ref{GammaDeltaCuspDlog}) implies that the isomorphism (\ref{EqnWeakAutoDualPower}) extends to an isomorphism
\[
(\bar{\omega}^\Delta_\univ)^{\otimes q-1}\otimes (\bar{\omega}^\Delta_\univ)^{\otimes q-1} \to (\Omega^1_{X_1^\Delta(\frn)_{R_0}/R_0}(2\Cusps^{\Delta}_{R_0}))^{\otimes q-1}.
\]
Since $\Cusps^{\Delta}_{R_0}$ is non-empty, the corollary follows.
\end{proof}



\subsection{Case of level $\Gamma_1^\Delta(\frn,\wp)$}\label{SubsecCuspP}

For the structure around cusps of $X_1^\Delta(\frn,\wp)$, we first note that $Y_1^\Delta(\frn,\wp)_{R_0}$ is normal near infinity in the sense of \cite[(8.6.2)]{KM} by Lemma \ref{NormalNearInfty}. Thus the description around cusps using Tate-Drinfeld modules and normalization as in the beginning of \S\ref{SubsecCuspI} is also valid in this case.

The closed immersion $\lambda_{\infty,\wp}^\Lambda:C[\wp]\to \TD(\Lambda)$ defines a $\Gamma_0(\wp)$-structure on $\TD(\Lambda)$ over $T_0$. Hence we also have a map
\[
x^{\Delta,\wp}_\infty: \cT_0 \to X^\Delta_1(\frn,\wp)_{R_0}
\]
and a point $P_\infty^{\Delta,\wp}\in X^\Delta_1(\frn,\wp)_{R_0}$.

More generally, for any $\Xi=(a,b)\in \sH$, consider the map $R_0((x))\to R_0((w))=R_0((x))[w]/(G_{f_b}(w))$ and the Tate-Drinfeld module $\TD(f_b\Lambda)$ over $R_0((w))$. The latter has a canonical $\Gamma_0(\wp)$-structure $\cC$ given by the closed immersion $\lambda_{\infty,\wp}^{f_b\Lambda}$ of Lemma \ref{TDGamma1}. We denote by $Z=[\Gamma_0(\wp)]_{\TD(f_b\Lambda)}$ the scheme representing the functor sending each scheme $T$ over $R_0((w))$ to the set of $\Gamma_0(\wp)$-structures on $\TD(f_b\Lambda)|_T$. It is finite over $R_0((w))$ and thus Noetherian. We denote by $\cG_\univ$ the universal $\Gamma_0(\wp)$-structure on $Z$.

For any Noetherian scheme $T$ over $R_0((w))$ and any $\Gamma_0(\wp)$-structure $\cG$ on $\TD(f_b\Lambda)|_T$, the theory of Hilbert schemes shows that the functor $\sHom_{T,A}(\cG,\underline{A/(\wp)})$ is representable, locally of finite presentation and separated over $T$. From the etaleness of $\underline{A/(\wp)}$, we see that the group scheme $\sHom_{T,A}(\cG,\underline{A/(\wp)})$ is also formally etale over $T$. Hence it is etale over $T$ and thus its zero section is a closed and open immersion. We write its complement as $U_T$.

By composing with $\pi_{\infty,\wp}^{f_b\Lambda}: \TD(f_b\Lambda)[\wp]\to \underline{A/(\wp)}$, the universal $\Gamma_0(\wp)$-structure $\cG_\univ$ gives a map
\[
Z=[\Gamma_0(\wp)]_{\TD(f_b\Lambda)}\to \sHom_{Z,A}(\cG_\univ,\underline{A/(\wp)})=Z\sqcup U_Z.
\]
Hence the left-hand side is decomposed accordingly, and the component over $Z$ agrees with the section $\Spec(R_0((w)))\to Z$ given by $\cC$. From this, we can show that we have the same description of the complete local ring at $P_\infty^{\Delta,\wp}\in X^\Delta_1(\frn,\wp)_{R_0}$ and a similar extended invertible sheaf $\bar{\omega}_\univ^{\Delta,\wp}$ which is compatible with $\bar{\omega}_\univ^\Delta$, as in Lemma \ref{GammaDeltaDesc}. Furthermore, after passing to $R_\frn((w))$, we can also show that the formal completion of $X_1^\Delta(\frn,\wp)_{R_\frn}$ along the cusp corresponding to $\cC$ over the component of $\Xi$ is isomorphic to $R_\frn[[w]]$ via the projection to $X_1^\Delta(\frn)_{R_\frn}$. We refer to this cusp as the unramified cusp over $\Xi$.



\subsection{Canonical subgroups of Tate-Drinfeld modules}

In this subsection we consider the case $R_0=\cO_K$. Thus we have the Tate-Drinfeld module $\TD(f\Lambda)$ over the ring $\cO_K((x))$. Put $d=\deg(\wp)$ as before. We denote the normalized $\wp$-adic valuation of $\cO_K((x))$ by $v_\wp$.

\begin{lem}\label{TDOrd}
The Tate-Drinfeld module $\TD(f\Lambda)$ over $\cO_K((x))$ has ordinary reduction.
\end{lem}
\begin{proof}
Put $\bar{\Phi}^{f\Lambda}_\wp(X)=\Phi^{f\Lambda}_\wp(X)\bmod \wp$, which is an element of the ring $k(\wp)[[x]][X]$. From (\ref{EqnTDRed}), we see that the coefficient of $X^{q^d}$ in $\bar{\Phi}^{f\Lambda}_\wp(X)$ is an $x$-adic unit and those of larger degree have positive $x$-adic valuations. By Lemma \ref{TDCoeff}, the coefficient of $X^{q^{2d}}$ is non-zero. An inspection of the Newton polygon of $\bar{\Phi}^{f\Lambda}_\wp(X)$ shows that this polynomial has at least $q^{2d}-q^d$ non-zero roots in an algebraic closure of $k(\wp)((x))$. Thus the reduction of $\TD(f\Lambda)$ modulo $\wp$ is ordinary.
\end{proof}

By the exact sequence (\ref{TDfrnExact}), the closed immersion $\lambda^{f\Lambda}_{\infty,\wp^n}$ identifies $C[\wp^n]$ with a closed $A$-submodule scheme of $\TD(f\Lambda)[\wp^n]$, which we denote by $\cC_n^{f\Lambda}$. We refer to $\cC_n^{f\Lambda}$ as the canonical subgroup of $\TD(f\Lambda)$ of level $n$. The reduction $\overline{\cC_n^{f\Lambda}}$ modulo $\wp$ agrees with $\Ker(F_d^n)$ of the reduction of $\TD(f\Lambda)$. Thus the pull-backs of $\cC_n^{f\Lambda}$ to $(\cO_K/(\wp^m))((x))$ and the $\wp$-adic completion $\cO_K((x))^\wedge$ agree with the canonical subgroups of level $n$ of $\TD(f\Lambda)$ over them in the sense of Lemma \ref{ExistCanSub}.
We have $\cC_n^{f\Lambda}=\nu_f^*(\cC_n^{\Lambda})$. 

We define the canonical and canonical etale isogenies of level one for $\TD(f\Lambda)$ as the natural maps
\[
\pi^{f\Lambda}: \TD(f\Lambda)\to \TD(f\Lambda)/\cC_1^{f\Lambda},\quad \rho^{f\Lambda}: \TD(f\Lambda)/\cC_1^{f\Lambda}\to \TD(f\Lambda).
\]
They satisfy
\begin{equation}\label{EqnTDFVCompos}
\rho^{f\Lambda}\circ \pi^{f\Lambda}=\Phi^{\TD(f\Lambda)}_{\wp},\quad \pi^{f\Lambda}\circ \rho^{f\Lambda}=\Phi^{\TD(f\Lambda)/\cC_1^{f\Lambda}}_{\wp}.
\end{equation}
By Lemma \ref{DMQFF} (\ref{DMQFF-Dr}), the quotient $\TD(f\Lambda)/\cC_1^{f\Lambda}$ has a natural structure of a Drinfeld module of rank two which makes these isogenies compatible with $A$-actions. 
The $\Gamma_1^\Delta(\frn)$-structure $(\lambda_{\infty,\frn}^{f\Lambda},[\mu_{\infty,\frn}^{f\Lambda}])$ on $\TD(f\Lambda)$ induces that on $\TD(f\Lambda)/\cC_1^{f\Lambda}$, which we denote by $(\bar{\lambda}_{\infty,\frn}^{f\Lambda},[\bar{\mu}_{\infty,\frn}^{f\Lambda}])$.

Since the power series $e_{f\Lambda}(X)\in \cO_K[[x]][[X]]$ is entire, any root $\beta\neq 0$ of $\Phi^C_\wp(Z)$ in its splitting field $L$ over $K$ defines an element $e_{f\Lambda}(\beta)$ of $\cO_L[[x]]$. From Lemma \ref{CarlitzRed} and (\ref{EqnDefExp}) we obtain
\begin{equation}\label{EqnValBeta}
\Phi^C_\wp(\beta)=0,\beta\neq 0 \Rightarrow e_{f\Lambda}(\beta)\in \beta (\cO_L[[x]]^\times).
\end{equation}
Then we put
\[
\Psi^{f\Lambda}_\wp(X)=\wp X\prod_{\Phi^C_\wp(\beta)=0,\beta\neq 0}\left( 1-\frac{X}{e_{f\Lambda}(\beta)}\right)\in \cO_K[[x]][X].
\]
As in the proof of \cite[Ch.~2, Lemma 1.2]{Lehmkuhl}, we see that this is an $\bF_q$-linear additive polynomial, and 
(\ref{EqnValBeta}) implies that its leading coefficient is an element of $\cO_K[[x]]^\times$. Hence $X\mapsto 
\Psi^{f\Lambda}_\wp(X)$ defines an isogeny of $\bF_q$-module schemes over $\cO_K((x))$
\[
\pi^{f\Lambda}_\wp : \TD(f\Lambda) \to \TD(f\Lambda).
\]

\begin{lem}\label{KerLambda}
$\Ker(\pi^{f\Lambda}_\wp)=\cC_1^{f\Lambda}$.
\end{lem}
\begin{proof}
	By comparing ranks, it is enough to show that the composite $\pi^{f\Lambda}_\wp \circ \lambda_{\infty,\wp}^{f\Lambda}$ is zero. From the definition of the map $\lambda_{\infty,\wp}^{f\Lambda}$, this amounts to showing that the image of $\Psi^{f\Lambda}_\wp(e_{f\Lambda}(Z))$ in the ring $\cO_K[[x]]\langle Z\rangle/(\Phi^C_\wp(Z))$
	is zero.
	For this, note that we have the equality of entire series in $K((x))\{\{Z\}\}$
	\begin{equation}\label{EqnExpP}
	\Psi^{f\Lambda}_\wp(e_{f\Lambda}(Z))=e_{\wp f\Lambda}(\Phi^C_\wp(Z)),
	\end{equation}
	since they have the same linear term $\wp$ and divisor $f\Lambda+(\Phi^C_\wp)^{-1}(0)$. Thus the equality also holds in $\cO_K[[x]]\langle Z\rangle$. Since the latter ring is Noetherian, the ideal $(\Phi^C_\wp(Z))$ is $x$-adically closed and thus it contains the element $e_{\wp f\Lambda}(\Phi^C_\wp(Z))$. 
\end{proof}

Thus the $a$-multiplication map of $\TD(f\Lambda)/\cC_1^{f\Lambda}$ for any $a\in A$ is given by a unique polynomial $\Phi'_a(X)$ satisfying
\[
\Psi^{f\Lambda}_\wp(\Phi^{f\Lambda}_a(X))=\Phi'_a(\Psi^{f\Lambda}_\wp(X)).
\]

Note that we have
\[
F_\wp(x)=\frac{1}{\Phi^C_\wp\left(\frac{1}{x}\right)}\in x^{q^d} (1+\wp x\cO_K[[x]]).
\]
We also have the $\cO_K$-algebra homomorphism
\[
\nu^\sharp_\wp: \cO_K((x)) \to \cO_K((x)),\quad x\mapsto F_\wp(x)
\]
and the induced map $\nu_\wp:\Spec(\cO_K((x)))\to \Spec(\cO_K((x)))$. For any element $F(X)=\sum_{l\geq 0} a_l X^l\in \cO_K((x))[[X]]$, we put $\nu^*_\wp(F)(X)=\sum_{l\geq 0} \nu^\sharp_\wp(a_l) X^l$, as before. 
Then we have 
\begin{equation}\label{EqnExpWpCompat}
\nu^*_\wp(e_{f\Lambda})(X)=e_{\wp f\Lambda}(X).
\end{equation}
Thus (\ref{EqnDefTDa}) yields
\begin{equation}\label{EqnExpPLa}
\nu_\wp^*(\Phi^{f\Lambda}_a)(e_{\wp f\Lambda}(X))=e_{\wp f\Lambda}(\Phi^C_a(X))
\end{equation}
for any $a\in A$. On the other hand, (\ref{EqnExpP}) and (\ref{EqnExpWpCompat}) yield
\begin{equation}\label{EqnPLa}
\Psi^{f\Lambda}_\wp(e_{f\Lambda}(X))=e_{\wp f\Lambda}(\Phi^C_\wp(X))=\nu^*_\wp(e_{f\Lambda})(\Phi^C_\wp(X)).
\end{equation}

\begin{lem}\label{TDQuot}
\[
(\TD(f\Lambda)/\cC^{f\Lambda}_1,\bar{\lambda}_{\infty,\frn}^{f\Lambda},[\bar{\mu}_{\infty,\frn}^{f\Lambda}])=\nu_\wp^*(\TD(f\Lambda),\wp \lambda_{\infty,\frn}^{f\Lambda},[\mu_{\infty,\frn}^{f\Lambda}]).
\]
\end{lem}
\begin{proof}
First let us show the equality $\nu_\wp^*(\TD(f\Lambda))=\TD(f\Lambda)/\cC_1^{f\Lambda}$. This amounts to showing 
\[
\Psi_\wp^{f\Lambda}(\Phi^{f\Lambda}_a(X))=\nu_\wp^*(\Phi^{f\Lambda}_a)(\Psi^{f\Lambda}_\wp(X))
\]
for any $a\in A$. It is enough to show the equality in the ring $K((x))[[X]]$. For this, (\ref{EqnDefTDa}), (\ref{EqnExpPLa}) and (\ref{EqnPLa}) yield
\begin{align*}
\Psi_\wp^{f\Lambda}(\Phi^{f\Lambda}_a(e_{f\Lambda}(X)))&=\Psi^{f\Lambda}_\wp(e_{f\Lambda}(\Phi^C_a(X)))=e_{\wp f\Lambda}(\Phi^C_{\wp a}(X))\\
&=\nu_\wp^*(\Phi^{f\Lambda}_a)(e_{\wp f\Lambda}(\Phi^C_{\wp}(X)))=\nu_\wp^*(\Phi^{f\Lambda}_a)(\Psi^{f\Lambda}_\wp(e_{f\Lambda}(X)))
\end{align*}
and the claim follows by plugging in $e_{f\Lambda}^{-1}(X)$. The $\Gamma_1(\frn)$-structure $\bar{\lambda}^{f\Lambda}_{\infty,\frn}$ is given by $X\mapsto \Psi^{f\Lambda}_\wp(e_{f\Lambda}(Z))$. By (\ref{EqnPLa}), the latter element is equal to $\nu_\wp^*(e_{f\Lambda})(\Phi^C_\wp(Z))$, which means $\bar{\lambda}^{f\Lambda}_{\infty,\frn}=\nu_\wp^*(\wp\lambda^{f\Lambda}_{\infty,\frn})$.

For the assertion on $[\bar{\mu}_{\infty,\frn}^{f\Lambda}]$, consider the ring $B_{0,\frn}^{f\Lambda}$ of (\ref{EqnDefBn}) and its base extension $\nu_\wp^*(B_{0,\frn}^{f\Lambda})$ by the map $\nu_\wp^\sharp$. These rings are free of rank $q^{\deg(\frn)}$ over $\cO_K((x))$. We have a homomorphism of $\cO_K((x))$-algebras
\[
\nu_\wp^*(B_{0,\frn}^{f\Lambda})\simeq \cO_K((x))[\eta]/(\Phi^C_\frn(\eta)-\Phi^C_{f\wp}(1/x))\to B_{0,\frn}^{f\Lambda}
\]
defined by $\eta\mapsto \Phi^C_\wp(\eta)$.
Since $(\wp,\frn)=1$, we have $\alpha\wp+\beta \frn=1$ for some $\alpha,\beta\in A$ and this map sends $\Phi^C_\alpha(\eta)+\Phi^C_{f \beta}(1/x)$ to $\eta$. Hence it is surjective and thus these two rings are isomorphic as $\cO_K((x))$-algebras.

Now a similar argument as above implies that, for the map $\bar{\mu}_{\infty,\frn}^{f\Lambda}: \underline{A/(\frn)}\to \nu_\wp^*(\cH_{\infty,\frn}^{f\Lambda})$, the restriction $\bar{\mu}_{\infty,\frn}^{f\Lambda}(1)|_{B_{0,\frn}^{f\Lambda}}$ is equal to the image of the element
\[
\nu_\wp^*(e_{f\Lambda})(\Phi^C_\wp(\eta))\in \nu_\wp^*(\TD(f\Lambda))(B_{0,\frn}^{f\Lambda}).
\]
On the other hand, for the pull-back $\nu_\wp^*(\mu_{\infty,\frn}^{f\Lambda}): \underline{A/(\frn)}\to \nu_\wp^*(\cH_{\infty,\frn}^{f\Lambda})$, the restriction $\nu_\wp^*(\mu_{\infty,\frn}^{f\Lambda})(1)|_{\nu_\wp^*(B_{0,\frn}^{f\Lambda})}$ is equal to the image of the element
\[
e_{f\Lambda}(\eta)\otimes 1 =\nu_\wp^*(e_{f\Lambda})(\eta)\in \nu_\wp^*(\TD(f\Lambda))(\nu_\wp^*(B_{0,\frn}^{f\Lambda})).
\]
Since they agree with each other in $\nu_\wp^*(\TD(f\Lambda))(B_{0,\frn}^{f\Lambda})$, we obtain $\bar{\mu}_{\infty,\frn}^{f\Lambda}=\nu_\wp^*(\mu_{\infty,\frn}^{f\Lambda})$.
\end{proof}

By Lemma \ref{TDQuot}, the canonical etale isogeny $\rho^{f\Lambda}$ induces an isomorphism of $\cO_K((x))$-modules
\[
((\rho^{f\Lambda})^*)^{-1}: \omega_{\TD(f\Lambda)}\otimes_{\cO_K((x)),\nu^\sharp_\wp} \cO_K((x))=\omega_{\TD(f\Lambda)/\cC^{f\Lambda}_1} \to \omega_{\TD(f\Lambda)}.
\]

\begin{cor}\label{TDWcanFixed}
\[
((\rho^{f\Lambda})^*)^{-1}(d X\otimes 1)=d X.
\]
\end{cor}
\begin{proof}
Since we have shown that the canonical isogeny $\pi^{f\Lambda}$ of level one for $\TD(f\Lambda)$ is given by $X\mapsto \Psi^{f\Lambda}_\wp(X)$, we have $(\pi^{f\Lambda})^*(dX)=\wp dX$. From (\ref{EqnTDFVCompos}), we obtain $(\rho^{f\Lambda})^*(dX)=dX$ in $\omega_{\TD(f\Lambda)/\cC^{f\Lambda}_1}$, which is identified with $dX\otimes 1$ via $\nu_\wp^*(\TD(f\Lambda))=\TD(f\Lambda)/\cC^{f\Lambda}_1$.
\end{proof}



\section{$\wp$-adic properties of Drinfeld modular forms}\label{Sec_DMF}

\subsection{Drinfeld modular forms}

Let $k$ be an integer. Let $M$ be an $A_\frn$-module. We define a Drinfeld modular form of level $\Gamma_1^\Delta(\frn)$ and weight $k$ with coefficients in $M$ as an element of 
\[
M_k(\Gamma_1^\Delta(\frn))_{M}=H^0(X_1^\Delta(\frn)_{A_\frn},(\bar{\omega}^\Delta_\univ)^{\otimes k}\otimes_{A_\frn}M).
\]
By Lemma \ref{GammaDeltaDesc} (\ref{GammaDeltaDescCompat}), the group $\bF_q^\times$ acts on the $A_\frn$-module $M_k(\Gamma_1^\Delta(\frn))_{M}$ via $c\mapsto \langle c \rangle_\Delta$. Since $q-1$ is invertible in $A_\frn$, we have a decomposition
\[
M_{k}(\Gamma^\Delta_1(\frn))_{M}=\bigoplus_{m\in\bZ/(q-1)\bZ} M_{k,m}(\Gamma_1(\frn))_{M},
\]
where the direct summand $M_{k,m}(\Gamma_1(\frn))_{M}$ is the maximal submodule on which the operator $\langle c\rangle_\Delta$ acts by the multiplication by $c^{-m}$ for any $c\in\bF_q^\times$. 
We say $f\in M_k(\Gamma_1^\Delta(\frn))_{M}$ is of type $m$ if $f\in M_{k,m}(\Gamma_1(\frn))_{M}$. 

Consider the map $x^\Delta_\infty:\Spec(A_\frn[[x]])\to X_1^\Delta(\frn)_{A_\frn}$ as in \S\ref{SubsecCuspI}.
For any $f\in M_k(\Gamma_1^\Delta(\frn))_{M}$, we define the $x$-expansion of $f$ at the $\infty$-cusp as the unique power series $f_\infty(x)\in A_\frn[[x]]\otimes_{A_\frn} M$ satisfying
\[
(x_\infty^\Delta)^*(f)=f_\infty(x) (d X)^{\otimes k}.
\]
We also have a variant $M_k(\Gamma_1^\Delta(\frn,\wp))_M$ of level $\Gamma_1^\Delta(\frn,\wp)$, using $X_1^\Delta(\frn,\wp)$, the sheaf $\bar{\omega}^{\Delta,\wp}_\univ$ and the $\infty$-cusp $x^{\Delta,\wp}_\infty$.

\begin{prop}\label{QexpBC}
\begin{enumerate}
\item\label{QexpBC-Q} ($x$-expansion principle) For any $A_\frn$-module $M$ and $f\in M_k(\Gamma_1^\Delta(\frn))_M$, if $f_\infty(x)=0$ then $f=0$. Moreover, for any $A_n$-modules $N\subseteq M$ and any $f\in M_k(\Gamma_1^\Delta(\frn))_M$, we have $f_\infty(x)\in A_\frn[[x]]\otimes_{A_\frn} N$ if and only if $f\in M_k(\Gamma_1^\Delta(\frn))_N$. The same assertions hold for the case of level $\Gamma_1^\Delta(\frn,\wp)$ if $M$ is an $A_\frn[1/\wp]$-module.
\item\label{QexpBC-BC} For any $k\geq 2$ and any $A_\frn$-module $M$, the natural map
\[
M_k(\Gamma^\Delta_1(\frn))_{A_\frn}\otimes_{A_\frn} M\to M_k(\Gamma^\Delta_1(\frn))_{M}
\]
is an isomorphism.
\end{enumerate}
\end{prop}
\begin{proof}
Since $X_1^\Delta(\frn)_{A_\frn}$ and $X_1^\Delta(\frn,\wp)_{A_\frn[1/\wp]}$ are smooth and geometrically connected, Krull's intersection theorem and Lemma \ref{GammaDeltaDesc} (and for the case of level $\Gamma_1^\Delta(\frn,\wp)$, the corresponding statements in \S\ref{SubsecCuspP}) imply the assertion (\ref{QexpBC-Q}), as in the proof of \cite[\S1.6]{Katz_p}. The assertion (\ref{QexpBC-BC}) follows from Corollary \ref{Ample}, similarly to the proof of \cite[Theorem 1.7.1]{Katz_p}.
\end{proof}

Note that our definition of Drinfeld modular forms is compatible with the classical one as in \cite{Gek_DMC,Gek_Coeff}; 
over $X(\frn)_{\bC_\infty}$ this follows from \cite[Theorem 1.79]{Goss_pi}, and the spaces of Drinfeld modular forms of 
level $\Gamma_1^\Delta(\frn)$ and weight $k$ in both definitions are the fixed parts of the natural action of 
$\Gamma_1^\Delta(\frn)/\Gamma(\frn)$ on them. We can also show that our $x$-expansion $f_\infty(x)$ of $f$ at the 
$\infty$-cusp agrees with Gekeler's $t$-expansion at $\infty$  (see \cite[Ch.~V, \S 2]{Gek_DMC}, while the 
normalization we adopt is as in \cite[\S5]{Gek_Coeff}) of the associated classical Drinfeld modular form to $f$.

By \cite[Proposition (6.11)]{Gek_Coeff} and Proposition \ref{QexpBC} (\ref{QexpBC-Q}), Gekeler's lift $g_d$ of the Hasse invariant is an element of $M_{q^d-1,0}(\Gamma_1(\frn))_{A_\frn}$ satisfying 
\begin{equation}\label{EqnGekHasseInv}
(g_d)_\infty(x)\equiv 1\bmod \wp.
\end{equation}



\subsection{Ordinary loci}

In the rest of the paper, we write as $Y_\univ=Y_1^\Delta(\frn)_{\cO_K}$ and $X_\univ=X_1^\Delta(\frn)_{\cO_K}$. For any positive integer $m$, the pull-back of any scheme $T$ over $\cO_K$ to $\cO_{K,m}=\cO_{K}/(\wp^m)$ is denoted by $T_m$.

Since we know that $X_{\univ,1}$ has a supersingular point \cite[Satz (5.9)]{Gek_Moduln}, the ordinary loci $X_{\univ,m}^\ord$ in $X_{\univ,m}$ and $Y_{\univ,m}^\ord$ in $Y_{\univ,m}$ are affine open subschemes of finite type over $\cO_{K,m}$. We put 
\[
B_{\univ,m}^\ord=\cO(Y_{\univ,m}^\ord).
\]
This is a flat $\cO_{K,m}$-algebra of finite type, and the collection $\{B_{\univ,m}^\ord\}_m$ forms a projective system of $\cO_K$-algebras with surjective transition maps. We define
\[
\hat{B}_{\univ}^\ord=\varprojlim_n B_{\univ,m}^\ord,\quad Y_{\univ}^\ord=\Spec(\hat{B}_{\univ}^\ord).
\]
Then we have $\hat{B}_{\univ}^\ord/(\wp^m)=B_{\univ,m}^\ord$ and $\hat{B}_{\univ}^\ord$ is flat over $\cO_K$. This implies that $\hat{B}_{\univ}^\ord$ is $\wp$-adically complete and topologically of finite type over $\cO_{K}$. Moreover, since $B_{\univ,1}^\ord$ is a regular domain, the ring $\hat{B}_{\univ}^\ord$ is reduced. Thus $\hat{B}_{\univ}^\ord$ is a reduced flat $\wp$-adic ring. On the other hand, we have a map $Y_{\univ}^\ord\to Y_\univ$ and we denote by $\cE_\univ^\ord$ the pull-back of the universal Drinfeld module to $\hat{B}_{\univ}^\ord$, which has ordinary reduction.

Now we can form the canonical subgroup $\cC_n=\cC_n(\cE_\univ^\ord)$ of level $n$ for $\cE_\univ^\ord$. As is seen in \S\ref{SubsecCanSub}, it has the $v$-structure induced from that of $\cE_\univ^\ord$, which is unique by Lemma \ref{SredV} (\ref{SredV-Red}). Lemma \ref{DualEtale} implies that its Taguchi dual $\cC_n^D$ is etale. We denote by $\cC_{n,m}$ the pull-back of $\cC_n$ to $Y_{\univ,m}^\ord$ endowed with the induced $v$-structure, and similarly for $(\cC_{n}^D)_m$. Then
the Taguchi dual $\cC_{n,m}^D$ of $\cC_{n,m}$ agrees with $(\cC_n^D)_m$ as a finite $v$-module and they are finite etale over $Y_{\univ,m}^\ord$.

\begin{lem}\label{CansubExtCusp}
The finite $v$-module $\cC_{n,m}^D$ over $Y_{\univ,m}^\ord$ extends to an etale finite $v$-module $\bar{\cC}_{n,m}^D$ over $X_{\univ,m}^\ord$ which is etale locally isomorphic to $\underline{A/(\wp^n)}$.
\end{lem}
\begin{proof}
Let $K_\frn$ be a splitting field of $\Phi^C_\frn(X)$ over $K$.
Consider the formal completion of $X_{\univ}|_{\cO_{K_\frn}}$ at the cusp corresponding to $\Xi=(a,b)$, which is isomorphic to $\Spec(\cO_{K_\frn}[[w]])$. 

We denote the $\wp$-adic completion of $\cO_{K_\frn}((w))$ by $\cO$, which is a reduced flat $\wp$-adic ring. The pull-back of $\cE_\univ^\ord$ to $\cO$ is isomorphic to that of the Tate-Drinfeld module $\TD(f_b\Lambda)$ over $\cO_{K_\frn}((w))$ to $\cO$. By the uniqueness of the canonical subgroup in Lemma \ref{ExistCanSub}, we have $\cC_n|_\cO\simeq \cC_n^{f_b\Lambda}|_\cO=C[\wp^n]$. Lemma \ref{SredV} (\ref{SredV-Red}) implies that this identification is compatible with $v$-structures, where we give $C[\wp^n]$ the induced $v$-structure from $C$. Taking modulo $\wp^m$, we obtain an isomorphism $\cC_{n,m}|_{\cO_{K_\frn,m}((w))}\simeq C[\wp^n]$ of $v$-modules over $\cO_{K_\frn,m}((w))$.

This implies that, by an fpqc descent, the finite $v$-module $\cC_{n,m}$ extends to a finite $v$-module $\bar{\cC}_{n,m}$ over $X_{\univ,m}^\ord$ such that its restriction to the formal completion at each cusp is isomorphic to $C[\wp^n]$ with the induced $v$-structure from $C$. Taking the dual yields the lemma.
\end{proof}

\begin{lem}\label{MonodromySurj}
Let $U$ be any non-empty open subscheme of $X_{\univ,m}^\ord$ and $\bar{\xi}$ any geometric point of $U$. Then the character of its etale fundamental group with base point $\bar{\xi}$
\[
r_{n,m} :\pi_1^\et(U)\to \pi_1^\et(X_{\univ,m}^\ord)\to (A/(\wp^n))^\times
\]
defined by $\bar{\cC}_{n,m}^D$ is surjective.
\end{lem}
\begin{proof}
We may assume $m=1$. Let $L$ be the function field of $X_{\univ,1}$. As in \cite[Theorem 4.3]{Katz_p}, it is enough to show that the restriction of $r_{n,1}$ to the inertia subgroup of $\Gal(L^\sep/L)$ at a supersingular point is surjective.

Take $\xi'\in X_{\univ,1}$ corresponding to a supersingular Drinfeld module over an algebraic closure $k$ of $k(\wp)$. The complete local ring of $X_{\univ,1}\times_{k(\wp)} k$ at $\xi'$ is isomorphic to $k[[u]]$. Let $\bE$ be the restriction of $\cE_\univ$ to this complete local ring. 
By \cite[Remark 3.15]{Shastry}, we have $\Lie(V_{d,\mathbb{E}})=-u$ and the restriction $\bE|_{k((u))}$ to the generic fiber is ordinary. By Theorem \ref{DualDM} (\ref{DualDMIsog}) and Lemma \ref{DualFV}, we have $\cC_n(\bE|_{k((u))})^D=\Ker(V^n_{d,\bE^D|_{k((u))}})$. Here $\bE^D|_{k((u))}$ is the dual of $\bE|_{k((u))}$, which is also ordinary by Proposition \ref{DualOrd}. Hence it suffices to show that the finite etale $A$-module scheme $\Ker(V^n_{d,\bE^D|_{k((u))}})$ defines a totally ramified extension of $k((u))$ of degree $\sharp(A/(\wp^n))^\times$.

For this, Proposition \ref{DualOrd} also implies that the map $\Lie(V_{d,\bE^D})$ is the multiplication by an element of $k[[u]]$ with normalized $u$-adic valuation one. Let $v_u$ be the normalized $u$-adic valuation on $k((u))$ and we extend it to its algebraic closure $k((u))^{\alg}$. Since the fiber of $\bE^D$ at $u=0$ is also supersingular, the map $V_{d,\bE^D}$ can be written as
\[
V_{d,\bE^D}(X)=a_0 X+\cdots+a_d X^{q^d}
\]
with some $a_i\in k[[u]]$ satisfying $v_u(a_0)=1$, $v_u(a_i)\geq 1$ for $1\leq i<d$ and $v_u(a_d)=0$. Then an inspection of the Newton polygon shows that any non-zero root $z$ of $V_{d,\bE^D}(X)$ satisfies $v_u(z)=1/(q^d-1)$ and there exists a root $z'$ of $V_{d,\bE^D}^n(X)$ with $v_u(z')=1/((q^d-1)q^{d(n-1)})=\sharp(A/(\wp^n))^\times$. This concludes the proof.
\end{proof}

Consider the quotient $\cE_\univ^\ord/\cC_1$ over $Y_{\univ}^\ord$, which has a natural structure of a Drinfeld module of rank two by Lemma \ref{DMQFF} (\ref{DMQFF-Dr}). Since the universal $\Gamma_1^\Delta(\frn)$-structure on $\cE_\univ$ induces that on $\cE_\univ^\ord/\cC_1$, we have a corresponding map $\pi_d: Y_{\univ}^\ord \to Y_\univ$. Since $\cE_\univ^\ord/\cC_1$ has ordinary reduction, the induced map $Y_{\univ,m}^\ord\to Y_{\univ,m}$ factors through $Y_{\univ,m}^\ord$. Hence $\pi_d$ also factors as $\pi_d: Y_{\univ}^\ord \to Y_{\univ}^\ord$. On the other hand, the endomorphism $\langle\wp^{-1} \rangle_\frn$ of $X_\univ$ defines endomorphisms of $X_\univ^\ord$ and $Y_\univ^\ord$, which we also denote by $\langle\wp^{-1} \rangle_{\frn}$. Put 
\[
\varphi_d=\langle\wp^{-1} \rangle_{\frn}\circ \pi_d.
\]
This gives the cartesian diagram
\[
\xymatrix{
\cE_\univ^\ord/\cC_1 \ar[r]\ar[d] & \cE_\univ^\ord\ar[r]\ar[d] & \cE_\univ^\ord\ar[d]\\
Y_{\univ}^\ord\ar[r]_{\pi_d} & Y_{\univ}^\ord\ar[r]_{\langle\wp^{-1} \rangle_{\frn}} & Y_{\univ}^\ord.
}
\]

\begin{lem}\label{PhiExt}
For any positive integer $m$, the induced map $\varphi_d: Y_{\univ,m}^\ord \to Y_{\univ,m}^\ord$ extends to $\bar{\varphi}_d: X_{\univ,m}^\ord \to X_{\univ,m}^\ord$ which is compatible with respect to $m$. Moreover, $\bar{\varphi}_d$
agrees with the $q^d$-th power Frobenius map on $X_{\univ,1}^\ord$.
\end{lem}
\begin{proof}
Let $K_\frn$ be a splitting field of $\Phi^C_\frn(X)$, as before. Put $\cO_{K_\frn,m}=\cO_{K_\frn}/(\wp^m)$. 
By an fpqc descent, it suffices to show the existence of an extension as in the lemma around each cusp over $\cO_{K_\frn,m}$. For this, first note that the automorphism of $\Spec(\cO_K[X]/(W_\frn(X)))$ given by $X\mapsto \Phi^C_\wp(X)$ preserves its connected components, since so does its restriction over $k(\wp)$ by Lemma \ref{CarlitzRed}. Hence we have an automorphism of $\Spec(\cO_{K_\frn})$ over $\cO_K$ defined by $\zeta\mapsto \Phi^C_\wp(\zeta)$, which we denote by $\sigma_\wp$. We define an endomorphism $\tilde{\nu}_\wp$ of $\Spec(\cO_{K_\frn}((w)))$ over $\cO_K$ by $\tilde{\nu}_\wp=\sigma_\wp\otimes \nu_\wp$.
 
Let $\Xi=(a,b)$ be any element of $\sH$ and $f_b$ the monic generator of $\Ann_A(b(A/(\frn)))$, as before. 
On the component defined by $\Xi$, we have the Tate-Drinfeld module $\TD(f_b\Lambda)$ over $\cO_{K_\frn}((w))$ endowed with a $\Gamma_1^\Delta(\frn)$-structure $(\lambda,[\mu])$. As in the proof of Lemma \ref{TDQuot}, using (\ref{EqnUnivLambda}) and (\ref{EqnUnivMu}) we see that the image of $(\lambda,[\mu])$ by the map $\TD(f_b\Lambda)\to \TD(f_b\Lambda)/\cC_1^{f_b\Lambda}$ can be identified with $\tilde{\nu}_\wp^*(\wp\lambda,[\mu])$. We denote by
\[
(\lambda,[\mu]): \Spec(\cO_{K_\frn,m}((w)))\to Y_{\univ,m}^\ord
\]
the map defined by the triple $(\TD(f_b\Lambda)|_{\cO_{K_\frn,m}((w))},\lambda,[\mu])$. Then we have the commutative diagram
\[
\xymatrix{
	Y_{\univ,m}^\ord \ar[r]^{\varphi_d} & Y_{\univ,m}^\ord \\
	\Spec(\cO_{K_\frn,m}((w))) \ar[r]_{\tilde{\nu}_\wp} \ar[u]^{(\lambda,[\mu])} & \Spec(\cO_{K_\frn,m}((w)))\ar[u]_{(\lambda,[\mu])},
}
\]
where the vertical arrows identify the lower term with the formal completion of $X_{\univ,m}^\ord$ at the cusp corresponding to $\Xi$ with $w$ inverted. Since we have $F_\wp(w)\in \cO_{K_\frn}[[w]]$, we obtain an extension of $\varphi_d$ to each cusp. 

Since the canonical subgroup $\cC_1$ is a lift of the Frobenius kernel, from (\ref{CarlitzModP}) we see that the morphism $\varphi_d:Y_{\univ,1}^\ord\to Y_{\univ,1}^\ord$ agrees with the $q^d$-th power Frobenius map.
Then the assertion on $X_{\univ,1}^\ord$ also follows, since it is integral and separated.
\end{proof}

We denote by $\omega_{\univ,m}^\ord$ and $\bar{\omega}_{\univ,m}^\ord$ the pull-backs of the sheaf $\bar{\omega}_\univ^\Delta$ to $Y_{\univ,m}^\ord$ and $X_{\univ,m}^\ord$, respectively.

\begin{prop}\label{OmegaExt}
Let $\rho_\univ: \cE_\univ^\ord/\cC_1\to  \cE_\univ^\ord$ be the canonical etale isogeny of $\cE_\univ^\ord$ over $Y_{\univ}^\ord$. Then the isomorphism of $\cO_{Y_{\univ,m}^\ord}$-modules
\[
F_{\omega_{\univ,m}^\ord}=(\rho_\univ^*)^{-1}: \varphi_d^*(\omega^\ord_{\univ,m})\simeq \omega_{(\cE_{\univ}^\ord/\cC_{1})_m}\to \omega^\ord_{\univ,m}
\]
extends to an isomorphism of $\cO_{X_{\univ,m}^\ord}$-modules
\[
F_{\bar{\omega}_{\univ,m}^\ord}:\bar{\varphi}_d^*(\bar{\omega}_{\univ,m}^\ord)\to  \bar{\omega}_{\univ,m}^\ord.
\]
\end{prop}
\begin{proof}
As in the proof of Lemma \ref{PhiExt}, it is enough to extend $F_{\omega_{\univ,m}^\ord}$ to each cusp over $\cO_{K_\frn,m}$. This follows from Corollary \ref{TDWcanFixed} and the construction of $\bar{\omega}^\Delta_\univ$.
\end{proof}



\subsection{Weight congruence}

First we give a version of the Riemann-Hilbert correspondence of Katz in our setting. Put $A_n=A/(\wp^n)$.

\begin{lem}\label{RHCorr}
Let $n$ be a positive integer. Let $S_n$ be an affine scheme which is flat over $A_n$ such that $S_1=S_n\times_{A_n}\Spec(A_1)$ is normal and connected. Let $\varphi_d:S_n\to S_n$ be a morphism over $A_n$ such that the induced map on $S_1$ agrees with the $q^d$-th power Frobenius map. We denote by $\pi_1^\et(S_n)$ the etale fundamental group for a geometric point of $S_n$. Then there exists an equivalence between the category $\Rep_{A_n}(S_n)$ of free $A_n$-modules of finite rank with continuous actions of $\pi_1^\et(S_n)$ and the category $F\text{-}\Crys^0(S_n)$ of pairs $(\cH,F_{\cH})$ consisting of a locally free $\cO_{S_n}$-module $\cH$ of finite rank 
and an isomorphism of $\cO_{S_n}$-modules $F_\cH:\varphi_d^*(\cH)\to \cH$.
\end{lem}
\begin{proof}
This follows by a verbatim argument as in the proof of \cite[Proposition 4.1.1]{Katz_p}. Here we sketch the argument for the convenience of the reader. For any object $M$ of $\Rep_{A_n}(S_n)$, let $T_n$ be a (connected) Galois covering of $S_n$ such that $\pi_1^\et(S_n)\to \Aut(M)$ factors through the Galois group $G(T_n/S_n)$ of it. By the etaleness, we can uniquely lift the $q^d$-th power Frobenius map on $T_1$ to a $\varphi_d$-equivariant endomorphism $\varphi_{T_n}$ of $T_n$ over $A_n$. 

We claim that the sequence
\begin{equation}\label{EqnExactPhiTnFixed}
\xymatrix{
0\ar[r] & A_n \ar[r] & \cO(T_n) \ar[r]^{\varphi_{T_n}-1} & \cO(T_n) 
}
\end{equation}
is exact. Indeed, since $T_n$ is flat over $A_n$ we may assume $n=1$, and in this case the claim follows since $\cO(T_1)$ is an integral domain.

We have an endomorphism on $M\otimes_{A_n} \cO(T_n)$ defined by $m\otimes f\mapsto m\otimes \varphi_{T_n}^*(f)$, and Galois descent yields an object $(\cH(M),F_{\cH(M)})$ of $F\text{-}\Crys^0(S_n)$. This defines a functor 
\[
\cH(-): \Rep_{A_n}(S_n)\to F\text{-}\Crys^0(S_n). 
\]
The exact sequence (\ref{EqnExactPhiTnFixed}) implies $(\cH(M)|_{T_n})^{\varphi_{T_n}-1}=M$ and thus the functor $\cH(-)$ is fully faithful.

We prove the essential surjectivity by induction on $n$. For $n=1$, it follows by applying the original result \cite[Proposition 4.1.1]{Katz_p} to the case where the extension $k/\bF_q$ there is $\bF_{q^d}/\bF_{q^d}$.  Suppose that the case of $n-1$ is valid. Let $(\cH,F_{\cH})$ be any object of $F\text{-}\Crys^0(S_n)$. By assumption, there exists a finite etale cover $T_{n-1}\to S_{n-1}$ such that $\cH|_{T_{n-1}}$ has an $F_{\cH}$-fixed basis $\bar{h}_1,\ldots,\bar{h}_r$.
By Hensel's lemma, we can lift $T_{n-1}$ to a finite etale cover $T_n\to S_n$. Take a lift $h_i$ of $\bar{h}_i$ to $\cH|_{T_n}$. We have
\[
F_\cH(h_1,\ldots,h_r)=(h_1,\ldots,h_r)(I+\wp^{n-1} N)
\]
for some matrix $N\in M_r(\cO(T_n))$. Then it is enough to solve the equation
\[
F_\cH((h_1,\ldots,h_r)(I+\wp^{n-1} N'))=(h_1,\ldots,h_r)(I+\wp^{n-1} N')
\]
over some finite etale cover of $T_n$. Since $\cO(T_n)$ is flat over $A_n$, the equation is equivalent to $N+\varphi_d(N')\equiv N'\bmod \wp$, from which the claim follows.
\end{proof}

\begin{cor}\label{ResOpenFF}
Let $U$ be any non-empty affine open subscheme of $S_n$. Note that, since $\varphi_d$ agrees with the $q^d$-th power Frobenius map on $S_1$, it induces a map $\varphi_d:U\to U$. Then the functor $F\text{-}\Crys^0(S_n)\to F\text{-}\Crys^0(U)$ defined by the restriction to $U$ is fully faithful.
\end{cor}
\begin{proof}
It follows from the fact that, since $S_1$ is normal and connected, the restriction functor $\Rep_{A_n}(S_n)\to \Rep_{A_n}(U)$ is fully faithful.
\end{proof}

By Lemma \ref{PhiExt}, $X_{\univ,m}^\ord$ satisfies the assumptions of Lemma \ref{RHCorr}.

\begin{prop}\label{HTTExtIsom}
By the equivalence of Lemma \ref{RHCorr}, the character 
\[
r_{n,n}:\pi_1^\et(X_{\univ,n}^\ord)\to A_n^\times
\]
of Lemma \ref{MonodromySurj} associated to $\bar{\cC}^D_{n,n}$ corresponds to the pair $(\bar{\omega}_{\univ,n}^\ord,F_{\bar{\omega}_{\univ,n}^\ord})$ of Proposition \ref{OmegaExt}.
\end{prop}
\begin{proof}
By Corollary \ref{ResOpenFF}, it is enough to show that the character of $\pi_1^\et(Y_{\univ,n}^\ord)$ associated to $\cC_{n,n}^D$ corresponds to the pair $(\omega_{\univ,n}^\ord,F_{\omega_{\univ,n}^\ord})$. By Proposition \ref{HTTIsom}, the Hodge-Tate-Taguchi map yields an isomorphism of invertible $\cO_{Y_{\univ,n}^\ord}$-modules
\[
\HTT: \cC_{n,n}^D\otimes_{\underline{A_n}}\cO_{Y_{\univ,n}^\ord}\to \omega_{\univ,n}^\ord.
\]
Note that, over any Galois covering $T_n\to Y_{\univ,n}^\ord$ trivializing $\cC_{n,n}^D$, the map $\HTT$ is compatible with Galois actions. Hence it suffices to show that this map is also compatible with Frobenius structures, where we consider $1\otimes \varphi_d$ on the left-hand side.

Since the natural map $B_{\univ,n}^\ord\to \cO_{K,n}((x))$ is injective, we reduce ourselves to showing that at the $\infty$-cusp the Hodge-Tate-Taguchi map over $\cO_{K,n}((x))$
\[
\HTT: \cC_n(\TD(\Lambda))_n^D\otimes_{\underline{A_n}} \cO_{\Spec(\cO_{K,n}((x)))}\to \omega_{\TD(\Lambda)}\otimes \cO_{K,n}((x))
\]
commutes with Frobenius structures. As is seen in the proof of Lemma \ref{CansubExtCusp}, the induced $v$-structure on $\cC_n(\TD(\Lambda))_n\simeq C[\wp^n]$ from $\cE_\univ^\ord$ agrees with that from $C$. By Lemma \ref{ConstDual} we have $\cC_n(\TD(\Lambda))_n^D\simeq \underline{A_n}$, and Lemma \ref{HTTCarlitz} implies that the isomorphism $\HTT$ is given by 
\begin{equation}\label{EqnHTTTriv}
\HTT(1)=((\lambda_{\infty,\wp^n}^{\Lambda})^*)^{-1}(d Z)=d X.
\end{equation}
Now the proposition follows from Corollary \ref{TDWcanFixed}.
\end{proof}

\begin{thm}\label{WeightCongr}
For $i=1,2$, let $f_i$ be an element of $M_{k_i}(\Gamma_1^\Delta(\frn))_{\cO_K}$. Suppose that their $x$-expansions at the $\infty$-cusp $(f_i)_\infty(x)$ satisfy the congruence
\[
(f_1)_\infty(x)\equiv (f_2)_\infty(x)\notequiv 0 \bmod \wp^n.
\]
Then we have 
\[
k_1\equiv k_2 \bmod (q^d-1)p^{l_p(n)},\quad l_p(n)=\min\{N\in \bZ\mid p^N\geq n\}.
\]
\end{thm}
\begin{proof}
By assumption, $f_1$ and $f_2$ do not vanish on a non-empty affine open subscheme $U$ of $X_{\univ,n}^\ord$ containing the $\infty$-cusp. Thus the quotient $f_1/f_2$ defines a nowhere vanishing section on $U$ of $(\bar{\omega}_{\univ,n}^\ord)^{\otimes k_1-k_2}$ with $x$-expansion at the $\infty$-cusp equal to one. Thus the restriction of $f_1/f_2$ to $\Spec(\cO_{K,n}((x)))$ around the $\infty$-cusp agrees with $(d X)^{k_1-k_2}$. By Corollary \ref{TDWcanFixed}, it is fixed by the restriction of the Frobenius map of $(\bar{\omega}_{\univ,n}^\ord)^{\otimes k_1-k_2}$. 
Since the natural map $B_{\univ,n}^\ord\to \cO_{K,n}((x))$ is injective, we see that the section $f_1/f_2$ on $U$ itself is fixed by the Frobenius map. Hence the restriction of the pair $((\bar{\omega}_{\univ,n}^\ord)^{\otimes k_1-k_2}, F_{\bar{\omega}_{\univ,n}^\ord}^{\otimes k_1-k_2})$ to $U$ is trivial. Then Corollary \ref{ResOpenFF} implies that the pair is trivial on $X_{\univ,n}^\ord$, and by Proposition \ref{HTTExtIsom} the $(k_1-k_2)$-nd tensor power of the character $r_{n,n}$ is trivial. Now Lemma \ref{MonodromySurj} shows that $k_1-k_2$ is divisible by the exponent of the group $(A/(\wp^n))^\times$, which equals $(q^d-1)p^{l_p(n)}$. This concludes the proof.
\end{proof}

Then Theorem \ref{MainWeight} follows by adding an auxiliary level of degree prime to $q-1$ and applying Theorem \ref{WeightCongr}.



\subsection{$\wp$-adic Drinfeld modular forms}

Let $\frX_{\univ}$ be the $\wp$-adic completion of $X_{\univ}=X_1^\Delta(\frn)$ and $\frX_{\univ}^\ord$ the formal open subscheme of $\frX_\univ$ on which the Gekeler's lift $g_d$ of the Hasse invariant is invertible. The latter is isomorphic to the $\wp$-adic completion of 
\begin{equation}\label{EqnOrdSym}
\Spec_{X_{\univ}}(\Sym((\bar{\omega}^\Delta_\univ)^{\otimes q^d-1})/(g_d-1)).
\end{equation}
Note that the reduction modulo $\wp^m$ of $\frX_\univ^\ord$ is equal to $X_{\univ,m}^\ord$. We see that $\frX_{\univ}^\ord$ is a Noetherian affine formal scheme by \cite[Corollaire 2.1.37]{Abbes}. 

Following \cite[Definition 3]{Goss_v}, we define the $\wp$-adic weight space $\bS$ as
\[
\bS=\bZ/(q^d-1)\bZ\times \bZ_p
\]
with the discrete topology on the first entry and the $p$-adic topology on the second entry. We embed $\bZ$ into it diagonally.

For any $\chi=(s_0,s_1)\in \bS$, we have a continuous endomorphism of $\cO_K^\times=\bF_{q^d}^\times\times (1+\wp \cO_K)$ defined by 
\[
x=(x_0,x_1)\mapsto x^\chi=x_0^{s_0}x_1^{s_1} 
\]
which preserves the subgroup $1+\wp^n\cO_K$. Composing it with the character $r_{n,n}:\pi_1^\et(X_{\univ,n}^\ord)\to A_n^\times=\cO_{K,n}^\times$, we obtain a character $r_{n,n}^{\chi}$. Let $\bar{\omega}_{\univ,n}^{\ord,\chi}$ be the associated invertible sheaf on $X_{\univ,n}^\ord$ via the correspondence of Lemma \ref{RHCorr}. Since they form a projective system with surjective transition maps, they give an invertible sheaf $\bar{\omega}_{\univ}^{\ord,\chi}$ on $\frX_{\univ}^\ord$ \cite[Proposition 2.8.9]{Abbes}. 

For any finite extension $L/K$, we put
\[
M_\chi(\Gamma_1^\Delta(\frn))_{\cO_L}:=H^0(\frX_{\univ}^\ord|_{\cO_{L}}, \bar{\omega}_{\univ}^{\ord,\chi}|_{\cO_{L}})=H^0(\frX_{\univ}^\ord, \bar{\omega}_{\univ}^{\ord,\chi}\otimes_{\cO_K}\cO_{L}).
\]
By \cite[Proposition 2.7.2.9]{Abbes}, we have
\[
M_\chi(\Gamma_1^\Delta(\frn))_{\cO_L}=\varprojlim_n H^0(X_{\univ,n}^\ord|_{\cO_{L,n}},\bar{\omega}_{\univ,n}^{\ord,\chi}|_{\cO_{L,n}})
\]
and thus it is flat over $\cO_L$.
Put 
\[
M_\chi(\Gamma_1^\Delta(\frn))_{L}=M_\chi(\Gamma_1^\Delta(\frn))_{\cO_L}[1/\wp].
\]
We refer to any element of this module as a $\wp$-adic Drinfeld modular form of tame level $\frn$ and weight $\chi$ over $L$. Since the action of $\bF_q^\times$ on $X_1^\Delta(\frn)$ via $c\mapsto \langle c\rangle_{\Delta}$ induces an action on $H^0(X_{\univ,n}^\ord|_{\cO_{L,n}},\bar{\omega}_{\univ,n}^{\ord,\chi}|_{\cO_{L,n}})$, the module $M_\chi(\Gamma_1^\Delta(\frn))_{L}$ is decomposed as
\[
M_\chi(\Gamma_1^\Delta(\frn))_{L}=\bigoplus_{m\in \bZ/(q-1)\bZ} M_{\chi,m}(\Gamma_1(\frn))_{L},
\]
where the space $M_{\chi,m}(\Gamma_1(\frn))_{L}$ of type $m$ forms is the maximal subspace on which $\langle c\rangle_{\Delta}$ acts by $c^{-m}$.

For any $\chi\in \bS$ and any positive integer $n$, we can find an integer $k$ satisfying $\chi\equiv k\bmod (q^d-1)p^{l_p(n)}$. Then we have an isomorphism $\bar{\omega}_{\univ,n}^{\ord,\chi}\simeq(\bar{\omega}_{\univ,n}^{\ord})^{\otimes k}$ compatible with Frobenius structures. Using this identification, we obtain a map of $x$-expansion
\begin{equation}\label{EqnDefVQexp}
H^0(X_{\univ,n}^\ord|_{\cO_{L,n}},\bar{\omega}_{\univ,n}^{\ord,\chi}|_{\cO_{L,n}})\to \cO_{L,n}[[x]],\quad f_n\mapsto (f_n)_\infty(x).
\end{equation}
For any such $k$ and $k'$, the correspondence of Lemma \ref{RHCorr} gives an isomorphism $(\bar{\omega}_{\univ,n}^{\ord})^{\otimes k}\simeq(\bar{\omega}_{\univ,n}^{\ord})^{\otimes k'}$ compatible with Frobenius structures. Since (\ref{EqnExactPhiTnFixed}) implies that such an isomorphism is unique up to the multiplication by an element of $A_n^\times$, by restricting to the $\infty$-cusp and using (\ref{EqnHTTTriv}) we see that it agrees with the multiplication by $g_d^{(k'-k)/(q^d-1)}$.
Since $(g_d)_\infty(x)^{p^{l_p(n)}}\equiv 1\bmod \wp^n$, the map (\ref{EqnDefVQexp}) is independent of the choice of $k$ and induces
\begin{equation}\label{EqnQexpP}
\begin{aligned}
M_{\chi,m}(\Gamma_1(\frn))_{L} &\to \cO_L[[x]][1/\wp]\\
f=(f_n)_n&\mapsto f_\infty(x):=\lim_{n\to \infty} (f_n)_\infty(x)
\end{aligned}
\end{equation}
which is an injection by Krull's intersection theorem. This map identifies our definition of $\wp$-adic Drinfeld modular forms with ``$\wp$-adic Drinfeld modular forms in the sense of Serre'' defined by Goss \cite[Definition 5]{Goss_v} and Vincent \cite[Definition 2.5]{Vincent}, by the following proposition. 

\begin{prop}\label{VsGoss}
The image of the injection (\ref{EqnQexpP}) agrees with the space of power series $F_\infty(x)\in \cO_L[[x]][1/\wp]$ which can be written as the $\wp$-adic limit of $x$-expansions $\{(h_n)_\infty(x)\}_n$, where $h_n$ is an element of $M_{k_n,m}(\Gamma_1(\frn))_{L}$ for some integer $k_n$.
\end{prop}
\begin{proof}
This can be shown as in the proof of \cite[Theorem 4.5.1]{Katz_p}. Indeed, let $f=(f_n)_n$ be an element of $M_{\chi,m}(\Gamma_1(\frn))_{\cO_L}$. For any $n$ we choose an integer $k_n\geq 2$ satisfying $\chi\equiv k_n\bmod (q^d-1)p^{l_p(n)}$. Note that, for any integer $k\geq 2$, the description (\ref{EqnOrdSym}) and Corollary \ref{Ample} give an isomorphism
\begin{align*}
H^0(X_{\univ,n}^\ord|_{\cO_{L,n}},&\bar{\omega}_{\univ,n}^{\ord,k}|_{\cO_{L,n}}) \to \\
&\left(\bigoplus_{j\geq 0} H^0(X_{\univ,n}|_{\cO_{L,n}},(\bar{\omega}_{\univ}^{\Delta}|_{\cO_{L,n}})^{\otimes k+j(q^d-1)})\right)/(g_d-1).
\end{align*}
Therefore, by Proposition \ref{QexpBC} (\ref{QexpBC-BC}), for each $f_n$ we can find an integer $k'_n\geq 2$ and an element $h_n\in M_{k'_n,m}(\Gamma_1(\frn))_{\cO_L}$ satisfying $k'_n\equiv k_n \bmod (q^d-1)p^{l_p(n)}$ and $(f_n)_\infty(x)\equiv (h_n)_\infty(x)\bmod \wp^n$. This yields $\lim_{n\to\infty}(h_n)_\infty(x)=f_\infty(x)$.

Conversely, let $F_\infty(x)=\lim_{n\to\infty} (h_n)_\infty(x)$ be as in the proposition. By Proposition \ref{QexpBC} (\ref{QexpBC-Q}), we may assume $h_n\in M_{k_n,m}(\Gamma_1(\frn))_{\cO_L}$. Multiplying powers of $g_d$ and dividing by $\wp$, we may assume $k_{n+1}>k_n$ and
\[
(h_{n+1})_\infty(x)\equiv (h_n)_\infty(x) \notequiv 0 \bmod \wp^n
\]
for any $n$. By choosing an isomorphism of $\cO_K$-modules $\cO_L\simeq \cO_K^{\oplus [L:K]}$, we identify the $\cO_{X_{\univ}}$-module $(\bar{\omega}_{\univ}^{\Delta})^{\otimes k}\otimes_{\cO_{K}}{\cO_{L}}$ with $((\bar{\omega}_{\univ}^{\Delta})^{\otimes k})^{\oplus [L:K]}$ and $\cO_{L}[[x]]$ with $(\cO_{K}[[x]])^{\oplus [L:K]}$, which are compatible with $x$-expansions. Then Theorem \ref{WeightCongr} implies $k_{n+1}\equiv k_n \bmod (q^d-1)p^{l_p(n)}$ and thus, in the $\wp$-adic weight space $\bS$, the sequence $(k_n)_n$ converges to an element $\chi$ satisfying $\chi \equiv k_n \bmod (q^d-1)p^{l_p(n)}$. Proposition \ref{QexpBC} (\ref{QexpBC-Q}) implies $h_{n+1}\equiv h_n g_d^{(k_{n+1}-k_n)/(q^d-1)}\bmod \wp^n$ and thus $(h_n)_n$ defines an element $f$ of $M_{\chi,m}(\Gamma_1(\frn))_{\cO_L}$ satisfying $f_\infty(x)=F_\infty(x)$. This concludes the proof.
\end{proof}

\begin{thm}\label{GammaWp}
Let $f$ be a Drinfeld modular form of level $\Gamma_1(\frn)\cap \Gamma_0(\wp)$, weight $k$ and type $m$ over $\bC_\infty$ with $x$-expansion coefficients at $\infty$ in the localization $A_{(\wp)}$ of $A$ at $(\wp)$. Then $f$ is a $\wp$-adic Drinfeld modular form of tame level $\frn$, weight $k$ and type $m$. Namely, the $x$-expansion $f_\infty(x)$ at the unramified cusp over the $\infty$-cusp is in the image of the map (\ref{EqnQexpP}) for $\chi=k$.
\end{thm}
\begin{proof}
By Proposition \ref{QexpBC} (\ref{QexpBC-Q}), we may assume $f\in M_k(\Gamma_1^\Delta(\frn,\wp))_{\bF_q(t)}$. By flat base change, we can find an element $g\in M_k(\Gamma_1^\Delta(\frn,\wp))_{A_{(\wp)}}$ such that its image $M_k(\Gamma_1^\Delta(\frn,\wp))_{\bF_q(t)}$ agrees with the element $\wp^l f$ for some non-negative integer $l$.

For any integer $n>0$, put $Y^\wp_{\univ,n}=Y_1^\Delta(\frn,\wp)\times_{A_\frn}\Spec(\cO_{K,n})$. The canonical subgroup $\cC_{1,n}$ over $Y_{\univ,n}^\ord$ gives a section of the natural projection
\[
\xymatrix{
 & Y^\wp_{\univ,n}\ar[d] \\
Y_{\univ,n}^\ord\ar[ur]\ar[r] & Y_{\univ,n}.
}
\]
Pulling back $g$ by this section, we obtain an element $g_n$ of the module $H^0(Y_{\univ,n}^\ord, (\bar{\omega}_{\univ,n}^{\ord})^{\otimes k})$. On each cusp corresponding to $\Xi=(a,b)$, the pull-back of $g_n$ along this cusp agrees with the pull-back of $g$ along the unramified cusp over $\Xi$. Hence $g_n\in  H^0(X_{\univ,n}^\ord, (\bar{\omega}_{\univ,n}^{\ord})^{\otimes k})$. Since
\[
\wp^l f_\infty(x)=g_\infty(x)=\lim_{n\to\infty} (g_n)_\infty(x),
\]
this implies that $f$ is a $\wp$-adic modular form of weight $k$.
\end{proof}







\end{document}